\documentclass[amssymb,amsfonts,refcheck,11pt,verbatim,righttag]{amsart}

\usepackage{amssymb,mathrsfs,hyperref}

\usepackage{graphicx,color,tikz,caption,subcaption}

\numberwithin{equation}{section}
\theoremstyle{plain}

\newtheorem{thm}{Theorem}[section]
\newtheorem{lem}[thm]{Lemma}
\newtheorem{pro}[thm]{Proposition}
\newtheorem{cor}[thm]{Corollary}
\newtheorem{ex}[thm]{Example}

\newtheorem{rem}[thm]{Remark}

\def\R {{\Bbb R}}
\def\N {{\Bbb N}}

\def\I {{\mathcal I}}

\def\W {{\mathcal W}}

\def\E{{\mathbb E}}

\def\x{{\bf x}}

\begin{document}
\baselineskip 16pt
\title{Projections of planar Mandelbrot measures}

\author{Julien Barral}
\address{LAGA (UMR 7539), D\'epartement de Math\'ematiques, Universit\'e Paris 13 (Sorbonne-Paris-Cit\'e), 99 avenue Jean-Baptiste Cl\'ement , 93430  Villetaneuse, France, and DMA (UMR 8553), Ecole Normale Sup\'erieure, 45 rue d'Ulm, 75005 Paris, France}
\email{barral@math.univ-paris13.fr}
\author{De-Jun Feng}
\address{
Department of Mathematics\\
The Chinese University of Hong Kong\\
Shatin,  Hong Kong\\
} \email{djfeng@math.cuhk.edu.hk}


\date{}

\begin{abstract}
Let $\mu$ be a planar Mandelbrot measure and $\pi_*\mu$ its orthogonal projection on one of the main axes.  We study the thermodynamic and geometric properties of $\pi_*\mu$. We first show that  $\pi_*\mu$ is  exactly dimensional, with $\dim(\pi_*\mu)=\min(\dim(\mu),\dim(\nu))$, where~$\nu$ is the Bernoulli product measure obtained as the expectation of $\pi_*\mu$. We also prove that $\pi_*\mu$ is absolutely continuous with respect to $\nu$ if and only if $\dim(\mu)>\dim(\nu)$, and find sufficient conditions for the equivalence of these measures. Our results provides a new proof of Dekking-Grimmett-Falconer formula for the Hausdorff and box dimension of the topological support of $\pi_*\mu$, as well as a new variational interpretation. We obtain the free energy function $\tau_{\pi_*\mu}$ of $\pi_*\mu$ on a wide subinterval $[0,q_c)$ of $\R_+$. For $q\in[0,1]$, it is given by a variational formula which sometimes yields phase transitions of order larger than~1. For $q>1$, it is given by $\min(\tau_\nu,\tau_\mu)$, which can exhibit first order phase transitions. This is in contrast with the analyticity of $\tau_\mu$ over $[0,q_c)$. Also, we prove the validity of the multifractal formalism for $\pi_*\mu$
at each $\alpha\in (\tau_{\pi_*\mu}'(q_c-),\tau_{\pi_*\mu}'(0+)]$.\end{abstract}

\maketitle

\section{Introduction}

Mandelbrot measures are statistically self-similar measures introduced in early seventies by B. Mandelbrot in~\cite{mandelbrot} as a simplified model for energy dissipation in intermittent turbulence. In $\R^2$, such a non-trivial  random measure $\mu$ is built on $[0,1]^2$ and  is characterized by the equality

\begin{equation}\label{self-similarity}
\mu=\sum_{0\le i,j\le m-1} W_{i,j}\,  \mu^{(i,j)}\circ S_{i,j}^{-1},
\end{equation}
 where $m$ is an integer $\ge 2$,  $S_{i,j}$ are similarity maps on $\R^2$ defined by
  $$
  S_{i,j}(x,y)=\left(\frac{x+i}{m},\frac{y+j}{m}\right),
  $$
$W_{i,j}$ are non-negative random variables satisfying
$$
\E\left (\sum_{0\le i,j\le m-1} W_{i,j}\right )=1$$
 and
 $$
D:=-\E\left (\sum_{0\le i,j\le m-1} W_{i,j}\log_m(W_{i,j})\right )>0,
$$
  $\mu^{(i,j)}$ are independent copies of $\mu$,  which are  also independent of the weights~$W_{i,j}$.

The topological support of $\mu$, denoted by $K$,  is a statistically self-similar limit set so that
$$
K=\bigcup_{\substack{0\le i,j\le m-1\\W_{i,j}>0}} S_{i,j}(K_{i,j}),
$$
where  $K_{i,j}$ are independent copies of $K$.

 The fine geometric properties of $\mu$ were initially studied by Mandelbrot himself in \cite{mandelbrot,mandel2}, as well as  by  Kahane and Peyri\`ere in \cite{KP}. It was established that $\mu$ is exactly $D$-dimensional, i.e. the local dimension of $\mu$ equals $D$ on a set of full  $\mu$-measure.  Moreover,  a statistical description of the mass distribution of $\mu$  at small scales was given by Mandelbrot by using  large deviations properties of the branching random walk naturally associated with $\mu$.

On the other hand, the topological and measure theoretic properties  of $K$ have been  studied intensively \cite{mandel3,Pey,CCDu,DekGri,Fal, GMW,DekMee,FalGri,BenNasr,Wa07,DekSi,MorSiSo,RamSi1,RamSi2,FalJin,PerRam}.

Mandelbrot measures, as well as self-similar measures and Gibbs measures, are typical objects illustrating the multifractal formalism,   which emerged  in the middle of the eighties from turbulence theory \cite{FrPa} and hyperbolic dynamical systems \cite{HaJeKaPrSh,Collet},
 in order to describe geometrically at small scales the distribution of a  measure, or the H\"older singularities of a function;
  this formalism can be viewed as  a geometric counterpart of large deviations theory.  For measures, it can be  defined as follows.

If $(X,d)$ is a locally compact metric space and $\mu$ is a positive and finite compactly supported measure, denoting its topological support as $\mathrm{supp}(\mu)$, the $L^q$-spectrum of $\mu$ is a kind of free energy concave function defined by
$$
\tau_\mu:q\in\R\mapsto  \liminf_{r\to 0+}\frac{\log \sup\Big \{\sum_i \mu(B(x_i,r))^q\Big \}}{\log(r)},
$$
where the supremum is taken over all the centered packings of $\mathrm{supp}(\mu)$ by closed balls of radius  $r$.
If $(X,d)$ possesses the Besicovitch property (like Euclidean $\R^d$ or any symbolic space endowed with the standard metric), for $\alpha\in\R$ one always has (see e.g. \cite{BMP,Olsen1995,LN})
$$
\dim_H E(\mu,\alpha)\le \tau_\mu^*(\alpha):=\inf\{\alpha q-\tau_\mu(q):q\in\R\},
$$
where
$$
E(\mu,\alpha)=\left \{x\in\mathrm{supp}(\mu): \lim_{r\to 0+} \frac{\log(\mu(B(x,r)))}{\log(r)}=\alpha\right \},
$$
 here $\dim_H$ stands for the Hausdorff dimension, and  we adopt the convention that $$\dim_H \emptyset=-\infty.$$
 One says that {\it the multifractal formalism holds for $\mu$ at $\alpha$} if $\dim_H E(\mu,\alpha)=\tau_\mu^*(\alpha)$, and one says that {\it it holds for $\mu$}  if this equality holds for all $\alpha$, i.e. the Hausdorff spectrum $\alpha\mapsto \dim_H E(\mu,\alpha)$ and $\tau_\mu$ form a Legendre pair. Furthermore, one says that there is { \it $k$-th order phase transition} at $q$ for $\mu$ if $\tau_\mu$ has a $(k-1)$-th order derivative but no $k$-th order derivative at $q$.

In this paper we will investigate the multifractal structure of  the orthogonal projections of a Mandelbrot measure $\mu$ on the horizontal and vertical axes, and its relation with that of $\mu$. For this purpose,  we recall that  under mild assumptions, defining for $q\in \R$
$$
T(q)=-\log_m\sum_{0\le i,j\le m-1}\mathbb E(\mathbf{1}_{\{W_{i,j}>0\}}W_{i,j}^q),
$$
one has $\tau_\mu=T$, hence $\tau_\mu$ is analytic, on  the interval $\{q\in\R: T^*(T'(q))\ge 0\}$ (see Section~\ref{MAMM}).

In our study of projections of $\mu$, we will  consider the range $q\ge 0$ for the $L^q$-spectrum. This restriction is often met in the geometric study of measures obtained via projection schemes, like self-similar measures obtained as projections of Bernoulli products on self-similar sets satisfying the weak separation condition (see  e.g. \cite{FengLau} and  the references therein) or  self-affine measures obtained as projections of Bernoulli products on almost all the attractors associated with a given finite collection of contractive linear maps \cite{Fal99,BF}.

The case of {orthogonal projections ${\pi_\ell}_*\mu$ of $\mu$} on almost every line {$\ell$} passing through the origin is essentially similar to that of Gibbs measures treated in \cite{BaBh}. In this case, due to Mastrand projection theorem, one is naturally led to consider the case where $D\le 1$, for otherwise the projection of $\mu$ is absolutely continuous with respect to Lebesgue measure and it is hard to say more  about the multifractality  in general. Then, since $D\le 1$, the dimension of the projection is still $D$, and there are two possible behaviors in terms of the $L^q$-spectrum. If $\dim_H K\le 1$ as well, then {$\tau_{{\pi_\ell}_*\mu}=\tau_\mu$} on the interval $[0,q_2]$, where $q_2$ is defined by $\tau_\mu(q_2)=2$ (notice that  $q_2\ge 3$ due to the concavity of $\tau_\mu$ and the facts that $\tau_\mu(1)=0$ and $\tau_\mu(0)\ge - \dim_H\mathrm{supp}(\mu)=-1$ in this case). If $\dim_H K> 1$, there is a second order phase transition at the unique $\widetilde q\in[0,1]$ at which $\tau_{\mu}^*(\tau_\mu'(\widetilde q))=1$; more precisely,  the $L^q$-spectrum {$\tau_{{\pi_\ell}_*\mu}$} is analytic over $(0,\widetilde q)$ and $(\widetilde q,q_2)$ but not twice differentiable at $\widetilde q$; specifically, it is linear on $[0,\widetilde q]$ and equals $\tau_\mu$ on $[\widetilde q,q_2]$. {Also, the multifractal formalism is valid at any $\alpha\in \tau_{{\pi_\ell}_*\mu}'([0,q_2])$. It is worth mentioning that the preservation of the $L^q$-spectrum over $[1,q_2]$ is a fact valid for any measure (see \cite{HK,BB}).}

The situation is significantly different with the main axes. At first, it is worth noticing that for a Gibbs measure associated with a H\"older potential on the unit square, e.g. for  the self-similar measures  obtained when the weights $W_{i,j}$ are constant, its projection on any of the main directions is still a Gibbs measure of this kind \cite{ChU}, so no special new phenomenon appears related to its multifractal nature.  Things turn out to be more interesting with (random) Mandelbrot measures.

Denote by $\pi$ the orthogonal projection on one of the main axes. It is known (Dekking and Grimmett \cite{DekGri}, Falconer \cite{Fal}), that $\dim_H \pi (K)$ in general differs from the typical value obtained by  Mastrand's projection theorem when one projects on almost every line. Instead of being equal to $\min (\dim_HK,1)$, $\dim_H \pi (K)$ is given by a variational formula: denoting $N_i=\#\{0\le j\le m-1: W_{i,j}>0\}$, one has
\begin{equation}\label{dimproj}
\dim_H \pi (K)=\inf_{0\le h\le 1} \log_m\sum_{i=0}^{m-1}\E(N_i)^h.
\end{equation}
Moreover, this dimension equals the box counting dimension of $\pi (K)$.  It turns out that understanding the geometric structure of the projection $\pi_*\mu$ of the Mandelbrot measure $\mu$ heavily relies on its expectation, which is the Bernoulli product measure~$\nu$ associated with the probability vector $\left (p_i=\sum_{j=0}^{m-1}\E(W_{i,j})\right )_{0\le i\le m-1}$, for which it is known that
$$
\tau_\nu(q)=-\log_m\sum_{\substack{0\le i\le m-1\\ p_i>0}} p_i^q.
$$

In this paper, we show (Theorems~\ref{ABS} and \ref{DIM}) that when $\mu\neq 0$, $\pi_*\mu$ is exactly dimensional with $\dim (\pi_*\mu)=\dim(\mu)=D$ if and only if $\dim (\mu)\le \dim (\nu)$, in which case $\pi_*\mu$ is singular with respect to $\nu$, while if $\dim (\mu)>\dim (\nu)$ then $\pi_*\mu$ is absolutely continuous with respect to~$\nu$. We also find sufficient conditions for the two measures to be equivalent. Exact dimensionality and ``dimension conservation properties'' of projections of Mandelbrot measures on {\it all} the lines  have already been established in \cite{FalJin}; however, the result of \cite{FalJin} is not quantitative, whilst for the main axes we provide the precise values for the dimensions, which differ from those given by Mastrand's theorem for almost every line when $\nu$ is not the Lebesgue measure. Also, as a consequence of Theorem~\ref{DIM} we get  a new variational interpretation of Dekking-Grimmett-Falconer formula for $\dim_H\pi(K)$ (Corollary~\ref{DGF}).

Regarding the multifractal analysis (Theorem~\ref{MA}), for $q\ge 1$ we prove that  $$\tau_{\pi_*\mu}(q)=\min (\tau_\mu(q),\tau_\nu(q))$$ on a  non-trivial interval $[1,\widetilde q_c)$. This fact is a source of first order phase transitions when  the graphs of $\tau_\nu$ and $T$ cross each other transversally. For $0< q\le 1$, we prove that $\tau_{\pi_*\nu}$ is given by the following  variational formula:
$$
\tau_{\pi_*\mu}(q)=\displaystyle -\inf\left\{ \log_m\sum_{i=0}^{m-1}\left (\sum_{j=0}^{m-1} \E(W_{i,j}^s)\right )^{q/s}:q\le s\le 1\right \},
$$
which converges to the value of $\dim_H \pi (K)$ given by \eqref{dimproj} as $q$ tends to $0$. The function $\tau_{\pi_*\mu}$ is differentiable over $[0,1]$.  It coincides with $\tau_\mu(q)$ when the infimum is attained at $s=q$ and $\tau_\nu(q)$ when it is attained at $s=1$. Otherwise, the infimum  is attained at a unique $s(q)\in (q,1)$, and this property holds on a neighborhood of $q$ over which by definition of $s(q)$ one has $\tau_{\pi_*\mu}(q)> \max (\tau_\mu(q),\tau_\nu(q))$. These possible changes of analytic expressions  lead to  phase transitions  of orders greater than or equal to 2. In particular,  each transversal crossing of  the graphs of $\tau_\nu$ and $T$ gives rise to such a phase transition.

We also verify  the validity of the multifractal formalism over $(\tau_{\pi_*\mu}'(q_c-),\tau_{\pi*\mu}'(0+)]$. When applied to the so-called branching  measure on $K$, our result yields a partial multifractal classification of the asymptotic number of squares of a given generation necessary to cover the fibers $\pi^{-1}(\{x\})$, $x\in \pi(K)$ (see Corollary~\ref{covnumb}).

Let us finally  mention  that Mandelbrot martingales in various Bernoulli random environments play an important role in our study.

The paper is organized as follows. We will work with Mandelbrot measures on the symbolic space $\{0,\ldots,m-1\}^{ \N}\times \{0,\ldots,m-1\}^{ \N}$, for this offers a simpler framework to expose ideas and techniques. We explain in Appendix~\ref{D} the general simple principle which makes it possible to transfer the results to the Euclidean case. In Section~\ref{GenMA}  we recall basic facts from multifractal formalism, as well as the  formal definition of Mandelbrot measures and a precise known result for their multifractal analysis. In Section~3 we present  in complete rigor our main results in this symbolic context, while Section~4 contains comments and examples related to phase transitions. Section~\ref{pfDIM} provides the proof of our results related to the dimension of the projected measures, as well as the new variational interpretation of the Hausdorff dimension of their topological support. Sections~6 to~8 provide the proof of Theorem~\ref{MA} about the multifractal analysis of the projection. Specifically, Section 6 deals with the differentiability property of the function identified to be the $L^q$-spectrum of $\pi_*\mu$, Section 7 exhibits the sharp lower bound for the $L^q$-spectrum, and Section 8 deals both with the sharp upper bound for the  $L^q$-spectrum and the Hausdorff spectrum. Sections~5,~7 and~8 use moments estimates developed in Section~\ref{MEST} for quantities related to Mandelbrot martingales in Bernoulli  environments, as well as other basic results gathered in the Appendix.

\section{Preliminaries on multifractal formalism and Mandelbrot measures}\label{GenMA}

Throughout this paper, we use $\N$ to denote the set of natural numbers, i.e. $\N=\{1,2,\ldots\}$.  Let us first restate the multifractal formalism in this context.

\subsection{Multifractal formalism on symbolic spaces}

Let $m\ge 2$ be an integer. For  $n\ge 0$ let $\Sigma_n=\{0,\ldots,m-1\}^n$.  By convention, $\Sigma_0$ consists of the empty word $\epsilon$. Then define $\Sigma_*=\bigcup_{n\ge 0}\Sigma_n$, $(\Sigma\times\Sigma)_*=\bigcup_{n\ge 0}(\Sigma_n\times\Sigma_n)$, and $\Sigma=\{0,\ldots,m-1\}^\N$.  The sets $\Sigma_*$ and $(\Sigma\times\Sigma)_*$ act in the standard way by concatenation on $\Sigma_*\cup \Sigma $ and $(\Sigma\times\Sigma)_*\cup(\Sigma\times\Sigma)$ respectively. We denote by  $\sigma$ the standard left shift operation on $\Sigma_*\cup(\Sigma\times\Sigma)$. The length of a word $w\in\Sigma_*$, i.e. its number of letters, is denoted as $|w|$.

For $x=x_1\cdots x_p\cdots \in\Sigma$,   set $x_{|n}=x_1\cdots x_n$ if $n\ge 1$ and $\epsilon$ if $n=0$. For $u\in\Sigma_*$, set $[u]=\{x\in\Sigma: x_{||u|}=u\}$.

The set $\Sigma$ is endowed with the standard metric distance $$d(x,x')=m^{-\sup\left\{n: \, x_{|n}=x'_{|n}\right\}},$$ and $\Sigma\times \Sigma$ is endowed with the distance $d((x,y),(x',y'))=\max (d(x,x'),d(y,y'))$.

Given a positive and finite Borel measure $\rho$ on $\Sigma$ or $\Sigma\times\Sigma$, its topological support, i.e. the smallest closed set carrying the whole mass of $\mu$ is denoted as $\mathrm{supp}(\rho)$, and its lower and upper local dimensions at $x \in\mathrm{supp}(\rho)$ are defined as
$$
\underline \dim_{\rm loc}(\rho, x)=\displaystyle \liminf_{n\to \infty} \frac{\log (\rho([ x_{|n}]))}{(-n\log (m))} \quad\text{and}\quad \overline \dim_{\rm loc}(\rho,x)=\displaystyle \limsup_{n\to \infty} \frac{\log (\rho([x_{|n}]))}{(-n\log (m))}
$$
respectively.  Let
\begin{align*}
\underline \dim_H(\rho)&=\inf\{\dim_H E: \, \rho(E)>0 \mbox{ and $E$ is a Borel set}\}  \; \text{and} \\
\quad  \overline \dim_P(\rho)&=\inf\{\dim_P E:  \rho(E)=\|\rho\|  \mbox{ and $E$ is a Borel set}\},
\end{align*}
 where $\dim_PE$ stands for the packing dimension of $E$ (see e.g. \cite{Mattila}) and $\|\rho\|$ stands for the total mass of~$\rho$ .

 It is well known that  (see e.g. \cite{Cut0,Cut})
\begin{align*}
\underline \dim_H(\rho)&=\sup\{ s:\;  \underline \dim_{\rm loc}(\rho,x)\geq s \mbox{ for $\rho$-almost every $x$}\} \; \mbox{and}\\
\overline \dim_P(\rho)&=\inf\{s:\;    \overline \dim_{\rm loc} (\rho,x)\leq s \mbox{ for $\rho$--almost every $x$}\};
\end{align*}
when these two dimensions coincide, one says that  $\rho$ is exactly  dimensional  and  writes~$\dim(\rho)$ for the common value.

The $L^q$-spectrum of $\rho$ is the mapping  $\tau_\rho:\; \R\to \R\cup\{-\infty\}$ given by
$$
\tau_\rho(q)=\liminf_{n\to\infty} \frac{-1}{n}\log_m\sum_{w\in S_n}\mathbf{1}_{\{\rho([w])>0\}}\rho([w])^q\quad (q\in\R),
$$
where $S_n$ stands for $\Sigma_n$ or $\Sigma_n\times \Sigma_n$. It is well known that (cf. \cite{Ngai})
$$
\tau_\rho'(1+)\le \underline \dim_H(\rho)\le \overline \dim_P(\rho)\le \tau_\rho'(1-).
$$

For all $\alpha\in\R $, set

\begin{align*}
\underline E(\rho,\alpha)&=\{x\in\mathrm{supp}(\rho):  \underline \dim_{\rm{loc}}(\rho,x)=\alpha\},\\
\overline E(\rho,\alpha)&=\{x\in\mathrm{supp}(\mu):  \overline \dim_{\rm{loc}}(\rho,x)=\alpha\}
\end{align*}

and
$$
E(\rho,\alpha)=\underline E(\rho,\alpha)\cap \overline E(\rho,\alpha).
$$
Then one also always has (see e.g. \cite{LN, Olsen1995}) that
$$
\dim_H E(\rho,\alpha)\le \max (\dim_H \underline E(\rho,\alpha), \dim_H \overline E(\rho,\alpha))\le \tau_\rho^*(\alpha),
$$
where the Legendre transform of $f:\R\to\R\cup\{-\infty\}$ is defined as $$f^*:\alpha\in \R\mapsto \inf_{q\in\R}(\alpha q-f(q)),$$ and  a negative dimension means that the set is empty. One says that the multifractal formalism holds at $\alpha$ if $$\dim_H E(\rho,\alpha)= \tau_\rho^*(\alpha).$$
Also, if $\alpha\le \tau_\rho'(0+)$, one has (see e.g. \cite{LN, Olsen1995})
\begin{equation}\label{MF}
\dim_H \underline E^{\le }(\rho,\alpha)\le \tau_\rho^*(\alpha),
\end{equation}
where
$$
\underline E^{\le }(\rho,\alpha)=\{x\in\mathrm{supp}(\rho):  \underline \dim_{\rm{loc}}(\rho,x)\le \alpha\}.
$$

\subsection{Multifractal analysis of the Mandelbrot measures on  $\Sigma\times\Sigma$}\label{MAMM}

 Now let us formally define the  Mandelbrot measures on $\Sigma\times\Sigma$. We consider a non-negative random vector $$W=(W_{i,j})_{(i,j)\in\Sigma_1\times\Sigma_1}$$ whose entries are integrable. For $q\in \R$ we define
\begin{equation}\label{defpsi}
T(q)=T_W(q)=-\log_m\sum_{(i,j)\in\Sigma_1\times\Sigma_1}\mathbb E(\mathbf{1}_{\{W_{i,j}>0\}}W_{i,j}^q).
\end{equation}

We denote $N=\sum_{(i,j)\in\Sigma_1\times\Sigma_1}\mathbf{1}_{\{W_{i,j}>0\}}$, and assume that $\mathbb P(N\in\{0,1\})<1$.

To build a Mandelbrot  measure on $\Sigma\times \Sigma$ we assume that $T(1)=0$ and consider  a sequence $(W(u,v))_{(u,v)\in \bigcup_{n\ge 0}\Sigma_n\times\Sigma_n}$ of independent copies of $W$, defined on a probability space $(\Omega,\mathcal A,\mathbb{P})$.

For each $n\ge 1$ let $\mu_n=\mu_{W,n}$ be the measure on $\Sigma\times \Sigma$ whose density with respect to the measure of maximal entropy is constant over each cylinder $[u,v]:=[u]\times[v]$ of generation $n$ and given by $m^{2n}Q(u,v)$, where
$$
Q(u,v)=\prod_{j=1}^{n}W_{u_j,v_j}(u_{|j-1},v_{|j-1}).
$$
Denote the total mass of $\mu_n$ as $Y_n$, i.e. $$Y_n=\sum_{|u|=|v|=n} Q(u,v).$$ By construction the sequence $(Y_n)_{n\ge 1}$ is a non-negative martingale of expectation $1$ with respect to the filtration $(\sigma(W(u,v): |u|=|v|\le n-1))_{n\ge 1}$, thus it converges to a limit, which we denote by $Y$.

Let $T_n=\{(u,v)\in\Sigma_n\times\Sigma_n: Q(u,v)>0\}$. The sequence $(T_n)_{n\ge 1}$ represents the generations of a Galton-Watson process with offspring distribution given by that of $N$. We have $$K_n:={\rm supp}(\mu_n)=\bigcup_{(u,v)\in T_n}[u]\times[v].$$

For $n\ge k\ge 1$ and $(u,v)\in \Sigma_k\times\Sigma_k$,  the statistical self-similarity of the construction yields $(\mu_n([u]\times[v])= Q(u,v) Y_{n-k}(u,v)$, with $(Y_{n-k}(u,v))_{(u,v)\in\Sigma_k\times\Sigma_k}$ a family of independent copies of $Y_{n-k}$, also independent of $\sigma(W(u,v): |u|=|v|\le k-1)$.

Consequently, with probability 1, there exists a family $(Y(u,v))_{(u,v)\in \Sigma_k\times\Sigma_k,k\ge 1})$ of copies of $Y$ such that for each $k\ge 1$ and  $(u,v)\in \Sigma_k\times\Sigma_k$,
\begin{equation}\label{Yuv}
\lim_{n\to\infty} \mu_n([u]\times[v])=Q(u,v) Y(u,v).
\end{equation}
Moreover, the random variables  $Y(u,v)$, $(u,v)\in \Sigma_k\times\Sigma_k$, are independent,  and generate a $\sigma$-field independent of   $\sigma(W(u,v): |u|=|v|\le k-1)$. By construction, this means that $\mu_n$ weakly converges to a measure $\mu$ defined by
$$
\mu([u]\times[v])= Q(u,v) Y(u,v).
$$
Moreover, $\mu$ is positive  (i.e. $Y>0$) with positive probability if and only if $T'(1-)>0$; and this is also equivalent to the uniform integrability of $(Y_n)_{n\ge 1}$, that is $\E(Y)=1$ (\cite{KP,DL}). From now on we assume that this condition (i.e., $T'(1-)>0$) holds; in this case, it is known (cf. \cite{KP,K87}) that  the measure $\mu$, if non-degenerate,  is exactly dimensional  and
$$
\dim(\mu)=T'(1-)\text{ almost surely on }\{\mu\neq 0\}.
$$
Also, the events $\{\mu\neq 0\}$ and $\{K:=\bigcap_{n\ge 1} K_n\neq\emptyset\}$ coincide up to a set of probability~0 over which we have $K=\mathrm{supp}(\mu)$ (see Proposition~\ref{Kmu} for a proof). In addition, the inequality $T'(1-)>0$ and the concavity of $T$ imply that $T(0)=-\log_m(\E(N))<0$, i.e. $\E(N)>1$.

We have the following result regarding the multifractal analysis  of $\mu$ (see also \cite{HoWa,Falconer1994,Olsen1994,Mol,Ba00} for slightly less sharp versions).

\begin{thm}[\cite{AB}]\label{AB} Suppose that $T$ is finite on a neighborhood of $0$ and that conditionally on $N\neq 0$ one has $N\ge 2$. Define $f(\alpha)=T^*(\alpha)$ if $T^*(\alpha)\ge 0$ and $f(\alpha)=-\infty$ otherwise. With probability 1,  conditionally on $\{\mu\neq 0\}$, $\tau_\mu=f^*$ and the multifractal formalism holds at all $\alpha$ in the domain of $\tau_\mu^*=f$. In particular, $\tau_\mu(q)=T(q)$ at each $q\in\R$ such that $T^*(T'(q))\ge 0$.
\end{thm}

Since  we mainly want to focus on new phenomena associated with $\pi_*\mu$, we will avoid to deal with too many technicalities and discard the case  when  
 $$\sup\{q\ge 1: T(q)>-\infty\}= \sup\{q\ge 1: T^*(T'(q))>0\}.$$

  Thus, when we  study the validity of the multifractal formalism for $\pi_*\mu$, our assumptions will be:
\begin{equation}\label{assumMA}
\begin{split}
&\bullet \ \mathbb P(N\in\{0,1\})<1,\ T'(1-)>0;\\
&\bullet \ T \text{ is finite on a neighborhood of $0$};\\
 &\bullet \   \text{either $\exists$   $q_{c}>1$ such that $T^*(T'(q_{c}^-))=0$}\\
& \quad \; \text{or $T^*(T'(q))>0$ for all $q\ge 0$, in which case we set $q_{c}=\infty$}.
\end{split}
\end{equation}

We drop the assumption that $N\ge 2$ when $N\neq 0$ because this  does not affect the validity of  Theorem~\ref{AB} for the local dimensions  $\alpha$  associated with non-negative $q$ by Legendre duality,  and our study of $\pi_*\mu$ we will only focus on the case $q\ge 0$.

\section{Main results for projections of Mandelbrot measures}

 Throughout this section we assume that $\mathbb P(N\in\{0,1\})<1$ and $T'(1-)>0$. We are interested in the geometric properties of the measure $\pi_*\mu$, where $\pi$ stands for the canonical projection onto the first factor of $\Sigma\times \Sigma$. We are also concerned with   the disintegrations of $\mu$ associated with the projection  $\pi$.

\medskip

For $0\le i,j\le m-1$ set
$$
p_i=\sum_{j=0}^{m-1}\E(W_{i,j}) \text{ and }
V_{i,j}=\begin{cases} W_{i,j}/p_i&\text{  if $p_i>0$,}\\
{1}/{m}&\text{ otherwise.}
\end{cases}
$$
(In fact those $i$ for which $p_i=0$ will play no role  in our study.)   
Then  write  $V_i:=(V_{i,j})_{j\in \Sigma_1}$ and  define 
\begin{equation}\label{defpsii}
T_i(q)=T_{V_i}(q)=-\log_m\sum_{j\in\Sigma_1}\mathbb E(\mathbf{1}_{\{V_{i,j}>0\}}V_{i,j}^q),\qquad q\in \R.
\end{equation}
 Notice that $T_i(1)=0$ for all $0\le i\le m-1$.

Let $\nu$ stand for the Bernoulli product measure on $\Sigma$ associated with the probability vector $(p_0,\ldots,p_{m-1})$, that is $$\nu([x_1\ldots x_n])=p_{x_1}\ldots p_{x_n}$$ for $n\geq 1$ and $x_1,\ldots, x_n\in \{0,1,\ldots, m-1\}$.

By construction  we have
\begin{equation}\label{psi}
m^{-T(q)}=\sum_{i,j}\mathbb{E}(\mathbf{1}_{\{W_{i,j}>0\}}W_{i,j}^q)=\sum_{i,j}\mathbf{1}_{\{p_i>0\}}p_i^q\mathbb{E}(V_{i,j}^q)=\sum_{i}\mathbf{1}_{\{p_i>0\}}p_i^qm^{-T_i(q)}.
\end{equation}
Consequently,
\begin{equation}\label{dimrel}
T'(1-)=\sum_{i}p_i(T'_i(1-)-\log_m(p_i))=\Big (\sum_i p_iT'_i(1-)\Big )+\dim(\nu),
\end{equation}
where we recall that
$$
\dim(\nu)=-\sum_{i=0}^{m-1} p_i\log(p_i)/\log(m).
$$

%
Notice that $\nu=\mathbb{E} (\pi_*\mu)$, and recall that a direct calculation yields
$$
\tau_\nu(q)=-\log_m\sum_{i=0}^{m-1}p_i^q \quad (q\in\R).
$$
For $q\in\R$, we denote by $\nu_q$ the Bernoulli product measure on $\Sigma$ associated with the probability vector $\big(p_0^qm^{\tau_\nu(q)},\ldots,p_{m-1}^qm^{\tau_\nu(q)}\big )$.

Below we discard two trivial situations.

We first discard the case when $p_i=1$ for some $0\le i\le m-1$, which means that the measure $\mu$ is supported on a deterministic vertical line hence is a Mandelbrot measure on a line, for which the multifractal nature is analogue to that of a 1-dimensional Mandelbrot measures. For $0\le i\le m-1$, we set
\begin{equation}\label{Ni}
N_i=\#\{0\le j\le m-1:W_{i,j}>0\}.
\end{equation}
We also discard the case when $N_i=1$ almost surely for all $0\le i\le m-1$, which implies that $\pi_*\mu$ is a Mandelbrot measure on a line as well.



\subsection{Absolute continuity and dimension}
This section gathers our results on the absolute continuity/singularity of $\pi_*\mu$ with respect to $\nu=\mathbb{E}(\pi_*\mu)$, and on the dimension of  $\pi_*\mu$ and its associated conditional measures in the natural disintegration of $\mu$ along $\pi_*\mu$-almost every fiber $\{x\}\times\Sigma$. The result on $\dim(\pi_*\mu)$ also yields a new variational principle for $\dim \pi(K)$.

\begin{thm}\label{ABS}
With probability 1, conditionally on $\{\mu\neq 0\}$:
\begin{enumerate}
\item If $\dim (\mu)>\dim(\nu)$, then
\begin{itemize}
\item[(i)] $\pi_*\mu$ is absolutely continuous with respect to $\nu$.
\item[(ii)]
Suppose that $T$ is finite in a neighborhood of $1$. Then the  density of $\pi_*\nu$  with respect to $\nu$ is in $L^s$ for all $s$ in the following non-empty set $$\left\{s\in (1,2]: \; T(s)>0 \text{ and }\sum_{i=0}^{m-1}p_im^{-T_i(s)} <1\right\}.$$

\item[(iii)] Suppose, moreover, that  there exists $c\in (0,1)$ such that both $$\Bbb P\left(\sup_{0\leq j\leq m-1} V_{i,j}>c\right)=1 \mbox{ and }\; \E(\#\{j: V_{i,j}>c\})>1$$ hold for all $i$ such that $p_i>0$. Then $\nu$ is absolutely continuous with respect to $\pi_*\mu$, and its density with respect to $\pi_*\mu$ is in $L^s$ for some $s>1$.
 \end{itemize}
\item If $\dim (\mu)\le \dim(\nu)$, then $\pi_*\mu$ and $\nu$ are mutually singular.

\end{enumerate}
\end{thm}

\begin{thm}\label{DIM}With probability 1, conditionally on $\{\mu\neq 0\}$:
\begin{enumerate}
\item If $\dim (\mu)>\dim(\nu)$ then $\pi_*\mu$ is exactly dimensional with dimension $\dim(\nu)$; while  if $\dim (\mu)\le \dim(\nu)$ and $T$ is finite in a neighborhood of $1$, then $\pi_*\mu$ is exactly dimensional with dimension $\dim(\mu)$.

\item 

For $\pi_*\mu$-almost every $x$, the conditional measure $\mu^x$ is exactly dimensional with dimension $\dim (\mu)-\dim(\pi_*\mu)=\dim(\mu)-\dim(\nu)=\sum_{i=0}^{m-1}p_iT'_i(1)$ if $\dim(\mu)>\dim(\nu)$, and dimension 0 if $\dim (\mu)\le \dim(\nu)$ and $T$ is finite in a neighborhood of~$1$.

\end{enumerate}
\end{thm}

\begin{rem}{\rm Recall that in \cite{FalJin} Falconer and Jin have  already proven that with probability~1, conditionally on $\mu\neq 0$, for $\pi_*\mu$-almost every $x$ one has $\dim (\mu^x)=\dim (\mu)-\dim(\pi_*\mu)$, without specifying the value of $\dim(\pi_*\mu)$, hence of $\dim (\mu^x)$. Here we give an alternative proof of the exact dimensionality of  $\pi_*\mu$ and find its dimension, and in the case when $\dim (\mu)>\dim(\nu)$, we also give an alternative proof of the exact dimensions of $\mu^x$ and the dimension preservation $\dim(\mu)=\dim(\pi_*\mu)+\dim(\mu^x)$. When  $\dim (\mu)\le \dim(\nu)$, the exact dimensionality of $\mu^x$ follows from \cite{FalJin}, but our result on $\dim (\pi_*\mu)$ yields that the dimension of $\mu^x$ is 0. Without  using the work in \cite{FalJin}, our result on  $\dim (\pi_*\mu)$ and simple arguments  would give $\underline{\dim}_H(\mu^x)=0$.
}\end{rem}

The previous statement makes it possible to derive the   dimension  formula of  $\pi(K)$ by using an adapted Mandelbrot measure, whilst in \cite{Fal} Falconer builds statistically self-similar subsets of $\pi(K)$ of Hausdorff dimension smaller than but arbitrarily close to the value given by \eqref{dimproj}.   The new point is the variational principle invoking Mandelbrot measures in \eqref{vp} and the related uniqueness property.

\begin{cor}[Dekking-Grimmett-Falconer formula revisited]\label{DGF} Let
\begin{equation}\label{varphi}
\varphi:h\ge 0\mapsto  \log\left (\sum_{i=0}^{m-1}\E(N_i)^h\right )/\log(m).
\end{equation}
With probability 1, conditionally on $K\neq\emptyset$, we have
\begin{equation}
\begin{split}
\dim_H \pi(K)&=\dim_{B}(\pi(K))\\
&=\inf_{0\le h\le 1}\varphi(h)\\
\label{vp}&=\max\{\dim (\pi_*\mu'): \mu' \ \text{{\rm is a Mandelbrot measure supported on $K$}}\}.
\end{split}
\end{equation}
Moreover, the maximum in \eqref{vp} is attained at a unique point if and only if $\varphi'(0)\le 0$, i.e. $\sum_{i=0}^{m-1}\log(\E(N_i))\le 0$.
\end{cor}
\begin{rem}{\rm One has $\dim_H \pi(K)=\dim_H K$ if and only if $\sum_{i=0}^{m-1} \E(N_i)\log \E(N_i)\le 0$, i.e.  in \eqref{vp} the infimum is attained at $h=1$.}

\end{rem}
%

%
%

\subsection{Multifractal formalism for $\pi_*\mu$}
 Assume \eqref{assumMA} and define
 $$
 \widetilde q_c=
 \begin{cases}
q_{c}&\text{if $q_c<\infty$ and }\tau_{\nu}(q_{c})\ge T(q_{c})\\
\inf\{q> q_c: \tau_\nu(q)\ge T(q)\}&\text{ otherwise}
\end{cases},
$$
with the convention that $\inf \emptyset=q_c$.
Let
$$
\tau: q\mapsto \begin{cases}
\displaystyle -\inf\left \{\log_m\sum_{i=0}^{m-1}\E(N_i)^h: 0\le h\le 1\right \}&\text{ if }q=0,\\
\displaystyle -\inf\left\{ \log_m\sum_{i=0}^{m-1} p_i^q m^{-qT_i(s)/s}:q\le s\le 1\right \}&\text{ if }0<q\le 1,\\
\min(\tau_\nu(q),T(q))&\text{ if }1<q< \widetilde q_c\text{ or }q=\widetilde q_c <\infty.
\end{cases}
$$

\begin{thm}\label{MA}
The function $\tau$ is differentiable everywhere except at the possible points in $(1,\widetilde q_c)$ at which the graphs of $T$ and $\tau_\nu$ cross each other transversally.
Moreover,
\begin{enumerate}
\item with probability 1, conditionally on $\{\mu\neq 0\}$, for all   $ q\in [0,\widetilde q_c)$  we have
\begin{equation}\label{convtaun}
\tau(q)=\lim_{n\to\infty}\frac{-1}{n}\log_m\sum_{|u|=n}\mathbf{1}_{\{\pi_*\mu([u])>0\}}\pi_*\mu([u])^q.
\end{equation}
In particular $\tau_{\pi_*\mu}(q)=\tau(q)$. Also, if $\widetilde q_c=q_c<\infty$, we have $\tau_{\pi_*\mu}(q)=qT'(q_c-)$
 for $q>q_c$.
\item If $\alpha\in(\tau'(\widetilde q_c-),\tau'(0+)]$, with probability 1, conditionally on $\{\mu\neq 0\}$, the multifractal formalism holds at $\alpha$.


\end{enumerate}
\end{thm}

\begin{rem}
 Notice that when $q_c<\infty$, the equality $\tau_\nu(q_c)=T(q_c)$ cannot hold if $\tau_\nu'(q_c)\le T'(q_c-)$, for this would imply that $\tau_\nu^*(\tau_\nu'(q_c))\le T^*(T'(q_c-))=0$, while $\tau_\nu^*\circ \tau_\nu'$ is always positive over $\R$.
\end{rem}

\begin{rem}
\label{rem-2.9}
The same conclusion as in \eqref{convtaun} and Theorem~\ref{MA}(2) holds  if we replace $\pi_*\mu([x_{|n}])$ by $\pi_*\mu_n([x_{|n}])$ in the definition of the level sets $E(\pi_*\mu,\alpha)$.
\end{rem}

\begin{rem}\label{specialtau}
(1) One has  $\tau=T=\tau_\nu$ over $[0,q_c)$ if and only if for all $0\le i\le m-1$ such that $p_i>0$ one has $\mathbb{E}(N_i)=1$ and $W_{i,j}\in \{0,p_i\}$ almost surely  for all $0\le j\le m-1$, i.e. $T_i$ is identically equal to 0 (see Section~\ref{difftau}).

\medskip

(2) On the other hand, a sufficient condition to have $\tau=\tau_\nu$ over $\R_+$ and $\tau_\nu>T$ over $(0,1)$ and  $\tau_\nu<T$ over $(1,\infty)$ is   $(\mathcal P)$: there exists a  partition $\{I_1,\ldots, I_L\}$ of $\{0\le i\le m-1:p_i>0\}$ such that $(i)$ $p_i$ does not depend on $i\in I_k$ and $\prod_{i\in I_k}\mathbb E(N_i)=1$ for each $1\le k\le L$; $(ii)$ there exists $1\le k\le L$ such that $\# I_k\ge 2$ and $\E(N_i)\neq 1$ for at least two values of $i\in I_k$; $(iii)$ for all $0\le i\le m-1$ such that $p_i>0$, one has $W_{i,j}\in \{0,p_i/\mathbb{E}(N_i)\}$ almost surely  for all $0\le j\le m-1$. See the proof of Lemma~\ref{lem-4.8}.

\end{rem}

\begin{rem}\label{intersections}
In all the examples we have examined numerically and for which we do not have $\tau_\nu=T$, the functions $\tau_\nu$ and~$T$ coincide at  three points at most. We do not know whether this is a general fact.
\end{rem}

\begin{rem} We think (and know that it is true on some intervals) that the validity of the multifractal formalism for $\pi_*\mu$ holds almost surely for all $\alpha\in (\tau'(\widetilde q_c-)),\tau'(0+)]$. However, we dediced to limit the technicalities as most as possible, the most important facts being the new behaviors associated with the projection. In particular, the proof of the validity of the multifractal formalism will show that the possible phase transitions separate the domain of possible exponents $\alpha$  into intervals over which the computation of the Hausdorff dimension of  the sets $E(\mu,\alpha)$ uses different arguments, this being in contrast with what happens for the measure $\mu$ itself (see \cite{AB}). 
\end{rem}

\begin{rem} [Similar result for critical Mandelbrot measures] When $T'(1-)=0$, under mild assumptions there exists a substitute to the degenerate Mandelbrot measure $\mu$, namely a critical Mandelbrot measures $\widetilde\mu$, which satisfies the same statistical self-similarity \eqref{self-similarity} with the set $K$ as its support, but $\mathbb E(\|\mu\|)=\infty$; the multifractal analysis of this measure is considered in \cite{Ba00}. Defining $q_c$ like when   $T'(1-)>0$, we have $q_c=1$. Furthermore, defining  $\widetilde q_c=1$ and $\nu$ as for $\mu$,  the conclusions of Theorem~\ref{MA} holds for $\pi_*\widetilde\mu$.
\end{rem}

\section{Phases transitions. Remarks and examples}
This section gathers a series of remarks and examples related to phase transitions associated with $\pi_*\mu$.
\begin{rem}
{\rm Let $S$ denote the set of non-analytic points of $\tau$ in  $(0,\widetilde q_c)$. Then $S$ is  discrete and  possibly empty. Moreover,  the cardinality of $S\cap (0,1]$ is not less than the number of times that the graphs of $T$ and $\tau_\nu$ cross each other transversally over $(0,1)$. These properties will be established in Section~5.
}
\end{rem}


%

Now we give some examples to illustrate Theorem~\ref{MA}.

\begin{ex}[Lognormal canonical cascades] Let us consider the standard lognormal canonical cascade, for which the weights $W_{i,j}$ are independent and  $W_{i,j}\sim m^{-2} \exp (\beta N-\beta^2/2)$, where $\beta\ge 0$ and  $N\sim \mathcal N(0,1)$. We have
$$
T(q)=2(q-1)-\frac{\beta^2}{2\log(m)}q(q-1).
$$
A necessary and sufficient condition for $\mu$ to be almost surely positive is $T'(1)=2-\frac{\beta^2}{2\log(m)}>0$, i.e. $\beta\in [0,2\sqrt{\log m})$.

Fix $\beta\in (0,2\sqrt{\log m})$ (we discard the case $\beta=0$ which corresponds to $\mu$ being the resctriction of  the Lebesgue measure to $[0,1]^2$). Then, the dimension of $\mu$ equals $2-\frac{\beta^2}{2\log(m)}$, and the measure $\nu$ is simply the Lebesgue measure restricted to $[0,1]$, so $\tau_\nu(q)=q-1$. Also, due to Theorem~\ref{ABS}, the measure $\pi_*\mu$ is almost surely equivalent to the Lebesgue measure if and only if $T'(1)>1$, i.e. $\beta\in [0, \sqrt{2\log(m)})$.

We also have $T'(q)q-T(q)=2-\frac{\beta^2q^2}{2\log(m)}$, so $q_c=2\sqrt{\log(m)}/\beta$. Moreover, $T(q)=\tau_\nu(q)$ if and only if  $q=1$ or $q=q_0:=2\log(m)/\beta^2$.

Thus, if $\beta\in [0,\sqrt{\log(m)}]$, we have $q_c\le\widetilde q_c=q_0$;  if $\beta\in (\sqrt{\log(m)}, \sqrt{2\log(m)})$,  $q_c=\widetilde q_c$ and $\tau_\nu$ and $T$ cross once transversally at $q_0\in (1,q_c)$, and do not cross over $[0,1)$; if $\beta= \sqrt{2\log(m)}$ then $\tau_\nu$ and $T$ cross at $1=q_0$ only and $q_c=\widetilde q_c$; if $\beta\in (\sqrt{2\log(m)}, 2\sqrt{\log(m)})$, then $T$ and $\tau_\nu$ cross once  transversally at  $q_0\in (0,1)$ and do not cross over $(1,\infty)$.

The previous observations and the definition of $\tau$ yield, with probability 1:

\begin{itemize}
\item if $\beta\in (0, \sqrt{\log m}]$, then $\tau(q)=q-1$ over $[0,q_0=\widetilde q_c]$ (and $q_0>1$).

\item If $\beta\in (\sqrt{\log(m)}, \sqrt{2\log(m)})$, $\tau(q)=q-1$ over $[0,q_0]$, $\tau(q)=T(q)$ over $[q_0,q_c=\widetilde q_c]$, and $q_0\in (1,q_c)$.

In this case $\pi_*\mu$ provides new examples  of statistically self-similar measures absolutely continuous with respect to the Lebesgue measure over $[0,1]$, with a non trivial Hausdorff spectrum and a first order phase transition, here at $q_0$ (see also \cite{Feng12} for deterministic examples for which, however, the Hausdorff spectrum is not described at the phase transition).

\item If  $\beta=\sqrt{2\log(m)}$, $\tau(q)=q-1$ over $[0,1=q_0]$ and $\tau(q)=T(q)$ over $[1,q_c=\widetilde q_c]$.

\item  If $\beta\in (\sqrt{2\log(m)}, 2\sqrt{\log(m)})$ then $q_0<1$, and a calculation using the definition of $\tau$ over $[0,1]$ shows that $\tau(q)= -1+T'(\sqrt{q_0}) q$ over  $[0, \sqrt{q_0}]$ and $\tau(q)=T(q)$ over $[ \sqrt{q_0}, q_c=\widetilde q_c]$.

In the last two cases, for which $\dim (\mu)\le 1$, our result provides, for the special directions of projection considered in this paper, the same information as that given by \cite{BaBh} for almost every direction, and recalled in Section 1.

\end{itemize}

Illustrations are provided by Figure~\ref{Fig0}.

\end{ex}

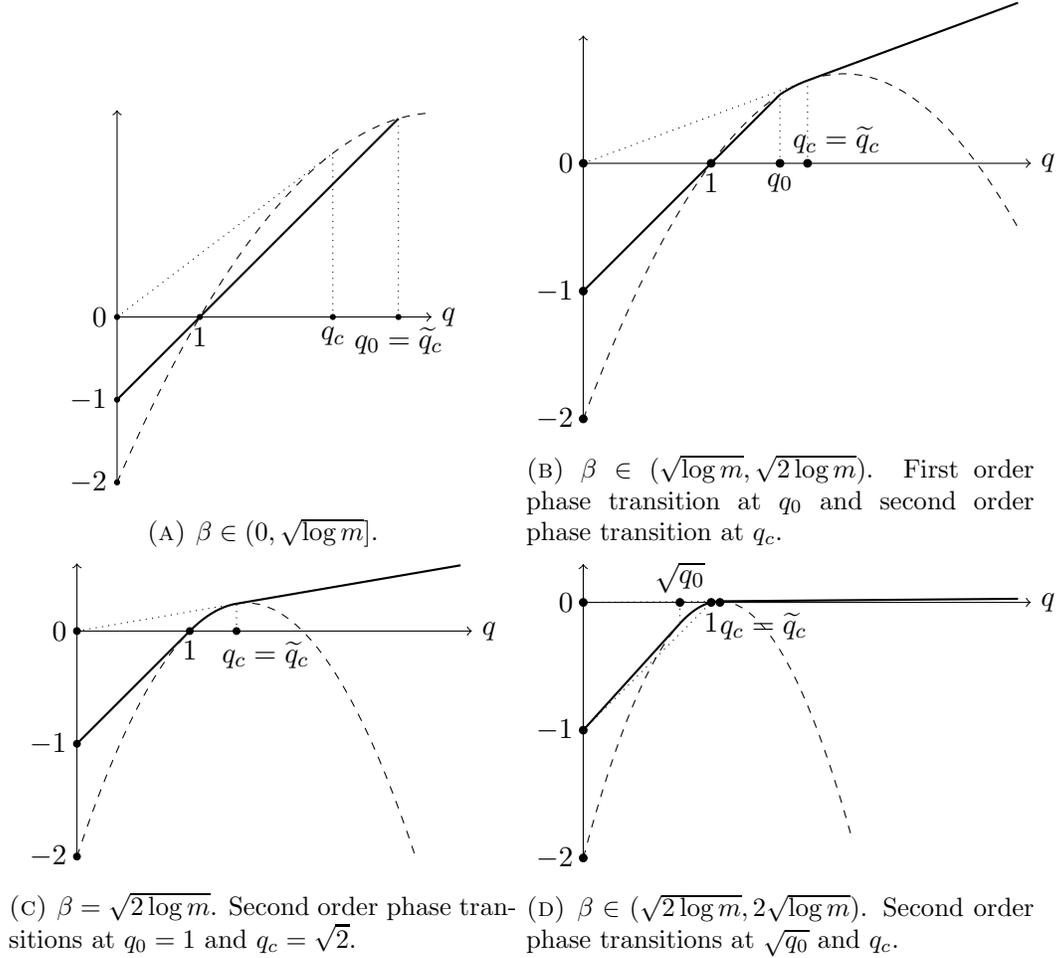
\begin{figure}[ht]
\begin{subfigure}[b]{0.45\textwidth}
\centering
\begin{tikzpicture}[xscale=1.1,yscale=1.1]
\draw [->] (0,-2) -- (0,2.5)  ;
\draw [->] (0,0) -- (3.8,0) node [right] {$q$};
\draw [thick, domain=0:3.4] plot (\x, {\x-1});
\draw [dashed, domain=0:3.8] plot (\x, {2*(\x-1)-(1/3.4)*\x*(\x-1)});
\draw [fill] (0,0) circle [radius=0.03] node [left] {$0$};
\draw [fill] (1,0) circle [radius=0.03] node [below] {$1$};
\draw [fill] (0,-1) circle [radius=0.03] node [left] {$-1$};
\draw [fill] (0,-2) circle [radius=0.03] node [left] {$-2$};
\draw [dotted] (3.4,0) -- (3.4, 2.4)  ;
\draw [dotted] ({2*sqrt{1.7}},0) -- ({2*sqrt{1.7}}, {2*((2*sqrt{1.7})-1)-(1/3.4)*(2*sqrt{1.7})*((2*sqrt{1.7})-1)})  ;
\draw [fill] (3.4,0) circle [radius=0.03] node [below] {$q_0=\widetilde q_c$};
\draw [fill] ({2*sqrt{1.7}},0) circle [radius=0.03] node [below] {$q_c$};
\draw [dotted, domain=0:2*sqrt{1.7}] plot (\x, {\x*((2*((2*sqrt{1.7})-1)-(1/3.4)*(2*sqrt{1.7})*((2*sqrt{1.7})-1))/(2*sqrt{1.7}))});
\end{tikzpicture}
\caption{$\beta\in (0, \sqrt{\log m}]$.}\label{fig1a}
\end{subfigure}
\begin{subfigure}[b]{0.45\textwidth}
\centering
\begin{tikzpicture}[xscale=1.7,yscale=1.7]
\draw [->] (0,-2) -- (0,1)  ;
\draw [->] (0,0) -- (3.5,0) node [right] {$q$};
\draw [thick, domain=0:2/1.3] plot (\x, {\x-1});
\draw [thick, domain=2/1.3:2/(sqrt{1.3})] plot (\x, {2*(\x-1)-.65*\x*(\x-1)});
\draw [dashed, domain=0:2/1.3] plot (\x, {2*(\x-1)-.65*\x*(\x-1)});
\draw [dotted, domain=0:2/(sqrt{1.3})] plot (\x, {\x*((2*(2/(sqrt{1.3})-1)-.65*(2/(sqrt{1.3}))*((2/(sqrt{1.3}))-1))/(2/(sqrt{1.3}))});
\draw [thick, domain=2/(sqrt{1.3}):3.4] plot (\x, {\x*((2*(2/(sqrt{1.3})-1)-.65*(2/(sqrt{1.3}))*((2/(sqrt{1.3}))-1))/(2/(sqrt{1.3}))});
\draw [dashed, domain=2/(sqrt{1.3}):3.4] plot (\x, {2*(\x-1)-.65*\x*(\x-1)});
\draw [dotted] (2/1.3,0) -- (2/1.3, 2/1.3-1)  ;
\draw [dotted] ({2/(sqrt{1.3})},0) -- ({2/(sqrt{1.3})}, {2*((2/(sqrt{1.3}))-1)-.65*(2/(sqrt{1.3}))*((2/(sqrt{1.3}))-1)})  ;
\draw [fill] (0,0) circle [radius=0.03] node [left] {$0$};
\draw [fill] (1,0) circle [radius=0.03] node [below] {$1$};
\draw [fill] (0,-1) circle [radius=0.03] node [left] {$-1$};
\draw [fill] (0,-2) circle [radius=0.03] node [left] {$-2$};
\draw [fill] (2/1.3,0) circle [radius=0.03] node [below] {$q_0$};
\draw [fill] (2/sqrt{1.3},0) circle [radius=0.03] node [above] {$\quad\quad q_c=\widetilde q_c$};
\end{tikzpicture}
\caption{$\beta\in (\sqrt{\log m}, \sqrt{2\log m})$. First order phase transition at $q_0$ and  second order phase transition at $q_c$.}\label{fig1a}
\end{subfigure}
\begin{subfigure}[b]{0.45\textwidth}
\centering
\begin{tikzpicture}[xscale=1.5,yscale=1.5]
\draw [->] (0,-2) -- (0,.6)  ;
\draw [->] (0,0) -- (3.5,0) node [right] {$q$};
\draw [thick, domain=0:1] plot (\x, {\x-1});
\draw [thick, domain=1:sqrt(2)] plot (\x,{2*(\x-1)-\x*(\x-1)});
\draw [dashed, domain=0:1] plot (\x,{2*(\x-1)-\x*(\x-1)});
\draw [dashed, domain=sqrt(2):3] plot (\x,{2*(\x-1)-\x*(\x-1)});
\draw [dotted, domain=0:sqrt(2)] plot (\x, {\x*((2*(sqrt(2)-1)-(sqrt(2))*(sqrt(2)-1))/(sqrt(2)))});
\draw [thick, domain=sqrt(2):3.4] plot (\x, {\x*((2*(sqrt(2)-1)-(sqrt(2))*(sqrt(2)-1))/(sqrt(2)))});
\draw [dotted] ({sqrt(2)},0) -- ({sqrt(2)}, {2*((sqrt(2))-1)-(sqrt(2))*((sqrt(2))-1)})  ;
\draw [fill] (0,0) circle [radius=0.03] node [left] {$0$};
\draw [fill] (1,0) circle [radius=0.03] node [below] {$1$};
\draw [fill] (0,-1) circle [radius=0.03] node [left] {$-1$};
\draw [fill] (0,-2) circle [radius=0.03] node [left] {$-2$};
\draw [fill] ({sqrt(2)},0) circle [radius=0.03] node [below] {$\quad\quad q_c=\widetilde q_c$};
\end{tikzpicture}
\caption{$\beta= \sqrt{2\log m}$. Second order phase transitions at $q_0=1$  and $q_c=\sqrt{2}$. }\label{fig1a}
\end{subfigure}
\begin{subfigure}[b]{0.45\textwidth}
\centering
\begin{tikzpicture}[xscale=1.7,yscale=1.7]
\draw [->] (0,-2) -- (0,.3)  ;
\draw [->] (0,0) -- (3.5,0) node [right] {$q$};
\draw [dotted, domain=0:1] plot (\x, {\x-1});
\draw [thick, domain=sqrt(2)/sqrt(3.5):2/sqrt(3.5)] plot (\x,{2*(\x-1)-1.75*\x*(\x-1)});
\draw [dashed, domain=0:sqrt(2)/sqrt(3.5)] plot (\x,{2*(\x-1)-1.75*\x*(\x-1)});
\draw [dashed, domain=2/sqrt(3.5):2.1] plot (\x,{2*(\x-1)-1.75*\x*(\x-1)});
\draw [dotted, domain=0:2/sqrt(3.5)] plot (\x, {\x*((2*((2/sqrt(3.5))-1)-1.75*(2/sqrt(3.5))*((2/sqrt(3.5))-1))/((2/sqrt(3.5)))});
\draw [thick, domain=2/sqrt(3.5):3.4] plot (\x, {\x*((2*((2/sqrt(3.5))-1)-1.75*(2/sqrt(3.5))*((2/sqrt(3.5))-1))/((2/sqrt(3.5)))});
\draw [thick] (0,-1) -- ({sqrt(2)/sqrt(3.5)}, {2*((sqrt(2)/sqrt(3.5))-1)-1.75*(sqrt(2)/sqrt(3.5))*((sqrt(2)/sqrt(3.5))-1)})  ;
\draw [dotted] ({sqrt(2)/sqrt(3.5)},0) -- ({sqrt(2)/sqrt(3.5)}, {2*((sqrt(2)/sqrt(3.5))-1)-1.75*(sqrt(2)/sqrt(3.5))*((sqrt(2)/sqrt(3.5))-1)})  ;
\draw [fill] (0,0) circle [radius=0.03] node [left] {$0$};
\draw [fill] (1,0) circle [radius=0.03] node [below] {$1$};
\draw [fill] (0,-1) circle [radius=0.03] node [left] {$-1$};
\draw [fill] (0,-2) circle [radius=0.03] node [left] {$-2$};
\draw [fill] ({sqrt(2)/sqrt(3.5)},0) circle [radius=0.03] node [above] {$\sqrt{q_0}$};
\draw [fill] ({2/sqrt(3.5)},0) circle [radius=0.03] node [below] {$\quad\quad \quad q_c=\widetilde q_c$};
\end{tikzpicture}
\caption{$\beta\in (\sqrt{2\log m}, 2\sqrt{\log m})$. Second order phase transitions at $\sqrt{q_0}$  and $q_c$.}\label{fig1a}
\end{subfigure}
\caption{The  thick curve represents $\tau_{\pi_*\mu}$ over $[0,\widetilde q_c]$ in case (A) and $[0,\infty]$ in the other cases, while the dashed curve represents $T$.}\label{Fig0}
\end{figure}


Below we construct a concrete example so that $q_c=\infty$ and the function $\tau_{\pi_*\mu}$ has a non-differentiable point in $(1, \infty)$ (i.e. first order phase transition), and a non-$C^\infty$ smooth point in $(0,1)$ (i.e. phase transition of order $\geq 2$). It is illustrated in Figure~\ref{Fig1}.

\begin{ex}\label{Ex1}{\rm    Let $(p_0,\ldots, p_{m-1})$ be a positive probability vector different from the vector $(m^{-1},\ldots,m^{-1})$. We have $p_{\max}=\max\{p_i:0\le i\le m-1\}>m^{-1}$. We assume that $p_0=p_{\max}>\sqrt{p_{1}}=\ldots =\sqrt{p_{m-1}}$. Fix $\beta$ in the interval $(m,mp_{\max}^{-1})$ and $\lambda\in (0,1)$. Let $(V_{i,j})_{1\le i\leq {m-1},0\leq j\le m-1}$ be a family of random variables which take value $\beta/m$ with probability $\lambda\beta^{-1}$ and $c_{m,\beta,\lambda}= \frac{\beta(1-\lambda)}{m(\beta-\lambda)}$ with probability $ 1-\lambda\beta^{-1}$. Let $$V_{0,0}\in \left(\max_{1\le i\le m-1}\frac{p_i}{p_{\max}^2},\; 1\right),$$
 $V_{0,1}=1-V_{0,0}$, and $V_{0,j}=0$ if $j\ge 2$; also suppose that $V_{0,1}<V_{0,0}$. Set $W_{i,j}=p_iV_{i,j}$ for all $0\le i,j\le m-1$ and define the functions $T_i$ and $T$ as previously.

For all $1\le i\le m-1$,  for all $0\le j\le m-1$ by construction we have $W_{i,j} \le p_i\beta/m<W_{0,0}<p_{0}<1$, and we also have $W_{0,1}<W_{0,0}<1$. Consequently,  $T'(1)>0$. Also,  for all $0\le i\le m-1$, we have
$
\E\left(\sum_{j=0}^{m-1} V_{i,j}\right )=1
$
and for $1\le i\le m-1$, we have
$$
-\log(m)T_i'(1)=\E\left(\sum_{j=0}^{m-1} V_{i,j}\log( V_{i,j})\right )=\lambda \log\left (\frac{\beta}{m} \right ) +m (1-\lambda\beta^{-1}) c_{m,\beta,\lambda}\log(c_{m,\beta,\lambda})
$$
and $-\log(m)T_0'(1)=V_{0,0}\log(V_{0,0})+V_{0,1}\log(V_{0,1})$. Thus, if we take  $\lambda$ close enough to 1 and $V_{0,0}$ close enough to 1,  we have $\sum_{i=0}^{m-1} p_iT'_{i}(1)<0$. Thus $0<T'(1)<\tau_\nu'(1)$, and   $T<\tau$ near $1+$. Now let us make  $T(q)$ explicit:
$$
T(q)=-\log_m\left ((p_0V_{0,0})^q+(p_0V_{0,1})^q+\sum_{i=1}^{m-1}m\lambda\beta^{-1}\Big (p_i\frac{\beta}{m} \Big )^q+m(1-\lambda\beta^{-1})(p_ic_{m,\beta,\lambda}) ^q\right ).
$$
We have $T(q)=-q\log_m (p_0V_{0,0})+o(1)$ as $q\to\infty$, with $-\log_m (p_0V_{0,0})>-\log_m (p_0)>0$ since $V_{0,0}<1$. This shows that $T^*\circ T'$ does not vanish over $\R_+$ and $q_c=\infty$. Moreover $\tau_\nu(q)=-q\log_m (p_{\max})+o(1)$ as $q\to\infty$, so $\tau_\nu(q)<T(q)$ near $\infty$. It follows that there is a first order phase transition over $(1,\infty)$.

Now let us look at the situation over $[0,1]$. We have  $-\log_m(m(m-1)+2)=T(0)<\tau_\nu(0)=-1$,  and $T'(1)<\tau_\nu'(1)$ implies that $\tau_\nu<T$ near $1-$. Thus, the graphs of $\tau_\nu$ and $T$ cross each other on $[0,1]$, and we know from Theorem~\ref{MA} that there is at least one phase transition of order at least 2. Let us be a little bit more precise. Set
$$
G(q,s)=\sum_{i=0}^{m-1}p_i^qm^{-T_i(s)q/s}.
$$
We have $$\frac{\partial G}{\partial s} (q,s)= -s^{-2}q\log (m) \sum_{i=0}^{m-1}p_i^qm^{-T_i(s)q/s} T_i^*(T_i'(s)).$$
 By construction, we have  $\frac{\partial G}{\partial s} (1,1)= -\log (m) \sum_{i=0}^{m-1}p_iT_i'(1)>0$.

Thus, by continuity of $\frac{\partial G}{\partial s} (q,s)$, for $q$ near $1-$ we have $\frac{\partial G}{\partial s} (q,s)>0$ for all $s\in [q,1]$, which   implies that $\tau$ is attained at $s=q$ and $\tau(q)=T(q)$. Also $T_i^*(T_i'(0))=-T_i(0)=1>0$ if $i\ge 1$, $T_0^*(T_0'(0))=-T_0(0)=\log_m(2)>0$ and by construction  $\sum_{i=0}^{m-1} T_i'(1)<0$. Consequently, for all $q$ near $0+$ we have $\frac{\partial G}{\partial s} (q,q)<0$ and $\frac{\partial G}{\partial s} (q,1)>0$, which implies that $\tau(q)$ is attained at some $s\in (q,1)$ and $\tau(q)>\max(\tau_\nu(q),T(q))$.
}
\end{ex}

\begin{center}
\begin{figure}
       {\includegraphics[width=0.3\textwidth,height=0.4\textwidth]{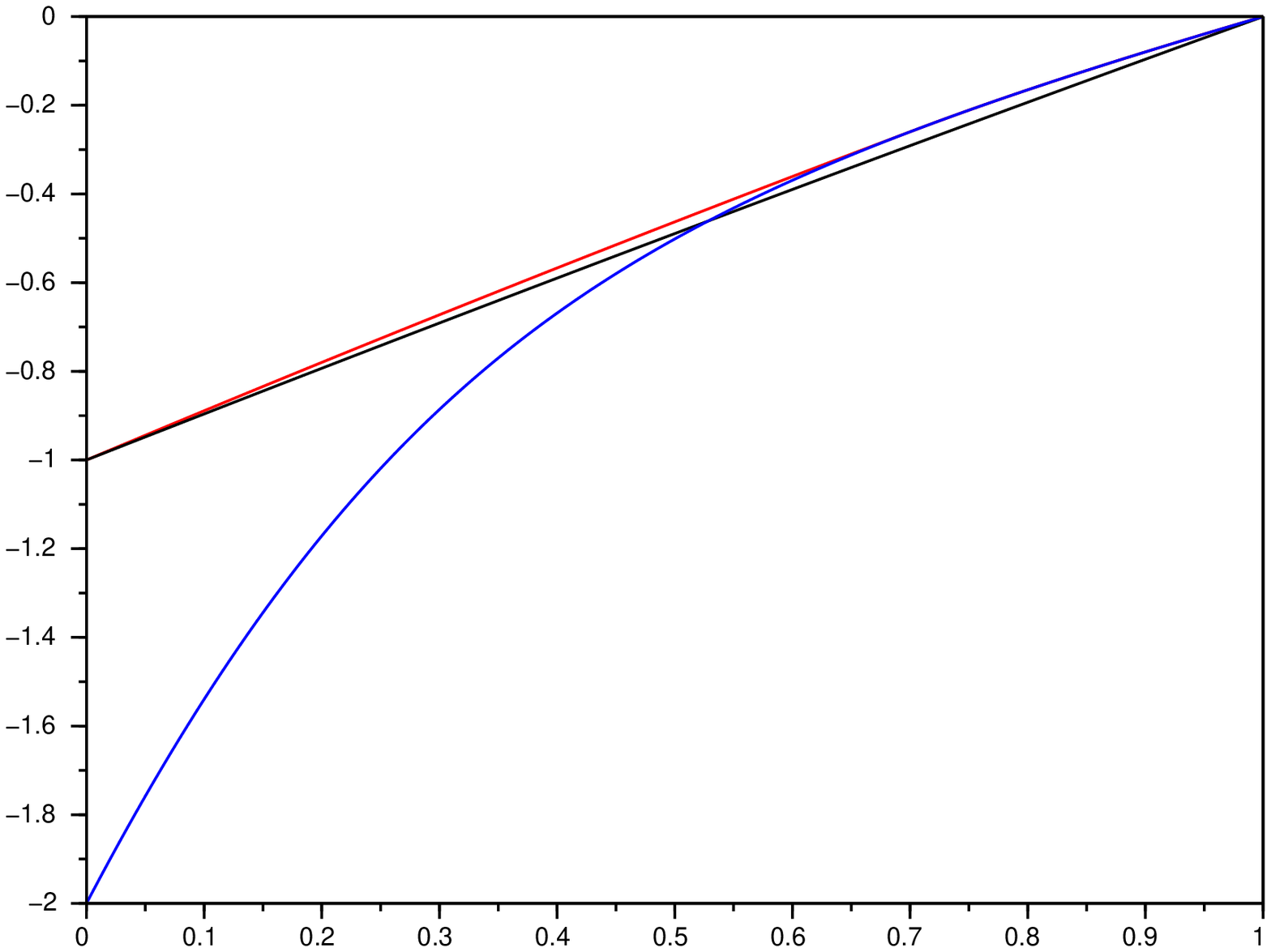}} {\includegraphics[width=0.56\textwidth,height=0.4\textwidth]{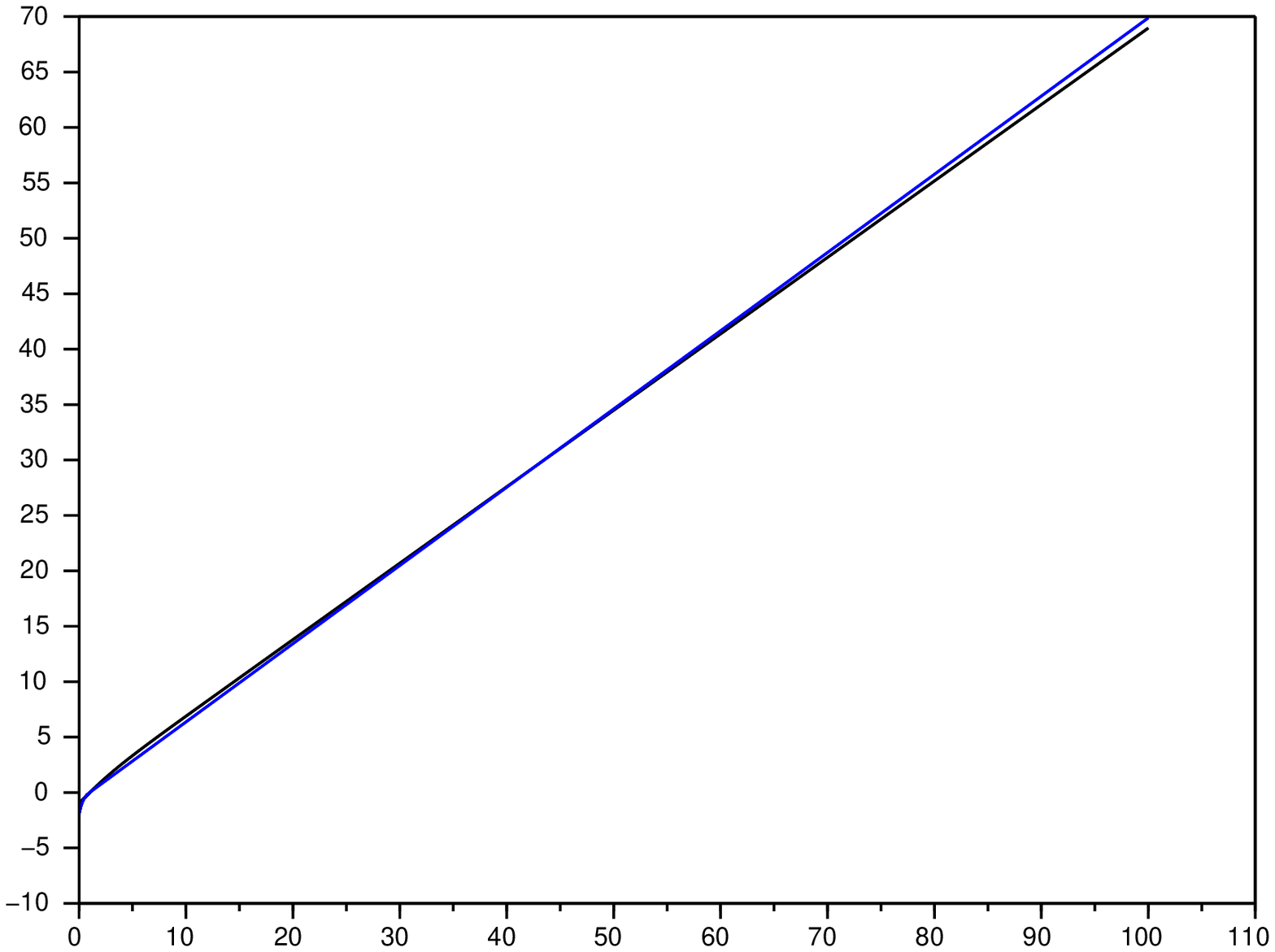}}
\vskip 0cm
\caption{Illustration of Example~\ref{Ex1} with $m=2$, $p_0=.62$, $\beta=3.22$, $\lambda=.99$ and $V_{00}=.99$ (blue curve: $T$; black curve: $\tau_\nu$; red curve: $\tau$). $q_c=\widetilde q_c=\infty$. One second order phase transition at some $q_0\in (0,1)$ and one first order phase transition at some $q'_0 \in (1,\infty)$.  $\tau_{\pi_*\mu}=\tau>\max(\tau_\nu,T)$ over $[0,q_0)$, $\tau_{\pi_*\mu}=T=\tau_\mu$ over $[q'_0,q_0]$, and $\tau_{\pi_*\mu}=\tau_\nu$ over $[q_0,\infty)$.}
\label{Fig1}
\end{figure}
\end{center}
\begin{ex}\label{Ex2}{\rm  This example exhibits two phase transitions over $[0,1]$ and no first order phase transition over $(1,q_c)$, with $q_c<\infty$ and $\tau_{\pi_*\mu}=\tau_\mu$ over $(1,\infty)$. Take $m=2$, $p_0\in (0,1)$, and $N_0$ and $N_1$ two  random integers taking values in $\{0,1,2\}$ and with positive expectation. Then for $0\le i,j\le 1$ define $V_{i,j}=(\E(N_i))^{-1} \mathbf{1}_{\{j\le N_i-1\}}$. This yields $T_i(q)=(q-1)\log \E(N_i)$, hence $T_i^*(T_i'(s))=-T_i(0)=T_i'(1)=\log \E(N_i)$ for all $s\ge 0$. Also,
$$
T(q)=-\log_2 \left (p_0^q \E(N_0)^{1-q}+(1-p_0)^q \E(N_1)^{1-q}\right ).
$$

We require $\E(N_0)<1$,
\begin{eqnarray}
\label{r1} &&\E(N_0) \E(N_1)>1,\\
\label{r2}&&  \E(N_0)^{p_0}\E(N_1)^{1-p_0}<1,\\
 \label{r3} &&\left(\frac{\E(N_0)}{p_0}\right)^{p_0}\left(\frac{\E(N_1)}{1-p_0}\right)^{1-p_0}>1.
\end{eqnarray}
Properties \eqref{r2} and \eqref{r3} yield $\tau_\nu'(1)>T'(1)>0$. Also \eqref{r1} implies $\E(N_1)+\E(N_2)>2$ hence $T(0)<-1=\tau_\nu (0)$.  The graphs of $\tau_\nu$ and $T$ cross each other on $[0,1]$.  Let $G$ be defined as in the previous example. Property \eqref{r1} yields $\frac{\partial G}{\partial s} (q,s)<0$ for all $s\in [q,1]$ if $q$ is close enough to 0, hence $\tau(q)$ is attained at $q=1$; $\tau(q)=\tau_\nu(q)$. Moreover, \eqref{r2} implies that $\frac{\partial G}{\partial s} (q,s)>0$ for all $s\in [q,1]$ if $q$ is close to $1$, hence $\tau(q)$ is attained at $s=q$: $\tau(q)=T(q)$. Then our study of $\tau$ in Section~\ref{study of tau} implies that on a non trivial interval we have $\tau(q)>\max(\tau_\nu(q),T(q))$, i.e. $\tau$ is given by a third analytic expression.

It is also possible to choose the parameters so that $T'(1)<-\log_2(p_0)=\tau_\nu'(+\infty)$ and $p_0>\E(N_0)$, hence $\tau_\nu>T$ over $(1,\infty)$ and $q_c<\infty$, which implies that $\tau_{\pi_*\mu}(q)=\tau_\mu(q)$ over $[1,\infty)$. A concrete choice is $p_0=.8$, $\E(N_0)=.6$ and  $\E(N_1)=1.8$.
}
\end{ex}
\begin{center}
\begin{figure}
       {\includegraphics[width=0.6\textwidth,height=0.6\textwidth]{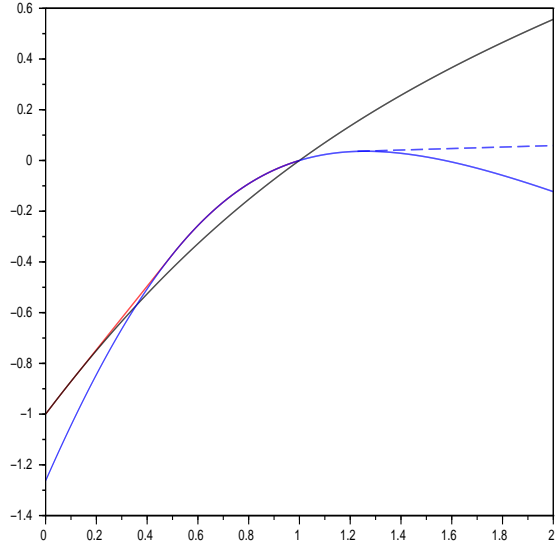}}
\caption{Illustration of Example~\ref{Ex2} with $p_0=.8$, $\E(N_0)=.6$ and  $\E(N_1)=1.8$ (blue curve: $T$; dashed blue curve: $\tau_\mu$; black curve: $\tau_\nu$; red curve: $\tau$). $q_c=\widetilde q_c\simeq 1.229<\infty$. Second order phase transition at some $q_0<q'_0$ in $(0,1)$, no first order phase transition over $(1,q_c)$, and one second order phase transition at $q_c$. $\tau_{\pi_*\mu}=\tau_\nu$ over $[0,q_0]$, $\tau_{\pi_*\mu}=\tau>\max(\tau_\nu,T)$ over $(q_0,q'_0)$, $\tau_{\pi_*\mu}=T=\tau_\mu$ over $[q'_0,q_c]$, and $\tau_{\pi_*\mu}(q)=\tau_\mu(q)=T'(q_c) q$ over $[q_c,\infty)$.}
\label{Fig2}
\end{figure}
\end{center}
\begin{ex}[Previous example continued] We can use the same model as in Example~\ref{Ex2} to get other different behaviors. See Figures~\ref{Fig3} to \ref{Fig5}.
\end{ex}
\begin{center}
\begin{figure}
       {\includegraphics[width=0.6\textwidth,height=0.55\textwidth]{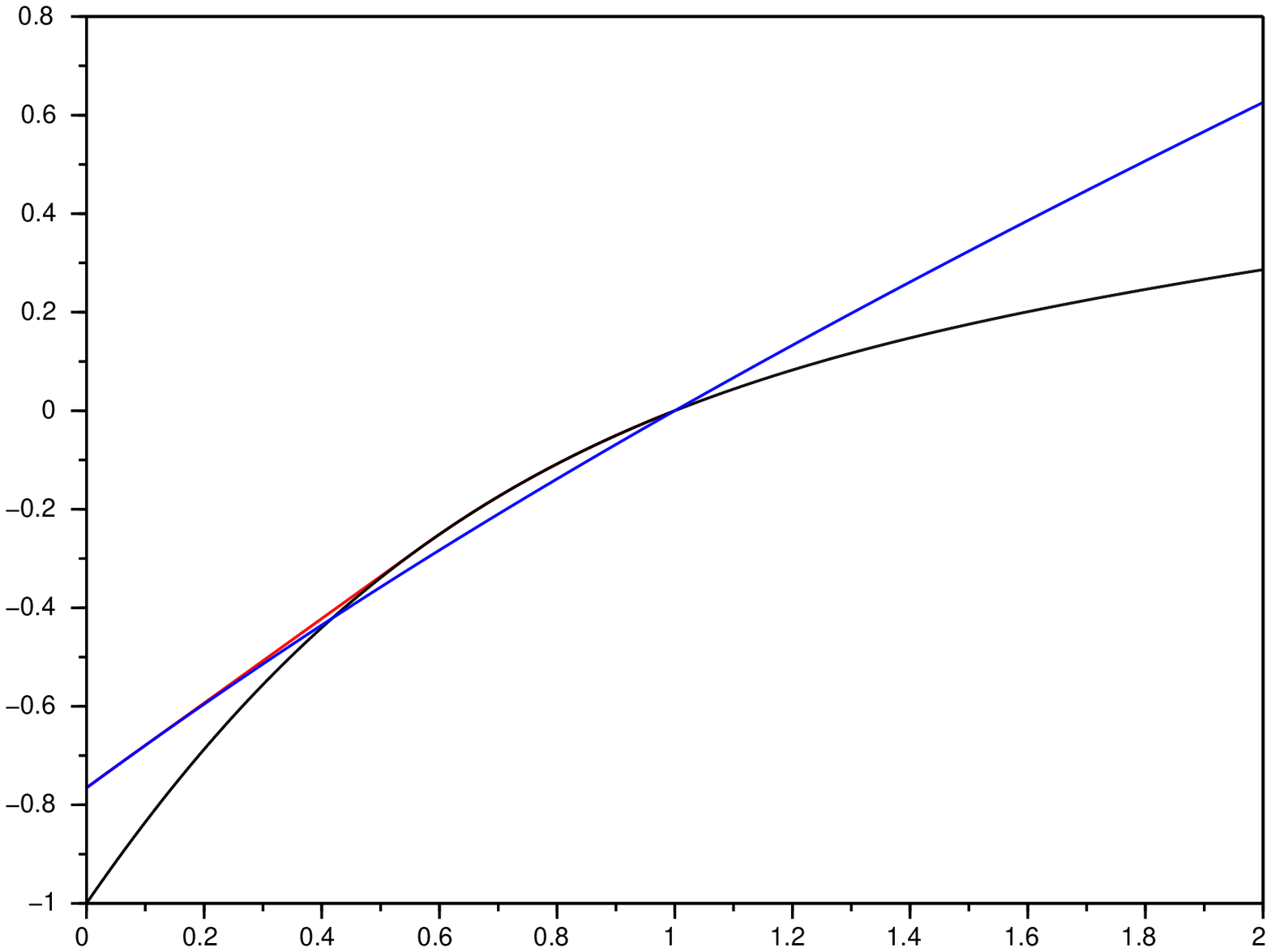}}
\caption{Same model as in  Example~\ref{Ex2} with $p_0=.1$, $\E(N_0)=.4$ and  $\E(N_1)=1.3$ (blue curve: $T$; black curve: $\tau_\nu$; red curve: $\tau$).  $q_c=\widetilde q_c=\infty$. Second order phase transitions at some $q_0<q'_0$ in $(0,1)$,  no first order phase transition over $(1,\infty)$. $\tau_{\pi_*\mu}=T=\tau_\mu$ over $[0,q_0]$, $\tau_{\pi_*\mu}=\tau>\max(\tau_\nu,T)$ over $(q_0,q'_0)$, and $\tau_{\pi_*\mu}=\tau_\nu$ over $[q'_0,\infty)$.}
\label{Fig3}
\end{figure}
\end{center}

\begin{center}
\begin{figure}
       {\includegraphics[width=0.6\textwidth,height=0.57\textwidth]{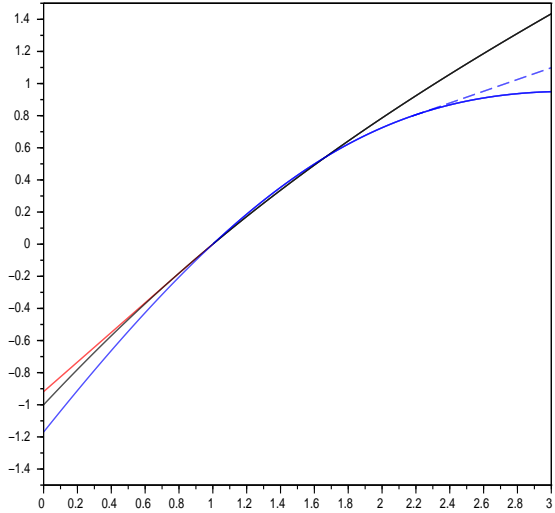}}
\caption{Same model as in  Example~\ref{Ex2} with $p_0=.3$, $\E(N_0)=.25$ and  $\E(N_1)=2$ (blue curve: $T$; dashed blue curve: $\tau_\mu$; black curve: $\tau_\nu$; red curve: $\tau$). $q_c=\widetilde q_c\simeq 2.176<\infty$. One second order phase transition at some $q_0\in (0,1)$. One first order phase transition at some $q'_0\in (1,q_c)$, and one second order phase transition at $q_c$.  $\tau_{\pi_*\mu}=\tau>\max(\tau_\nu,T)$ over $[0,q_0)$ (in particular $-\tau(0)=\dim_H \pi(K) <\min (-\tau_\nu(0),-T(0))$), $\tau_{\pi_*\mu}=\tau_\nu$ over $[q_0,q'_0]$ and $\tau_{\pi_*\mu}=T=\tau_\mu$ over $[q_0,q_c]$, and $\tau_{\pi_*\mu}(q)=\tau_\mu(q)=T'(q_c) q$ over $[q_c,\infty)$.}
\label{Fig4}
\end{figure}
\end{center}

\begin{center}
\begin{figure}
       {\includegraphics[width=0.6\textwidth,height=0.57\textwidth]{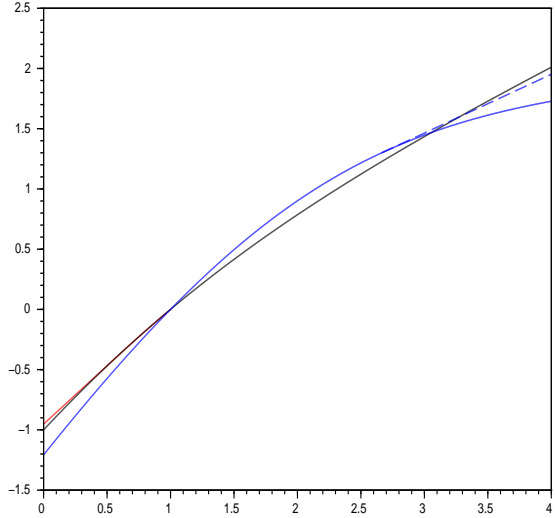}}
\caption{Same model as in  Example~\ref{Ex2} with $p_0=.3$, $\E(N_0)=.3$ and  $\E(N_1)=2$ (blue curve: $T$; dashed blue curve: $\tau_\mu$; black curve: $\tau_\nu$; red curve: $\tau$). $q_c\simeq 2.665<\infty$ and $\widetilde q_c\simeq 3.059$. One second order phase transition at some $q_0\in (0,1)$.   No first order phase transition over $[1,\widetilde q_c)$.  $\tau_{\pi_*\mu}=\tau>\max(\tau_\nu,T)$ over $[0,q_0)$ (in particular $-\tau(0)=\dim_H \pi(K) <\min (-\tau_\nu(0),-T(0))$) and $\tau_{\pi_*\mu}=\tau_\nu$ over $[q_0,\widetilde q_c]$.}\label{Fig5}
\end{figure}
\end{center}

%

\section{Proofs of Theorem~\ref{ABS}, Theorem~\ref{DIM}, and Corollary~\ref{DGF}}\label{pfDIM}

We first introduce the following new notations and definitions.


For each $u\in\Sigma_*$, we have
\begin{equation}\label{pimu}
\pi_*\mu([u])=\sum_{v\in\Sigma_{|u|}} \mu([u,v])
=\sum_{v\in\Sigma_{|u|}}Q(u,v)Y(u,v)=\nu(u) X(u),
\end{equation}
where
\begin{equation}\label{X(u)}
X(u)=\sum_{v\in\Sigma_{|u|}}Y(u,v)\prod_{j=1}^{|u|}V_{u_j,v_j}(u_{|j-1},v_{|j-1}).
\end{equation}

Define also
$$
\widetilde X(u)=\sum_{v\in\Sigma_{|u|}}\prod_{j=1}^{|u|}V_{u_j,v_j}(u_{|j-1},v_{|j-1}),
$$
and for all $x\in\Sigma$ and $n\ge 0$, set
$$
X_n(x)=X(x_{|n}) \quad  \text{and}\quad \widetilde X_n(x)=\widetilde X(x_{|n}).
$$

Now, let us start by presenting three results that will be used in this section. They will be proved in Section~\ref{MEST}, where they appear as Proposition~\ref{mom+estimate}, Corollary~\ref{momestimate}, and Proposition~\ref{negmom1} respectively.
\begin{pro}\label{mom+estimate'} Let $q>1$ such that $T(q)>0$.  Let $\eta$ be the Bernoulli product measure on $\Sigma$ generated by  a probability vector $(p_0',\ldots, p_{m-1}')$.  Set $A:=\max\{1, \sum_{i=0}^{m-1} p_i' m^{-T_i(q)}\}$. Then there exists  a polynomial  $f_q$ depending on $W$ and $q$ such that
\begin{equation}
\label{e-p1}
 A^n \leq  \mathbb{E}_{\mathbb{P}\otimes \eta} (X_n^q)\leq  f_q(n)  A^n,\quad \forall n\in \N.
\end{equation}
Furthermore, when $q\in (1,2]$, $f_q(n)$ can be taken as a constant.
\end{pro}

 \begin{cor}\label{momestimate'} Let $q>1$ such that $T(q)>0$. Then there exists  a polynomial  $f_q$ depending on $W$ and $q$ such that
\begin{equation}
\label{e-pp1}
 m^{-n \min \{\tau_\nu(q), T(q)\}} \leq  \mathbb{E}\Big (\sum_{u\in\Sigma_n}\pi_*\mu([u])^q\Big )\leq  f_q(n)  m^{-n \min \{\tau_\nu(q), T(q)\}}
\end{equation}
for all $n\in \N$.
 \end{cor}

 \begin{pro}\label{negmom2'}
 Let $\eta$ be the Bernoulli product measure on $\Sigma$ generated by  a probability vector $(p_0',\ldots, p_{m-1}')$. Assume that $T(q)>0$ and $\sum_{i=0}^{m-1} p_i' m^{-T_i(q)}<1$ for some $q\in (1,2]$, and that there exists $c\in (0,1)$ such that the following two conditions are satisfied:
 \begin{itemize}
 \item[(i)] $\Bbb P(\sup_{0\leq j\leq m-1} V_{i,j}>c)=1$ for all $0\leq  i\leq m-1$;
 \item[(ii)] $\E(\#\{j: V_{i,j}>c\})>1$  for all $0\leq  i\leq m-1$.
 \end{itemize}
 Then there exists $b>0$ such that $$
 \sup_{n\ge 1}\E_{\mathbb P\otimes \eta} (X_n^{-b})<\infty.
 $$
 \end{pro}

\subsection{Proof of Theorem~\ref{ABS}} \

\begin{proof}[Proof of Theorem~\ref{ABS}(1)] (i)
%
Since $(\Sigma,d)$ satisfies the Besicovitch covering property, we have almost surely
$
\pi_*\mu_\omega(\mathrm{d}x)= f(\omega,x) \, \nu(\mathrm{d}x)+\rho_\omega(\mathrm{d}x)$, where $\rho_\omega$ is a Borel measure singular with respect to $\nu$ and
\begin{equation}\label{density}
f(\omega,x)=\lim_{n\to\infty} \left(X_n(\omega,x)=\frac{\pi_*\mu_\omega([x_{|n}])}{\nu(x_{|n})}\right ),
\end{equation}
$\nu$-almost everywhere. Thus, if $\mathbb{E}_{\mathbb{P}\otimes \nu}( f)= \mathbb{E}(\|\pi_*\mu\|)=1$, we have $\rho_\omega=0$ almost surely, i.e. $\pi_*\mu$ is almost surely absolutely continuous with respect to $\nu$.

We know by  the construction of $\mu$ that for all $n\ge 1$ we have $\mathbb{E}_{\mathbb{P}\otimes \nu}(X_n)=\E(\|\mu_n\|)=1$. This implies that for all $\lambda\in (0,1)$ the sequence $(X_n^\lambda)_{n\ge 1}$ is uniformly integrable with respect to $\mathbb{P}\otimes \nu$,  hence by \eqref{density} we have  $\lim_{n\to\infty} \mathbb{E}_{\mathbb{P}\otimes \nu}(X_n^\lambda)=\mathbb{E}_{\mathbb{P}\otimes \nu}(f^\lambda)$.

Next  we claim that under $\mathbb{P}\otimes \nu$, $X_n$ converges in law to a  random variable $\widetilde X$. We postpone its proof to the next paragraph. Since  for any given $\lambda\in (0,1)$ the sequence $(X_n^\lambda)_{n\ge 1}$ is uniformly integrable,  we have $\lim_{n\to\infty}\mathbb{E}_{\mathbb{P}\otimes \nu}(X_n^\lambda)=\E_{\mathbb{P}\otimes \nu}(\widetilde X^\lambda)$. Hence $\mathbb{E}_{\mathbb{P}\otimes \nu}(f^\lambda)=\E_{\mathbb{P}\otimes \nu}(\widetilde X^\lambda)$ for all $\lambda\in (0,1)$. Furthermore, if $\E_{\mathbb{P}\otimes \nu}(\widetilde X)=1$, letting $\lambda$  tend to 1 we get $\mathbb{E}_{\mathbb{P}\otimes \nu}( f)=1$.

We now prove that $X_n$ converges in law to a  random variable $\widetilde X$ such that $\E_{\mathbb{P}\otimes \nu}(\widetilde X)=1$.

 By the definition of $X_n(x)$, for any $t>0$ we have
$$
\mathbb E_{\mathbb P\otimes\nu}(e^{-t X_n})=\mathbb E_{\mathbb P\otimes\nu} \Big (\prod_{v\in\Sigma_n} \phi_Y\Big (t\prod_{j=1}^{n}V_{x_j,v_j}(x_{|j-1},v_{|j-1})\Big ),
$$
where  $\phi_Y$ stands for the Laplace transform of $Y$, i.e. $\phi_Y(t)=\mathbb{E}(e^{-tY})$.

Let us show that $$M_n(x):=\max_{v\in\Sigma_n}\prod_{j=1}^{n}V_{x_j,v_j}(x_{|j-1},v_{|j-1})$$ converges in law to~0 under $ \mathbb P\otimes\nu$,  as $n$ tends to $\infty$. For $x\in \Sigma$, let  $\mathbb{Q}_x$ be the probability measure on $$(\Omega\times \Sigma,\sigma(V_{x_n}(x_{|n-1},v): n\ge 1,\, v\in \Sigma_{n-1})\otimes \mathcal{B}(\Sigma))$$
 whose restriction to $$\sigma(V_{x_j}(x_{|n-1},v): 1\le j\le n,\, v\in \Sigma_{j-1})\otimes \sigma([v]: v\in\Sigma_n)$$ is determined  by $$\mathbb{Q}_{x,n}(A\times [v])= \mathbb{E}\big (\mathbf{1}_A(\omega) \prod_{j=1}^{n}V_{x_j,v_j}(x_{|j-1},v_{|j-1})\big )$$ for $A\in \sigma(V_{x_j}(v): 1\le j\le n,\, v\in \Sigma_{j-1})$ and $v\in\Sigma_n$.
This yields a new skew product measure $\rho(\mathrm{d}\omega,\mathrm{d}x,\mathrm{d}y)=\nu (\mathrm{d}x)\mathbb{Q}_x(\mathrm{d}\omega,\mathrm{d}y)$ on $\Omega\times \Sigma^2$. A direct computation shows that  the random variables $(\omega,x,y)\mapsto V_{x_j,y_j}(x_{|j-1},y_{|j-1})$ are i.i.d. with respect to $\rho$, and their logarithms are of expectation
$$
\sum_{i=0}^{m-1} p_i \sum_{j=0}^{m-1} \E(V_{i,j} \log V_{i,j})=-\log(m) \sum_{i=0}^{m-1}p_i T_i'(1-)<0.
$$ It follows from the strong law of large numbers that for $\rho$-almost every  $(\omega,x,y)$ one has $$\lim_{n\to\infty} \prod_{j=1}^{n}V_{x_j,y_j}(x_{|j-1},y_{|j-1})=0.$$

 Now fix $\epsilon>0$. We have
\begin{eqnarray*}
\mathbb{P}\otimes\nu (M_n(x)\ge \epsilon)&\le& \E_{\mathbb{P}\otimes\nu}\Big (\sum_{v\in \Sigma_n}\mathbf{1}_{\{\prod_{j=1}^{n}V_{x_j,v_j}(x_{|j-1},v_{|j-1})\ge \epsilon\}}\Big )\\&\le & \E_{\mathbb{P}\otimes\nu}\Big (\sum_{v\in \Sigma_n}\epsilon^{-1} \mathbf{1}_{\{\prod_{j=1}^{n}V_{x_j,v_j}(x_{|j-1},v_{|j-1})\ge \epsilon\}} \prod_{j=1}^{n}V_{x_j,v_j}(x_{|j-1},v_{|j-1})\Big )\\
&=&\epsilon^{-1}\rho \Big(\Big\{\prod_{j=1}^{n}V_{x_j,y_j}(x_{|j-1},y_{|j-1})\ge \epsilon\Big\}\Big),
\end{eqnarray*}
and the right hand side converges to 0.

Consequently, since $\mathbb E(Y)=1$, we have $\phi_Y(u)=e^{-u+o(u)}$ near $0+$, so for each $t>0$ we can write
\begin{eqnarray}\label{LT}
\mathbb E_{\mathbb P\otimes\nu}(e^{-t X_n})&=&\mathbb E_{\mathbb P\otimes\nu}\Big (\mathbf{1}_{\{M_n(x)<\epsilon\}}e^{-t\widetilde X_n(1+O(\epsilon))}\Big )+ \mathbb E_{\mathbb P\otimes\nu}\Big (\mathbf{1}_{\{M_n(x)\ge \epsilon\}}e^{-t X_n}\Big )\\
\nonumber&=& \mathbb E_{\mathbb P\otimes\nu}\Big (e^{-t\widetilde X_n(1+O(\epsilon))}\Big )+R_n,
\end{eqnarray}
where $|R_n|\le 2 \mathbb P\otimes\nu (M_n(x)\ge \epsilon)$ and $\widetilde X_n(1+O(\epsilon))\ge 0$. On the other hand, the information gathered in Appendix~\ref{B1} applied with $\eta=\nu$ and $U_i=V_i$  shows that $(\widetilde X_n(x,\cdot))_{n\ge 1}$ is a Mandelbrot martingale in the random environment defined by $\nu$, and $\widetilde X_n$  converges $\mathbb{P}\otimes\nu$-almost surely to a limit $\widetilde X$. We then deduce from the bounded convergence theorem and the fact that $\mathbb P\otimes\nu (M_n(x)\ge \epsilon)$ tends to 0 as $n\to\infty$ that  $\E_{\mathbb P\otimes \nu}(e^{-tX_n})$ converges to $\E_{\mathbb P\otimes \nu}(e^{-t\widetilde X})$. Moreover, the condition $\dim(\mu)-\dim(\nu)=\sum_{i=0}^{m-1}p_iT_i'(1-)>0$ is sufficient for $(\widetilde X_n)_{n\ge 1}$ to be uniformly integrable (Theorem~\ref{thmBK}), hence $\E_{\mathbb P\otimes\nu}(\widetilde X)=1$.
\medskip

(ii) Since $T'(1-)>0$, the assumption that $T$ is finite on a neighborhood of $1$ implies that $T(s)>0$, hence $\E(Y^s)<\infty$ on a right neighborhood of $1$ (see \cite{KP} or \cite{DL}). Moreover, the assumption $\dim (\mu)>\dim(\nu)$ is equivalent to $\sum_{i=0}^{m-1} p_iT_i'(1)>0$, hence we have $\sum_{i=0}^{m-1}p_im^{-T_i(s)}<1$ on a right neighborhood of $1$. Also, if $s\in (1,2]$ and both $\E(Y^s)<\infty$ and $\sum_{i=0}^{m-1}p_im^{-T_i(s)}<1$, then $\sup_{n\ge 1}\mathbb E_{\mathbb P\otimes\nu}(X_n(x)^{s})<\infty$ by Proposition~\ref{mom+estimate'}. For  any such $s>1$, using \eqref{pimu} we get
\begin{eqnarray*}
\int_{\Sigma}\Big (\frac{\pi_*\mu([x_{|n}])}{\nu([x_{|n}])}\Big )^{s-1}\,\pi_*\mu(\mathrm{d}x)=\sum_{u\in\Sigma_n}\mathbf{1}_{\{\nu([u])>0\}} \Big (\frac{\pi_*\mu([u])}{\nu([u])}\Big )^{s-1}\pi_*\mu([u])= \sum_{u\in\Sigma_n}\nu([u]) X(u)^s.
\end{eqnarray*}
Thus
\begin{eqnarray*}
\sup_{n\ge 1}\E\left (\int_{\Sigma}\Big (\frac{\pi_*\mu([x_{|n}])}{\nu([x_{|n}])}\Big )^{s-1}\,\pi_*\mu(\mathrm{d}x)\right )=\sup_{n\ge 1}\mathbb E_{\mathbb P\otimes\nu}(X_n(x)^{s})<\infty.
\end{eqnarray*}
Consequently,  by the Fatou lemma,
\begin{eqnarray*}
\E\left (\liminf_{n\to\infty}\int_{\Sigma}\Big (\frac{\pi_*\mu([x_{|n}])}{\nu([x_{|n}])}\Big )^{s-1}\,\pi_*\mu(\mathrm{d}x)\right )&\le& \liminf_{n\to\infty}\E\left (\int_{\Sigma}\Big (\frac{\pi_*\mu([x_{|n}])}{\nu([x_{|n}])}\Big )^{s-1}\,\pi_*\mu(\mathrm{d}x)\right )\\
&<&\infty,
\end{eqnarray*}
from which we get
$$
\liminf_{n\to\infty}\int_{\Sigma}\Big (\frac{\pi_*\mu([x_{|n}])}{\nu([x_{|n}])}\Big )^{s-1}\,\pi_*\mu(\mathrm{d}x)<\infty\quad \text{a.s.}
$$
Due to \cite[Theorem 2.12(3)]{Mattila}, this implies both the absolute continuity of $\pi_*\mu$ with respect to $\nu$ and the desired result about the density of $\pi_*\mu$ with respect to $\nu$.

\medskip

(iii) At first notice that our assumption implies that the support of $\pi_*\mu$ equals that of $\nu$ almost surely. In particular, $X(u)>0$ for all $u$ such that $\nu(u)>0$. Thus, for any $s>1$ we have almost surely
$$
\sum_{u\in\Sigma_n}\mathbf{1}_{\{\pi_*\mu([u])>0\}}  \Big (\frac{\nu([u])}{\pi_*\mu([u])}\Big )^{s-1}\nu([u])= \sum_{u\in\Sigma_n}\mathbf{1}_{\{\nu([u])>0\}}\nu([u]) X(u)^{1-s}
$$
and
\begin{eqnarray*}
\E\Big (\sum_{u\in\Sigma_n}\mathbf{1}_{\{\pi_*\mu([u])>0\}}  \Big (\frac{\nu([u])}{\pi_*\mu([u])}\Big )^{s-1}\nu([u])\Big)=\mathbb E_{\mathbb P\otimes\nu}(X_n(x)^{1-s}).
\end{eqnarray*}
Due to our assumption on the random vectors $V_i$ and the fact that for $q$ close enough to~1 we have $\sum_{i=0}^{m-1}p_im^{-T_i(q)}<1$,  Proposition~\ref{negmom2'} yields $\sup_{n\ge 1} \mathbb E_{\mathbb P\otimes\nu}(X_n(x)^{1-s})<\infty$ if $s$ is close enough to 1. Similarly to (ii), this implies
$$
\liminf_{n\to\infty}\int_{\Sigma}\Big (\frac{\nu([x_{|n}])}{\pi_*\mu([x_{|n}])}\Big )^{s-1}\,\nu(\mathrm{d}x)<\infty\quad \text{a.s.}
$$
hence both the absolute continuity of $\nu$ with respect to $\pi_*\mu$ and the desired result about the density of $\nu$ with respect to $\pi_*\mu$.
\end{proof}

\begin{proof}[Proof of Theorem~\ref{ABS}(2)] If $\dim(\mu)<\dim(\nu)$, there is nothing to prove since $\overline \dim_P(\pi_*\mu)\le \dim(\mu)$.

Suppose now that $\dim(\mu)=T'(1)=\dim(\nu)$. This time, under $\mathbb P\otimes\nu$,  the martingale  $\widetilde X_n(\omega,x)$ converges to 0 almost surely since  $\sum_{i=0}^{m-1}p_iT_i'(1-)=0$ (see Theorem~\ref{thmBK} again). This implies that $M_n(x)=\max_{v\in\Sigma_n}\prod_{j=1}^{n}V_{x_j,v_j}(x_{|j-1},v_{|j-1})$ converge  to~0 almost surely under $ \mathbb P\otimes\nu$. Using \eqref{LT} this time yields the convergence in law to 0 for $X_n$, and $\E_{\mathbb P\otimes\nu}(f^\lambda)=0$ for all $\lambda\in (0,1)$. Consequently, $f=0$ with $\mathbb P\otimes\nu$ probability 1, which is equivalent to the fact that  $\pi_*\mu$ and $\nu$ are almost surely mutually singular.
\end{proof}

\subsection{Proof of Theorem~\ref{DIM}(1)}When $\dim(\mu)>\dim(\nu)$, since by Theorem~\ref{ABS}(1)(i) $\pi_*\mu$ is absolutely continuous with respect to $\nu$, we already know that if $\mu\neq 0$, we have $\dim(\pi_*\mu)=\dim(\nu)$. However, under the assumption that $T$ is finite in a neighborhood of 1, we give an alternative proof which works regardless of the respective positions  of~$\dim(\mu)$ and $\dim(\nu)$, and independently of absolute continuity considerations.

We will use Corollary~\ref{momestimate'} and the following elementary lemma.

\begin{lem}\label{lemma2.3} Let $\rho$ be a positive and finite Borel measure on $\Sigma$. Let $D\ge 0$. If for all $\epsilon>0$ there exists $q>1$ such that $\sum_{n\ge 1} m^{n (q-1)(D-\epsilon)}\sum_{|u|=n} \rho([u])^q<\infty$, then $\underline \dim_{\rm loc}(\rho,x)\ge D$ for $\rho$-almost every $x$. Also, if for all $\epsilon>0$ there exists $q\in(0,1)$ such that $\sum_{n\ge 1} m^{n(q-1)(D+\epsilon)}\sum_{|u|=n} \rho([u])^q<\infty$, then $\overline \dim_{\rm loc}(\rho,x)\le D$ for $\rho$-almost every $x$.
\end{lem}

\begin{proof}
Fix $\epsilon>0$. For all $q>1$ and $n\ge 1$, applying Markov's inequality we have
\begin{eqnarray*}
\rho\Big ( \Big \{x\in\Sigma: \; \frac{\log(\rho([x_{|n}]))}{-n\log(m)}\le D-\epsilon\Big\}\Big)&=&\rho\Big (\Big \{x\in\Sigma: \; \rho([x_{|n}])^{q-1}\ge m^{-n(q-1)(D-\epsilon)}\Big\}\Big)\\
&\le& m^{n(q-1)(D-\epsilon)}\ \int_\Sigma  \rho([x_{|n}])^{q-1}\, \rho(\mathrm{d}x)\\
&=&m^{n(q-1)(D-\epsilon)} \sum_{|u|=n} \rho([u])^q .
\end{eqnarray*}
Consequently, if $\sum_{n\ge 1}m^{n(q-1)(D-\epsilon)}\sum_{|u|=n} \rho([u])^q <\infty$, by the Borel-Cantelli lemma we get $\underline \dim_{\rm loc}(\rho,x)\ge D-\epsilon$ for $\rho$-almost every $x$.

The upper local dimension of $\rho$ is dealt with similarly.
\end{proof}

Recall that $\dim(\nu)=\tau_\nu'(1)$ and that almost surely, conditionally on $\mu\neq 0$, $\dim(\mu)=T'(1)$.  We deduce from corollary~\ref{momestimate'} that for $q>1$ close enough to 1, there exists a polynomial function $f_q$ such that for all $n\ge 1$ we have
\begin{eqnarray*}
\mathbb{E}\Big (\sum_{u\in\Sigma_n}\pi_*\mu([u])^q\Big )\le f_q(n)\cdot \begin{cases}
 m^{-n (q-1)\dim(\nu)+o(q-1)}&\text{if } T'(1)> \tau_\nu'(1)\\
m^{-n (q-1)T'(1)+o(q-1)}&\text{if } T'(1)\le \tau_\nu'(1)
\end{cases}
\end{eqnarray*}
as $q\to 1+$.  Fix $\epsilon>0$. Take $q$ close enough to $1$ so that the previous upper bound holds with $|o(q-1)|\le \epsilon(q-1)/4$. By Lemma~\ref{Borel} we conclude that, with probability 1, conditionally on $\mu\neq 0$,  for $n$ large enough we have $$m^{n(q-1)(D-\epsilon)}\sum_{|u|=n} \pi_*\mu([u])^q\le m^{-n\epsilon(q-1)/2},$$ with $D=\tau_\nu'(1)$ if $T'(1)>\tau_\nu'(1)$, and $D=T'(1)$ otherwise. Then Lemma~\ref{lemma2.3} yields the expected lower bound for $\underline \dim_{\rm loc}(\pi_*\mu,x)$, $\pi_*\mu$-almost everywhere.

To control $\overline \dim_{\rm loc}(\pi_*\mu,x)$, $\pi_*\mu$-almost everywhere, we only need to deal with the case $T'(1)> \tau_\nu'(1)$. Indeed, for $\pi_*\mu$-almost every $x$, we obviously have $\overline \dim_{\rm loc}(\pi_*\mu,x)\le \dim (\mu)$.

Now assume $T'(1)> \tau_\nu'(1)$. Let $q\in (0,1)$. We have
\begin{eqnarray*}
\mathbb{E}\Big (\sum_{u\in\Sigma_n}\pi_*\mu([u])^q\Big )
&=&\sum_{u\in\Sigma_n}\nu([u])^q \mathbb{E} (X(u)^q)\\
&\le & \sum_{u\in\Sigma_n}\nu([u])^q \mathbb{E} (X(u))^q= \sum_{u\in\Sigma_n}\nu([u])^q=m^{-n\tau_\nu(q)}.
\end{eqnarray*}
This is enough to conclude that $\overline \dim_{\rm loc}(\pi_*\mu,x)\le \tau_\nu'(1)$ for $\pi_*\mu$-almost every $x$
by using again Lemmas~\ref{Borel} and~\ref{lemma2.3}.

Putting together the previous arguments we conclude that with probability 1, conditionally on $\mu\neq 0$, $\pi_*\mu$ is exactly dimensional with $\dim(\pi_*\mu)=\dim(\nu)$ if $T'(1)>\tau_\nu'(1)$ and $\dim(\pi_*\mu)=\dim(\mu)$ if $T'(1)\le \tau_\nu'(1)$.

\subsection{Proof of Theorem~\ref{DIM}(2): Two approaches to the dimension of the conditional measures}\label{sec-3.2.1} We will give two different approaches to the calculation of the dimension of the conditional measures. The first one will assume that $T$ is finite in a neighborhood of $1$ and adapt the original approach by Peyri\`ere \cite{KP} and Liu-Rouault \cite{LR} to compute the dimensions of Mandelbrot measures. The second one, more, conceptual, will require no additional assumption, and combine a reduction to the case of Mandelbrot measures in random environment   with the  percolation approach developed by Kahane {\cite{K87}} to remove the extra hypothesis assumed on the moments of orders greater than 1 in \cite{KP,LR} in the case of Mandelbrot measures.

\medskip

{We start with preliminary definitions.} When $\mu\neq 0$, for $\pi_*\mu_\omega$-almost every $x$, there exists a conditional measure $\mu^x_\omega $ supported on $K^x=\pi^{-1} (\{x\})\cap K$, obtained as the weak-star limit, as $n\to\infty$, of the measures $\mu_{\omega,n}^x$  obtained on $\pi^{-1} (\{x\})$ by assigning uniformly the mass $\frac{\mu_\omega([x_{|n}]\times [J])}{\pi_*\mu_\omega([x_{|n}])}$ to each cylinder $[J]$ of generation~$n$, so that we have
$$
\mu_\omega (\mathrm{d}x,\mathrm{d}y)=\pi_*\mu_\omega (\mathrm{d}x)  \mu^x_\omega (\mathrm{d}y).
$$
{To be more specific, for any cylinder $[J]$, almost surely, the measurable set $$A_J=\{(\omega,x)\in\Omega\times \Sigma: \lim_{n\to\infty}\frac{\mu_\omega([x_{|n}]\times [J])}{\pi_*\mu_\omega([x_{|n}])} \text{ exists}\}$$ is of full $\widehat {\mathbb Q}$-probability, where  we define $\widehat {\mathbb Q }({\rm d}\omega,{\rm d}x)=\mathbb{P}({\rm d}\omega) \pi_*\mu_\omega(\mathrm{d}x)$, and for all $(\omega,x)$ in a subset $A'_J$ of $A_J$ of full $\widehat {\mathbb Q}$-probability, we have $ \mu^x_\omega (J)=\lim_{n\to\infty}\frac{\mu_\omega([x_{|n}]\times [J])}{\pi_*\mu_\omega([x_{|n}])}$. }

{ Suppose now that $T(1-)>\tau_\nu'(1)$, so that $\mathbb P$-almost surely,  $\pi_*\mu_\omega$ is absolutely continuous with respect to $\nu$.  There exists a measurable set $A'$ of full $\widehat {\mathbb Q}$-probability such that for all $(\omega,x)\in A'$, we have $$f_\omega(x)=\lim_{n\to\infty} \left(f_{\omega,n}(x):=\frac{\pi_*\mu_\omega ([x_{|n}])}{\nu([x_{|n}])}\right),$$
 where the limit exists and is positive. 

Set $A=A'\cap \bigcap_{J\in \Sigma_*}A'_J$. For all $(\omega,x)\in A$,   the sequence of measures $\widetilde\mu_{\omega,n}^x=  f_{\omega,n}(x) \mu_{\omega,n}^x $ weakly converges to the measure $\widetilde\mu_\omega^x$ defined as $ f_\omega(x) \mu_\omega^x$.

 Let
\begin{align*}
\Omega_A&=\{\omega:\  (\omega,x)\in A\text{ for some $x\in\Sigma$}\},\\
F^\omega&=\{x\in\Sigma: \ (\omega,x)\in A\}, \quad  \forall\,\omega\in\Omega_A.
\end{align*}

 Now, if $(\omega,x)\not\in  A$, set $\mu_\omega^x=\widetilde\mu_\omega^x=0$.

Also, for $n\ge 1$ and $(u,J)\in\Sigma_n\times\Sigma_n$ and $(L,K)\in \Sigma_p\times\Sigma_p$ we define
$$
Q^{u,J}(L,K)= \prod_{\ell=1}^p W_{L_\ell,K_\ell}(u (L_{|\ell-1}),J(K_{|\ell-1})),
$$
where $L_\ell$ and $K_\ell$ stand for the $\ell$-th letter of $L$ and $K$ respectively.

For $x\in  F^\omega$ and $J\in\Sigma_n$, we have
\begin{eqnarray*}
\widetilde \mu^x([J])=\lim_{p\to\infty} \widetilde \mu_p^x([J])= \lim_{p\to\infty} \frac{\mu([x_{|n}\cdot (\sigma^nx)_{|p}]\times [J])}{\nu ([x_{|n+p}])},
\end{eqnarray*}
and by construction of $\mu$, we have
$$
\frac{\mu([x_{|n}\cdot (\sigma^nx)_{|p}]\times [J])}{\nu ([x_{|n+p}])}= \frac{Q(x_{|n},J)}{\nu([x_{|n}])} \, X^{x_{|n},J}((\sigma^nx)_{|p}),
$$
where
\begin{eqnarray}\label{XnJ}
X^{x_{|n},J}((\sigma^nx)_{|p})=\sum_{K\in\Sigma_p}\frac{Q^{x_{|n},J}((\sigma^n x)_{|p},K)Y(x_{|n+p},JK)}{\nu([(\sigma^n x)_{|p}])}.
\end{eqnarray}
 Subsequently, we have
$$
\widetilde \mu^x([J])=\frac{Q(x_{|n},J)}{\nu([x_{|n}])} X^{x_{|n},J}(\sigma^nx),\text{ where }
X^{x_{|n},J}(\sigma^nx)=\lim_{p\to\infty} X^{x_{|n},J}((\sigma^n x)_{|p}),
$$
and for $y\in K^x$,
$$
\log (\widetilde\mu^x([y_{|n}]))=\log (Q(x_{|n},y_{|n})-\log (\nu([x_{|n}])) +\log (X^{x_{|n},y_{|n}}(\sigma^nx)).
$$

\noindent
{\bf First approach:}  Now observe that by  construction, ``$\mathbb{P}\otimes \nu$-almost surely, $\widetilde \mu^x$-almost everywhere'' is equivalent to ``$\mathbb{P}$-almost-surely, if $\mu\neq 0$, $\mu$-almost everywhere'', i.e. almost surely under the Peyri\`ere probability measure $\mathbb{P}(\mathrm{d}\omega)\mu_\omega (\mathrm{d}x,\mathrm{d}y)$. Under this measure, the random variables $\log (W_{x_k,y_k}(x_{|k-1},y_{|k-1})) - \log (p_{x_k})$, $k\ge 1$, are i.i.d.~and integrable, with expectation $\sum_{i,j}\E(W_{i,j}\log(W_{i,j}))-\sum_{i}p_j\log(p_j)$, hence by the strong law of large numbers we have
$$
\lim_{n\to\infty} \frac{\log (Q(x_{|n},y_{|n})-\log (\nu([x_{|n}]))}{-n\log m}=\dim (\mu)-\dim(\nu),
$$
$\mathbb{P}\otimes \nu$-almost surely, $\widetilde \mu^x$-almost everywhere.

To conclude  with this first approach, we  show that $\lim_{n\to\infty}\log (X^{x_{|n},y_{|n}}(\sigma^nx))/n=0$,  $\mathbb{P}\otimes \nu$-almost surely, $\widetilde \mu^x$-almost everywhere. To do so, we assume that $T$ is finite in a neighborhood of $1$. In particular, there exists  $\epsilon\in (0,1]$ such that $\mathbb E(Y^{1\pm\epsilon})<\infty$. By construction,  we have
\begin{eqnarray*}
\mathbb{E}\Big (\int_\Sigma \int_\Sigma (X^{x_{|n},y_{|n}}(\sigma^nx))^{\pm \epsilon}\widetilde \mu^x({\rm d}y)\nu(\mathrm{d}x)\Big )&=&\sum_{J\in\Sigma_n}\mathbb{E}\Big (\int_\Sigma  (X^{x_{|n},J}(\sigma^nx))^{\pm\epsilon}\widetilde \mu^x([J])\nu(\mathrm{d}x)\Big )\\
&=&\sum_{J\in\Sigma_n}\mathbb{E}\Big (\int_\Sigma  \frac{Q(x_{|n},J)}{\nu([x_{|n}])}(X^{x_{|n},J}(\sigma^nx))^{1\pm\epsilon}(\mathrm{d}x)\Big ).
\end{eqnarray*}
By the Fatou lemma we have
\begin{align*}
 \mathbb{E}\Big (& \int_\Sigma  \frac{Q(x_{|n},J)}{\nu([x_{|n}])}(X^{x_{|n},J}(\sigma_n(x)))^{1\pm\epsilon}\nu(\mathrm{d}x)\Big )\\
&\le\liminf_{p\to\infty} \mathbb{E}\Big (\int_\Sigma  \frac{Q(x_{|n},J)}{\nu([x_{|n}])}(X^{x_{|n},J}(\sigma^nx_{|p}))^{1\pm\epsilon}\nu(\mathrm{d}x)\Big )\\
&=\liminf_{p\to\infty} \mathbb{E}\Big ( \sum_{I\in\Sigma_n}Q(I,J) \sum_{L\in\Sigma_p}(X^{I,J}(L))^{1\pm\epsilon}\nu([L])\Big )\\
&=\liminf_{p\to\infty}  \sum_{I\in\Sigma_n}\mathbb{E}(Q(I,J))\mathbb{E}\Big ( \sum_{L\in\Sigma_p}(X^{I,J}(L))^{1\pm\epsilon}\nu([L])\Big ).
\end{align*}

Now we notice that $(X^{I,J}(L))_{L\in\Sigma_p}$ has the same probability distribution as  $(X(L))_{L\in\Sigma_p}$. Since $\sum_{i=0}^{m-1}p_iT_i'(1)>0$, using   Proposition~\ref{mom+estimate'}, for $\epsilon$ small enough we can get $C_\epsilon\ge 1$ such that for all $I,J\in\Sigma_n$ and $p\ge 1$
$$
\mathbb E\Big ( \sum_{L\in\Sigma_p}(X^{I,J}(L))^{1+\epsilon}\nu([L])\Big )\le C_{\epsilon} \max \left (1,\sum_{i=0}^{m-1}p_im^{-T_i(1+\epsilon)}\right)^n\le C_{\epsilon}.
$$
Moreover, $\mathbb E(X(L)^{1-\epsilon})\le \mathbb E(X(L))^{1-\epsilon}=1$. For such an $\epsilon$, we finally get that for all $\eta>0$, for all $n\ge 1$,
\begin{equation*}
\begin{split}
\mathbb{E} & \Big (\int_\Sigma \int_\Sigma (e^{\mp n\eta}X^{x_{|n},y_{|n}}(\sigma_n(x)))^{\pm\epsilon}\widetilde \mu^x({\rm d}y)\nu(\mathrm{d}x)\Big )\\
&\mbox{}\quad  \le C_{\epsilon} e^{-n\eta\epsilon}\sum_{(I,J)\in\Sigma_n\times\Sigma_n} \mathbb{E}(Q(I,J))=C_{\epsilon} e^{-n\eta\epsilon},
\end{split}
\end{equation*}
hence
$$
\mathbb{E}\int_\Sigma \int_\Sigma \sum_{n\ge 1}(e^{\mp n\eta}X^{x_{|n},y_{|n}}(\sigma_n(x)))^{\pm\epsilon}\widetilde \mu^x({\rm d}y)\nu(\mathrm{d}x)<\infty.
$$
It follows that with probability 1, if $\mu\neq 0$,  for $\nu$-almost every $x\in \mathrm{supp}(\mu)$,
$$
-\eta\le  \liminf_{n\to\infty}\log (X^{x_{|n},y_{|n}}(\sigma^n x))/n \le \limsup_{n\to\infty}\log (X^{x_{|n},y_{|n}}(\sigma^n x))/n\le \eta.
$$
Since $\eta>0$ can be taken arbitrarily small, we get the desired limit.

\medskip

\noindent
{\bf Second approach:} Now let us explain the more conceptual approach which does not assume anything else but $T'(1-)> \tau_\nu'(1)$.

Recall \eqref{XnJ} and write
\begin{eqnarray*}
X^{x_{|n},J}((\sigma^nx)_{|p})&=&\sum_{K\in\Sigma_p}\frac{Q^{x_{|n},J}((\sigma^n x)_{|p},K)Y(x_{|n+p},JK)}{\nu([(\sigma^n x)_{|p}])}\\
&=&\sum_{K\in\Sigma_p}Y(x_{|n+p},JK)\prod_{\ell=1}^pV_{x_{n+\ell},K_\ell}({x_{|n+\ell-1},J(K_{|\ell -1}})).\end{eqnarray*}

The proof of Theorem~\ref{DIM}(1)(a)(i) yields the convergence in law of $X_p$ to that of $\widetilde X_V$ as $p\to\infty$, where  $\widetilde X_V$ is the limit of the Mandelbrot martingale $\widetilde X_{V,n}$ defined in Appendix~\ref{B1} (take $U=V$ and $\eta=\nu$ there). This property  extends to the convergence, for every $n\ge 1$,  of the law of the vector $(X^{x_{|n},J}((\sigma^nx)_{|p}))_{J\in\Sigma_n}$ under $ \mathbb{P}\otimes\nu$  to that of the vector $(\widetilde X^{x_{|n},J}(\sigma^nx))_{J\in\Sigma_n}$, where  $\widetilde X^{x_{|n},J}(\sigma^nx)$ is the random variable  $\widetilde X_V^{x_{|n},J}(\sigma^nx)$ defined in \eqref{XUnx} in  Appendix~\ref{B1}. Consequently, since for each $p\ge 2$ we have the branching  property
$$
X^{x_{|n},J}((\sigma^nx)_{|p})=\sum_{j=0}^{m-1} X^{x_{|n+1},Jj}((\sigma^{n+1}x)_{|p-1}) V_{x_{n+1},j}(x_{|n},J),
$$
setting, for $p\ge n$,
$$
\rho_{\omega,n,p}^x:\; J\in \bigcup_{k=0}^n\Sigma_k\mapsto \widetilde X^{x_{||J|},J}((\sigma^{|J|}x)_{p})\cdot \prod_{k=1}^{|J|} V_{x_k,J_k}(x_{k-1},J_{k-1}),
$$
we get the convergence in law of $\mathbf{1}_{A}\rho_{\omega,n,p}^x$ to $\mathbf{1}_{A}\widetilde\mu^x_V$ restricted to $\bigcup_{k=0}^n\Sigma_k$ as $p\to\infty$, where $\widetilde\mu^x_V$ is the  Mandelbrot measure in a random environment described in \eqref{muUnx}.
However, by definition, for $(\omega,x)\in A$,   $\rho_{\omega,n,p}^x$ converges almost surely to $\widetilde\mu^x$ restricted to $\bigcup_{k=0}^n\Sigma_k$. Hence for all $n\ge 1$, we have the identity in distribution of the restrictions to  $\bigcup_{k=0}^n\Sigma_k$ of  $\widetilde\mu^x$ and $\mathbf{1}_{A}\widetilde\mu^x_V$, that is the equality in distribution of $\widetilde\mu^x$ and $\mathbf{1}_{A}\widetilde\mu^x_V$. Moreover, we notice that by construction,  up to a $\mathbb{P}\otimes \nu$ negligible set, $A$ is a subset of the set  $\Sigma_{A,V}$ of those points $(\omega,x)$ for which $\widetilde\mu^x_V\neq 0$. Consequently, if, conditionally on $\widetilde\mu^x_V\neq 0$, we have that $\widetilde\mu^x_V$ is exactly dimensional with dimension $\dim(\mu)-\dim(\nu)=\sum_{i=0}^{m-1}p_iT'_i(1-)$, then the same holds for $\widetilde\mu^x$.

Now, it is straightforward to adapt Kahane's percolation approach \cite{K87} developed for Mandelbrot measures in the so-called canonical case to get the conclusion.

At first we notice that conditionally on $\mu_V^x\neq 0$, the proof of the first approach applied to $\widetilde \mu^x_V$ instead of $\widetilde \mu^x$ yields $\overline \dim_{\mathrm{loc}}(\widetilde \mu^x_V,y)\le \sum_{i=0}^{m-1}p_iT'_i(1-))$  for $\widetilde \mu_V^x$-almost every $y$ (in the proof it just corresponds to proving that $-\eta\le \liminf_{n\to\infty}\widetilde X^{x_{|n},y_{|n}}_V(\sigma^nx)$ for all $\eta>0$, which holds because $\E_{\mathbb{P}\otimes \nu}(\widetilde X_V^{1-\epsilon})<\infty$ without additional assumption). Thus  $\overline\dim_P(\rho)\le \sum_{i=0}^{m-1}p_iT'_i(1-)$.

For each $\alpha\in (0,1)$, let $W^{(\alpha)}$ be a random variable taking value $m^\alpha$ with probability $m^{-\alpha}$ and value 0 with probability $1-m^{-\alpha}$. Then let $V^{(\alpha)}$ be a random vector whose coordinates are independent copies of $W^{(\alpha)}$, and consider $V^{(\alpha,\omega')}(u,v)_{(u,v)\in\Sigma_n\times\Sigma_n, n\ge 0}$, a sequence of independent copies of $V^{(\alpha)}$ defined on a space $(\Omega',\mathcal A',\mathbb{P}') $. For each $(\omega',x)\in\Omega'\times\Sigma$, consider the sequence of operators $(Q_{\lambda,n} (\omega',x))_{n\ge 1}$ acting on the finite non-negative Borel measures on $\pi^{-1}(\{x\})$ as follows:
$$
Q_{\alpha,n}(\omega',x)(\rho)(\mathrm{d}y)=  \left (\prod_{k=1}^n V^{(\alpha,\omega')}_{x_k,y_k}(x_{|k-1},y_{|k-1})\right)\cdot \rho (\mathrm{d}y).
$$
For each such measure $\rho$, $Q_{\alpha,n}(\omega',x)(\rho)$ is a martingale which converges $\mathbb{P}'$-almost surely in the weak-star topology to a measure denoted as $Q_\alpha(\omega',x)\cdot\rho$.  Moreover, one deduces from \cite[Corollaire du th\'eor\`eme 1]{K87} that if $\rho\neq 0$ and the martingale $(\|Q_{\alpha,n}(\omega',x)(\rho)\|)_{n\ge 1}$ is uniformly integrable (in Kahane's terminology it means that $Q_{\alpha,n}$ acts fully on $\rho$), then $\underline\dim_H(\rho)\ge \alpha$.

Now we consider the product space $(\Omega\times\Sigma)\times\Omega'$ endowed with the tensor $\sigma$-field $\mathcal A\otimes \mathcal B(\Sigma)\otimes \mathcal A'$ and the product probability measure $\mathbb{P}\otimes\nu\otimes\mathbb{P}'$. It remains to prove that for all $\alpha\in (0,\sum_{i=0}^{m-1}p_iT'_i(1-))$, for $\mathbb{P}\otimes\nu$-almost every $(\omega,x)$, conditionally on $\mu_V^x\neq 0$, the martingale $(\|Q_{\alpha,n}(\omega',x)(\widetilde\mu_V^x)\|)_{n\ge 1}$ is uniformly integrable. The proof follows similar lines as in the deterministic environment case (see \cite{WW,FK} for details).  It comes from the fact that the Mandelbrot martingale in the variable $(\omega,\omega')$
$$
\widetilde X_{V^{(\alpha)}V,n}(\omega,\omega',x)=\sum_{|v|=n} \prod_{k=1}^n V^{(\alpha,\omega')}_{x_k,y_k}(x_{|k-1},y_{|k-1}) V^\omega_{x_k,y_k}(x_{|k-1},y_{|k-1})
$$
in random environment $x$ taken under $\nu$ is uniformly integrable. Indeed,
$$
-\E_{(\mathbb P\otimes\mathbb{P}') \otimes\nu}\left (\sum_{j=0}^{m-1}V^{(\alpha,\omega')}_{x_1,j}V^\omega_{x_1,j})\log (V^{(\alpha,\omega')}_{x_1,j}V^\omega_{x_1,j})\right )=\sum_{i=0}^{m-1}p_iT'_i(1-)-\alpha>0;
$$
consequently, Theorem~\ref{thmBK} yields the desided conclusion.

\subsection{Proof of Corollary~\ref{DGF}} Let $\varphi :\; h\in \R_+\mapsto \log_m\sum_{i=0}^{m-1}\E(N_i)^{h}$.

We begin with the proof of  \eqref{vp}. The  upper bound for the box dimension of $\pi(K)$ can be obtained as a consequence of our approach to the multifractal analysis, or by using Falconer's argument in \cite{Fal}   (see also \cite{DekGri}). To see it, notice that  at a given generation $n$ of the construction, $\pi(K_n)$ is covered by at most $\sum_{|u|=n} (\#\{v\in\Sigma_n: Q(u,v)>0\})^h$, for all $0\le h\le 1$, which yields that the expectation of this number is at most $(\inf_{0\le h\le 1} m^{\varphi(h)})^n$.  Applying Lemma \ref{Borel}, we obtain that $\overline{\dim}_B(\pi(K))\leq \inf_{0\le h\le 1} \varphi(h)$.
  Thus it remains to  derive a lower bound for the Hausdorff dimension  of $\pi(K)$.

Let $h_0$ be a  point at which $\inf_{0\le h\le 1} \varphi(h)$ is attained. Due to the convexity and the analyticity of $\varphi$, such a point is not unique if and only if $\E(N_i)=1$ when $\E(N_i)>0$.  Let us consider the Mandelbrot measure associated with the following weights:
$$
W'_{i,j}=p'_i V'_{i,j}\quad\text{with } V'_{i,j}= \begin{cases}
\displaystyle \frac{\mathbf{1}_{\{W_{i,j}>0\}}}{\E(N_i)}&\text{  if $\E(N_i)>0$}\\
0&\text{otherwise}
\end{cases},
$$
where $$
p'=(p'_i)_{0\le i\le m-1}=\left (\frac{\E(N_i)^{h_0}}{\sum_{i=0}^{m-1}\E(N_i)^{h_0}}\right )_{0\le i\le m-1}.
$$

Let $\mu'$ be the associated Mandelbrot measure and $\nu'$ the Bernoulli product associated with~$p'$. We have
$$
\dim(\nu')=-\frac{ h_0 \sum_{i=0}^{m-1} \E(N_i)^{h_0}\log_m(\E(N_i))}{\sum_{i=0}^{m-1}\E(N_i)^{h_0}}+\log_m\left (\sum_{i=0}^{m-1}\E(N_i)^{h_0}\right )=\varphi(h_0)-h_0\varphi'(h_0)
$$
and
\begin{equation}
\label{e-ef}
\sum_{i=0}^{m-1}p'_iT'_{V'_i}(1)=  \frac{\sum_{i=0}^{m-1}\E(N_i)^{h_0} \log_m (\E(N_i))}{\sum_{i=0}^{m-1}\E(N_i)^{h_0}}=\varphi'(h_0).
\end{equation}

 Next we show that $\dim_H \pi(K) \geq \dim_H\pi(\mu')\geq \varphi(h_0)$, by considering the scenarios  $h_0=1$, $h_0\in (0,1)$ and $h_0=0$, separately.  First
 suppose that $h_0=1$. Then $\mu'$ is the so-called branching measure on $K$, and we see that
 $$\dim(\nu')+\sum_{i=0}^{m-1}p'_iT'_{V'_i}(1)= \varphi(1)=\log(\E(N))/\log(m)>0,$$
  hence $\mu'$ is non-degenerate with positive probability (a fact that can also be directly seen from $T_{W'}'(1)$). Moreover, since on $[0,1]$  $\varphi$ takes the minimum at $h=1$,  by smoothness of $\varphi$ we must have $\varphi'(1)\le 0$, consequently $\sum_{i=0}^{m-1}p'_iT'_{V'_i}(1)\le 0$, and thus by Theorem \ref{DIM},
  $\dim (\pi_*\mu') =\dim(\mu')$ and $\dim_H \pi(K)=\dim_H (K)=\varphi(1)$ when $K\neq\emptyset$.

Next suppose that $0<h_0<1$. We have $\varphi'(h_0)=0$, hence $\sum_{i=0}^{m-1}p'_iT'_{V'_i}(1)=0$ and thus by Theorem \ref{DIM},
$$\dim (\pi_*\mu') =\dim(\mu')=\dim(\nu')=\varphi(h_0),$$ yielding $\dim_H\pi(K)\ge \varphi(h_0)$ when $K\neq\emptyset$.

Finally suppose that $h_0=0$. Then $\varphi'(0)\ge 0$, so  $\sum_{i=0}^{m-1}p'_iT'_{V'_i}(1)\geq 0$ and thus by Theorem \ref{DIM},  $\dim(\pi_*\mu')=\dim(\nu')=\varphi(0)\le \dim(\mu')$ when $\mu'\neq 0$, and consequently, $\dim_H \pi(K) \ge \varphi(0)$ when $K\neq\emptyset$.

So far we have proved  \eqref{vp}.  Below we discuss the uniqueness problem regarding  the last variational relation in \eqref{vp}.

  Notice that the  Mandelbrot measure $\mu'$ considered above has a dimension equal to $\dim_H K$ if and only if $T_{W'}'(1)=-T_{W'}(0)=\log_m(\E(N))$, that is $T_{W'}$ is linear. If $h_0=1$, $\mu'$ is the branching measure. If $h_0<1$, since
  $$T_{W'}(q)=q \log_m\Big(\sum_{i=0}^{m-1} \E(N_i)^{h_0}\Big)-\log_m\sum_{i=0}^{m-1}\mathbf{1}_{\{p_i>0\}}\E(N_i)^{1+q(h_0-1)}$$
and the second derivative of $T_{W'}$ vanishes,     we get that $\E(N_i)=1$ for each $i$ such that $p_i>0$. Once again $\mu'$ is the branching measure.

For the uniqueness problem, the case when $\dim_H K=\dim_H \pi(K)$ is clear from the above discussion,  since the same argument in fact shows that a Mandelbrot measure   supported on $K$ whose dimension equals that of $K$ must be the branching measure. Thus we can suppose that $\dim_H K>\dim_H \pi(K)$.

Suppose that the maximum in \eqref{vp} is attained at a Mandelbrot measure $\mu''$ defined simultaneously with $\mu'$ and supported on $K$ conditionally on non-vanishing. Then it is easily seen that $\mu''$ is generated by a random vector $W''$ such that $W''_{i,j}>0$ only if $W_{i,j}>0$, and one can associate with $W''$ the probability vector $(p''_i=\sum_{j=0}^{m-1}\E(W''_{i,j}))_{0\le i\le m-1}$ and the vectors $V''_i=(W''_{i,j}/p''_i)_{0\le j\le m-1}$ if $p''_i>0$ and $0$ otherwise. Moreover, $p'_i>0$ implies $p''_i>0$ for otherwise the formula $\inf_{0\le h\le 1} \log_m\sum_{i=0}^{m-1}\E(N_i)^h$ for the Hausdorff dimension of $\pi(K)$ would  give a strictly smaller dimension.
By Theorem~\ref{DIM} we have
\begin{equation}\label{dimmu''}
\dim (\pi_*\mu'')= \min \Big (\dim (\nu''), \;  \dim (\nu'')+\sum_{i=0}^{m-1} p''_iT_{V''_i}'(1-)\Big ).
\end{equation}
Now, let us observe that $\sum_{i=0}^{m-1} p''_iT_{V''_i}'(1-)$ is always smaller than or equal to $$\sum_{i=0}^{m-1} p''_iT_{V'_i}'(1-)=\sum_{i=0}^{m-1} p''_i\log_m\E(N_i)
.$$
 This is due to the fact that $T_{V''_i}$ is concave, equal to 0 at~1, and
 \begin{equation}
 \label{e-ee}
 T_{V''_i}(0)=-\log_m\E(\sum_{j=0}^{m-1}\mathbf{1}_{\{W''_{i,j}>0\}})\ge -\log_m\E(N_i),
  \end{equation}
  implying that $T_{V''_i}'(1-)\le -T_{V''_i}(0) \le \log_m\E(N_i).$

Consequently, in order to optimize $ \dim (\nu'')+\sum_{i=0}^{m-1} p''_iT_{V''_i}'(1-)$, $\mu''$ must satisfy the condition that $T'_{V''_i}(1-)= T'_{V'_i}(1-)=\log_m\E(N_i)$. On the other hand, by concavity of $T_{V''_i}$ on $[0,1]$, we have  that $T_{V''_i}(0)\le -T'_{V''_i}(1-)$. Finally, since by \eqref{e-ee},  $T_{V''_i}(0)\ge -\log_m\E(N_i)=-T'_{V''_i}(1-)$, we get $T_{V''_i}(0)=-\log_m\E(N_i)=-T'_{V''_i}(1-)$, hence $T_{V''_i}$ is linear on $[0,1]$. This means that like for $V'_i$, the coordinates of the vector $V_i''$ equal either 0 or $1/\E(N_i)$. Since, moreover, we have $W''_{i,j}=0$ as soon as $W'_{i,j}=0$, we get $V''_i=V'_i$ almost surely. On the other hand, a simple study using Lagrange multipliers shows that $\dim (\nu'')+\sum_{i=0}^{m-1} p''_i\log_m\E(N_i)$ is optimal for $p''=p'$, the maximum being unique. In other words, the maximum over $\mu''$ of $\dim (\nu'')+\sum_{i=0}^{m-1} p''_iT_{V''_i}'(1-)$ is reached uniquely at~$\mu'$.

Now, suppose first that $\varphi'(0)\le 0$, i.e. the infimum of $\varphi$ over $[0,1]$ is reached at a unique $h_0\in (0,1 ]$, or at $h_0=0$ with $\varphi'(0)=0$. In both cases,  we have $\varphi'(h_0)\leq 0$, and  our study of $\mu'$  (cf.  \eqref{e-ef}) shows that  $\sum_{i=0}^{m-1} p'_i T_{V_i'}'(1^{-})= \sum_{i=0}^{m-1} p'_i\log_m\E(N_i)=\varphi'(h_0)\leq 0$, showing that $\dim (\pi_*\mu')=\dim (\mu')$ by Theorem \ref{DIM}. Consequently,  by the arguments in the last paragraph,  for any Mandelbrot measure $\mu''$ supported on $K$, we have

$$
\dim (\nu'')+\sum_{i=0}^{m-1} p''_iT_{V''_i}'(1-)\le \dim (\nu')+\sum_{i=0}^{m-1} p'_iT_{V'_i}'(1-)=\dim(\mu')= \dim (\pi_*\mu'),
$$
{where the first equality holds  if and only if $\mu''=\mu'$. Then, the relation \eqref{dimmu''} yields $\mu'$ as the unique Mandelbrot measure such that $\dim (\pi_*\mu')$ is maximal.

Next suppose that $\varphi'(0)>0$. Fix $\lambda>1$ and $U_\lambda$  a random variable independent of $V'$ and taking value $\lambda >1$ with probability $\lambda^{-1}$ and $0$ with probability $1-\lambda^{-1}$. Take $p''=p'$ and replace $V'$ by $V''=(V''_0,V''_1,V''_2,\ldots,V''_{m-1})$ with $V''_i=U_\lambda\cdot V'_i$. This yields a  Mandelbrot measure $\mu''$ different from $\mu'$, with the same expectation $\nu'$ and  $\sum_{i=0}^{m-1}p''_iT'_{V''_i}(1)=\sum_{i=0}^{m-1}p'_iT'_{V'_i}(1)-\log_m(\lambda)>0$ if $\lambda$ is close enough to 1. Consequently, $\dim (\pi_*\mu'')=\dim(\nu')=\dim(\pi_*\mu')$, and there is no uniqueness in this case.

\section{Proof of Theorem~\ref{MA}: Differentiability properties of the function $\tau$} \label{study of tau}

\subsection*{{\bf Differentiability  over $(0,1-]$}}\label{difftau} $\ $


Notice that the differentiability of $\tau$ over $(0,1-]$  automatically holds if $\tau\equiv T$ over $(0,1]$, and  that this holds in particular if $T_i$ is linear and $\E(N_i)=1$ for all $0\le i\le m-1$ such that $\E(N_i)>0$, i.e.  $T_i\equiv 0$ so that $T=\tau_\nu=\tau$ (the study achieved below shows that this  is also a necessary condition, which is equivalent to have $\E(N_i)=1$ and $V_{i,j}=\mathbf{1}_{\{W_{i,j}>0\}}$ for all $0\le j\le m-1$ almost surely). Moreover, still in this case, since we have excluded the case that $N_i=1$ for all $0\le i\le m-1$ such that $\E(N_i)>0$, by Theorem~\ref{ABS}(2),  $\pi_*{\mu}$ and $\nu$ are mutually singular, and thus $\pi_*\mu\neq\nu$ almost surely.

Now suppose that $\tau\not\equiv T$ over $(0,1]$. For $0<q\le s\le 1$ set
$$
G(q,s)= \sum_{i=0}^{m-1} p_i^q m^{-qT_i(s)/s}
$$
and
$$
g(q,s)=s^2(-q\log(m))^{-1}\frac{\partial G}{\partial s}(q,s)=\sum_{i=1}^{m-1}p_i^q m^{-qT_i(s))/s}T_i^*(T_i'(s)).
$$

Let $q\in (0,1]$. At first suppose that the infimum defining $\tau(q)$, i.e. the infimum of $G(q,\cdot)$, is reached at $s\in (q,1)$ (hence $q<1$). We claim that $s$ is unique and for all $q'$ in an open neighborhood of $q$ there exists a unique $s(q')\in (q',1)$ such that $\tau(q')=- \log_m G(q,s(q'))$.  To show this claim, notice  that at any $s_0\in (q,1)$ at which the infimum defining $\tau(q)$ is reached we have $g(q,s_0)=0$. Moreover, for all $s\in [q,1]$ we have
$$
\frac{\partial g}{\partial s}(q,s)= \sum_{i=1}^{m-1}p_i^q m^{-qT_i(s))/s}(-q\log(m)s^{-2}(T_i^*(T_i'(s)))^2+sT_i''(s))\le 0.
$$
Suppose that $T_i''(s)=0$ for some $i$. It means that
$$
\left (\mathbb{E}\sum_{j=0}^{m-1}V_{i,j}^s(\log(V_{i,j}))^2\right )\left (\mathbb{E}\sum_{j=0}^{m-1}V_{i,j}^s\right)= \left (\mathbb{E}\sum_{j=0}^{m-1}V_{i,j}^s(\log(V_{i,j}))\right )^2.
$$
It follows that by the Cauchy-Schwarz inequality, there exists a constant $c$ such that almost surely either $V_{i,j}=0$ or $V_{i,j}=c$, hence $c=1/\E(N_i)$. In this case we have $T_i^*(T_i'(s))=\log(\E(N_i))$.  Consequently, for $\frac{\partial g}{\partial s}(q,s)$ to be equal to 0 we need to have $\E(N_i)=1$  and $V_{i,j}=\mathbf{1}_{\{W_{i,j}>0\}}$ for all $0\le i,j\le m-1$ such that $p_i>0$, a situation that we have discarded by assuming that $\tau\neq T$ (notice that this property is equivalent to requiring that $T_i\equiv 0$ for all $0\le i\le m-1$ such that $p_i>0$). Thus $\frac{\partial g}{\partial s}(q,s)< 0$, hence $g(q,s)$ can vanish only at one point of $(q,1)$, that we denote by $s(q)$.  Then,  because $\frac{\partial g}{\partial s}(q,s(q))< 0$,  the implicit function theorem implies our claim, as well as the analyticity of $ s(\cdot)$ and   $\tau$ on any maximal interval of points $q$ such that $s(q)\in (q,1)$. In addition, $
s'(q)=-\frac{\frac{\partial g}{\partial q}(q,s(q))}{{\frac{\partial g}{\partial s}(q,s(q))}}
$. We also notice that the study of {$s\mapsto g(q,s)$ shows that $s\mapsto\frac{\partial G}{\partial s}(q,s)$} is negative on the left hand side of $s(q)$ and positive on the right hand side, so the  infimum of $G(q,\cdot)$  { over $[q,1]$} can be reached neither at $q$ nor at $1$.

Now suppose that  the infimum of $G(q,\cdot)$ is reached at $s_0\in\{q,1\}$. Suppose that  this infimum is reached  at another point of $[q,1]$ as well (this can hold only if $q<1$). Then, let $s_1\in (q,1)$ at which $G(q,\cdot)$ reaches a local maximum, hence $g(q,\cdot)$ vanishes. Our previous analysis of the sign of $g(q,\cdot)$, which is the opposite of the sign of $\frac{\partial G}{\partial s}(q,\cdot)$, shows that $ \frac{\partial G}{\partial s}(q,\cdot)$ is negative on the left of $s_1$, which is a contradiction. Thus the infimum of $G(q,\cdot)$ at $s_0$ is strict. We again denote this point $s_0$ by $s(q)$.

We notice that the  argument  in the above paragraph also shows that if $q$ is a point of $(0,1)$ at which $\tau_\nu$ and $T$ coincide, i.e. $G(q,1)=G(q,q)$, then  $\tau(q)$ cannot be attained at $q$ or $1$. This entails the fact that $\tau=T=\tau_\nu$  only if $T_i\equiv 0$ when $p_i>0$.

Next we prove that both  $\tau$ and $s(\cdot)$ are continuous over $(0,1]$. Suppose that $q\in (0,1]$. Let $(q_n)_{n\ge 1}$ be a sequence of points in $(0,1]$ such that $q_n\to q$. Without loss of generality, we can assume that $s(q_n)$ converges as well,  to a number, say $s_q$, which necessarily belongs to $[q,1]$ since $s(q_n)\in[q_n,1]$. It follows by continuity of $G$ that $G(q_n,s(q_n))\to G(q,s_q)$. Suppose that $s_q\neq s(q)$. Then, $G(q,s(q))<G(q,s_q)$, hence there exist $n_0>1$ and  $\epsilon>0$ such that for all $n\ge n_0$, for all $s\in[q_n,1]$ we have $$G(q_n,s)\geq G(q_n, s(q_n))>G(q,s(q))+\epsilon.$$
 However, there exists a sequence $(s_n)_{n\ge 1}$  such that $s_n\in[q_n,1]$ for all $n\ge n_0$ and $(q_n,s_n)\to (q,s(q))$. By continuity of $G$ over $[0,1]$, we have $G(q_n,s_n)\to G(q,s(q))$, but  $G(q_n,s_n)>G(q,s(q))+\epsilon$, which gives a contradiction. Consequently, we obtained the desired continuity property of $s(\cdot)$, and that of $\tau=-\log_mG(\cdot,s(\cdot))$.

Let us denote by $\mathcal I$ the set of the connected components  of $\{q\in(0,1): s(q)\in (q,1)\}$.

Let $E=(0,1]\setminus\bigcup_{I\in\mathcal I} I$. Let  $q_0\in E$. If $q_0$ is an interior point of $E$, then by continuity of $s$, we must have either $s(q)=q$ or $s(q)=1$ on the maximal interval $I_{q_0}$ containing $q_0$ and contained in $E$; as a consequence,   both $s(\cdot)$ and $\tau$ are analytic on  the interior of $I_{q_0}$.

Suppose that $q_0\in\partial E$ and $q_0<1$.  Notice that since $q_0$ is an accumulating point of $\bigcup_{I\in\mathcal I}  I$, by continuity of  $\frac{\partial G}{\partial s} $ and $s(\cdot)$, we have either $\frac{\partial G}{\partial s}(q_0,q_0)=0$ if $s(q_0)=q_0$ or $\frac{\partial G}{\partial s}(q_0,1)=0$ if $s(q_0)=1$.

Up to symmetry between the left and the right hand sides of $q_0$, there are essentially three situations.  There exists $\eta>0$ such that either $s(q)=q$ over $[q_0-\eta,q_0)$ and $s(q)\in (q,1)$ over $(q_0,q_0+\eta]$, $s(q)=1$ over $[q_0-\eta,q_0)$ and $s(q)\in (q,1)$ over $(q_0,q_0+\eta]$, or $s(q)\in (q,1)$ both over $[q_0-\eta,q_0)$ and $(q_0,q_0+\eta]$. It means that $q_0$ cannot be an accumulating point of boundary points of $E$. Indeed, suppose that on the contrary $q_0$ is such a point. Then $s(q_0)\in \{q_0,1\}$.  First assume that  $s(q_0)=q_0$. By the remark in  the last paragraph,  $\frac{\partial G}{\partial s}(q,q)$ should have    infinitely many zeros accumulating at $q_0$, which would imply that  $\frac{\partial G}{\partial s}(q,q)=0$ for all $q\in (0,1)$ by analyticity of $G$; but this does not hold, for otherwise we have $\tau=T$, a case that we discarded. Indeed if $\tau\neq T$, there exists $q_0\in(0,1)$ such that $s(q_0)\in (q_0,1)$. Then our previous study of $g(q_0,\cdot)$ shows that $\frac{\partial G}{\partial s}(q_0,q_0)=-(\log)g(q_0,q_0)/q_0<0$ since $g(q_0,\cdot)$ is strictly decreasing and $g(q_0,s(q_0))=0$. Next assume  $s(q_0)=1$. Again by the remark in the last paragraph we should have $\frac{\partial G}{\partial s}(q,1)=0$ and thus $g(q,1)=0$ for all $q\in (0,1)$, and  it follows that $g(q, s(q))>0$ whenever $s(q)\neq 1$, leading to a contradiction.

Finally suppose that $q_0=1$. The same approach as above shows that there exists $\eta>0$ such that either $s(q)=1$  or $s(q)\in (q,1)$ over $[1-\eta,1)$. Also, we notice that $0$ cannot be  an accumulating point of $\partial E$ since we assumed that the $T_i$ are finite and analytic in a neighborhood of $0$.

The previous arguments imply the following intermediate fact.

\begin{pro}\label{partition} The functions $\tau$ and $s(\cdot)$ are continuous over $(0,1]$. There exists a set $S$, finite or empty, such that for each connected component $I$ of $(0,1]\setminus S$, the functions $\tau$ and $s(\cdot)$ restricted to $I$ are analytic, and $I$ is a maximal interval over which either $s(q)=q$, $s(q)\in (q,1)$ or $s(q)=1$.
\end{pro}
It remains to prove the differentiability of $\tau$ at each $q\in S$. Let $q_0\in S$. If $q_0=1$, then there exists $\eta>0$ such that $s(q)\in (q,1)$ over $[1-\eta,1)$. The formula
\begin{equation}\label{tau'q}
\tau'(q)=-\frac{\frac{\partial G}{\partial q}(q,s(q))}{\log(m) G(q,s(q))}
\end{equation}
implies that $\tau'(q)$ has a limit at $1-$, hence by the mean value theorem $\tau$ is left differentiable at $1$.

Suppose that $q_0<1$. If $s(q)\in(q,1)$ for all $q$ in  $[q_0-\eta,q_0+\eta]\setminus \{q_0\}$ for some $\eta>0$, then formula \eqref{tau'q} and the continuity  of $s(\cdot)$ combined with the mean value theorem yield the fact that $\tau$ is $C^1$ at $q_0$.  If $s(q)=q$ on $[q_0-\eta,q_0)$ and $s(q)\in (q,1)$ on $(q_0,q_0+\eta]$, we first notice that $s(q)/q$ tends to $1$ as $q\to q_0+$ by continuity of $s(\cdot)$.  It is then almost direct to see that $\tau'(q)$ given by \eqref{tau'q} converges to $T'(q_0)$ as $q\to q_0+$. Indeed, one has
\begin{eqnarray}
\label{tau'q1}
\frac{\partial G}{\partial q}(q,s(q))&=& \sum_{i=0}^{m-1} p_i^qm^{-T_i(s(q)) q/s(q)} (\log (p_i)-\log(m) T_i(s(q))/s(q))\\
\label{tau'q2}&=&\sum_{i=0}^{m-1} p_i^qm^{-T_i(s(q)) q/s(q)} (\log (p_i)-\log(m) T_i'(s(q)),
\end{eqnarray}
due to the equality $\frac{\partial G}{\partial s}(q,s(q))=0$. Then, letting $q$ tend to $q_0+$ and using the fact that  $s(q)/q$ tends to 1,  we get $\lim_{q\to q_0+} \tau'(q)=T'(q_0)$. On the other hand, $\tau=T$ over $[q_0-\eta,q_0)$, hence $\tau$ is $C^1$ at $q_0$. The other cases can be treated similarly.

\subsection*{\bf {Concavity of $\tau$}} We will show later that the differentiability of $\tau$ over $[0,1]$ combined with other arguments yield the equality of $\tau$ with the $L^q$-spectrum of $\pi_*\mu$, conditionally on $\mu\neq 0$. Consequently, $\tau$ is concave, and automatically differentiable at the right hand side of $0$ if it is right continuous at 0.

\subsection*{{\bf Continuity and differentiability at 0}} Due to the previous discussion, it is enough to prove the continuity at 0. However, we will examine the value of $\tau'(0+)$. We  distinguish two cases.

At first suppose that $h(q)=q/s(q)$ does not tend to 0 as $q$ tends to 0. It follows that $s(q)$ tends to 0. Suppose that for some sequence $(q_n)_{n\ge 0}$ tending to 0 we have $h(q_n)\to h_*\in (0,1]$. The study achieved above gives  $g(q_n,s(q_n))=0$ if $q_n<s(q_n)<1$ and $g(q_n,s(q_n))\le 0$ if $s(q_n)=q_n$. This  implies that
$$\sum_{i=0}^{m-1} \E(N_i)^{h_*}\log_{ m}(\E(N_i))=\lim_{n\to +\infty} g(q_n,s(q_n))$$ vanishes if $h_*<1$ and is non positive if $h_*=1$. By convexity of the mapping $h\in [0,1]\mapsto \log_m\sum_{i=0}^{m-1} \E(N_i)^{h}$, we conclude that in any case, $$- \log_m\sum_{i=0}^{m-1} \E(N_i)^{h_*}=-\inf_{0\le h\le 1} \log_m\sum_{i=0}^{m-1} \E(N_i)^{h},$$ i.e. $h_*$ is the  point at which the minimum in~\eqref{dimproj} is attained.  Moreover we have $\lim_{n\to\infty} \tau(q_n)=- \log_m\sum_{i=0}^{m-1} \E(N_i)^{h_*}=\tau(0)$.
It follows that $\tau$ is right continuous at~$0$.

Now suppose that $h(q)=q/s(q)$ tends to 0 as $q$ tends to 0. We have $q<s(q)\le 1$ for $q$ small enough. From this it follows that $g(q,s(q))\ge 0$. Consequently, since $\sum_{i=0}^{m-1} \mathbf{1}_{\{p_i>0\}} \log(\E(N_i))=\lim_{q\to 0+} g(q,s(q))$ (because $h(q)$ tends to 0), this number is non negative. This implies that $\log_m \sum_{i=0}^{m-1} \mathbf{1}_{\{p_i>0\}}=\inf_{0\le h\le 1} \log_m\sum_{i=0}^{m-1} \E(N_i)^{h} $. On the other hand  $\lim_{q\to 0+} \tau(q)=-\log_m \sum_{i=0}^{m-1} \mathbf{1}_{\{p_i>0\}}$, hence $\tau$ is right continuous at 0, and $\tau(0)=\tau_\nu(0)$. In this case we set $h_*=0$.

In all the cases, we set
\begin{equation}
\label{e-p'}
p'_i=\left (\frac{\E(N_i)^{h_*}}{\sum_{i'=0}^{m-1}\E(N_{i'})^{h_*}}\right )_{0\le i\leq  m-1},
\end{equation}
with the convention $0^0=0$, and we denote by $\nu'$ the associated Bernoulli product.

\subsection*{{\bf The value of $\tau'(0+)$}} Now we use Proposition~\ref{partition} to determine the value of $\tau'(0+)$ and examine more precisely the behavior of $s(q)$ at $0+$. This will be used to prove the validity of the multifractal formalism for $\pi_*\mu$ at $\tau'(0+)$. Our observation is the following:
{\begin{pro}\label{tau'0} Let $p_i'$ be defined as in \eqref{e-p'}. One of the three following situations occurs:
\begin{itemize}
\item [(i)] $\tau=T$ near $0+$ and $\tau'(0+)=T'(0)$.
\item [(ii)] $\tau=\tau_\nu$ near $0+$ and $\tau'(0+)=\tau_\nu'(0)$. Moreover, $\sum_{i=1}^{m-1}p'_i T_i^*(T_i'(1))\ge 0$.
\item [(iii)] $\tau>\max(T,\tau_\nu)$  near $0+$, and there exists $s_0\in [0,1]$ such that
$$
\tau'(0+)=-\sum_{i=0}^{m-1} p'_i(\log_m(p_i)-T'_i(s_0)).
$$
Moreover, $\sum_{i=0}^{m-1} p'_iT_i^*(T'_i(s_0))=0$.
\end{itemize}
\end{pro}
}
\begin{proof} We treat  the  three cases considered in the statement separately.

{\bf Case 1: $\tau=T$ near $0+$.} In this case, we have $h_*=1$ and $\tau'(0+)=T'(0)$.

{{\bf Case 2: $\tau=\tau_\nu$ near $0+$.} We have $h_*=0$ and $\tau'(0+)=\tau_\nu'(0)$. Moreover,  for all $q>0$ close enough to 0 we have $s(q)=1$, which implies that  $g(q,s(q))=g(q,1)=\sum_{i=1}^{m-1}p_i^q T_i^*(T_i'(1))\ge 0$. Consequently, letting $q$ tend to 0 we get $\sum_{i=1}^{m-1}p'_i T_i^*(T_i'(1))\ge 0$.
}

{{\bf Case 3:  $\tau>\max(T,\tau_\nu)$ near $0+$.}
}

{Assume at first that $h_*\in (0,1]$. Letting $q$ tend to $0+$ in the equality $g(q,s(q))=0$ we obtain  $\sum_{i=0}^{m-1} p'_iT_i^*(T'_i(0))=0$. Then, applying \eqref{tau'q} and \eqref{tau'q2}  at $q$ close enough to $0+$ and letting $q$ tend to $0$ we obtain $\tau'(0+)=-\sum_{i=0}^{m-1} p'_i(\log_m(p_i)-T'_i(0))$; we then set $s_0=0$.
}

{Next assume that  $h_*=0$. From the discussion of the continuity of $\tau$ at 0 we deduce that $\tau(0)=\tau_\nu(0)$. Next, consider a sequence $(q_n)_{n\ge 1}$ converging to $0+$ such that $s(q_n)$ (which belongs to $(q_n,1)$) tends to $s_0\in [0,1]$. From the equality $g(q_n,s(q_n))=0$ we deduce that $\sum_{i=0}^{m-1} p'_iT_i^*(T'_i(s_0))=0$ by letting $n$ tend to $\infty$. Moreover, using   \eqref{tau'q} and \eqref{tau'q2} with $q_n$ and letting $n$ tend to $\infty$ yields $\tau'(0+)=-\sum_{i=0}^{m-1} p'_i(\log_m(p_i)-T'_i(s_0))$.}
\end{proof}
\subsection*{{\bf  Differentiability at 1}} Due to \eqref{tau'q}, if $q<s(q)<1$ in a left neighborhood of $1$, by \eqref{tau'q} we have $\tau'(1-)=T'(1)$. {This, together with the facts that $\tau\ge \max(\tau_\nu,T)$ over $[0,1]$ and $\tau(1)=T(1)=\tau_\nu(1)$ implies that $\tau_\nu'(1)\ge  T'(1)$.  Then, if the last inequality is strict, we have  $\tau_\nu>T$ hence $\tau=T$ on a right neighborhood of $1$, which yields the differentiability of $\tau$ at $1$.  If $\tau_\nu'(1)=T'(1)$, then $\min (\tau_\nu,T)$ must have a derivative equal to $T'(1)$ on the right of $1$, and we get the desired conclusion as well.}

If $s(q)=1$ in a left neighborhood  $1$, then there we have $\tau=\tau_\nu\ge T$, and  $\tau'(1-)=\tau_\nu'(1)$. {Then, similar argument as in the previous case (with the roles of $\tau_\nu$ and $T$ exchanged) yields the existence of $\tau'(1)$.}

The case $s(q)=q$ in a left neighborhood  $1$ is treated similarly.

In conclusion, we get
\begin{equation}\label{tau'1}
\tau'(1)=
\begin{cases}T'(1)&\text{if }T'(1)\le \tau_\nu'(1)\\
\tau'(1)=\tau_\nu'(1)&\text{otherwise }
\end{cases}.
\end{equation}

\subsection*{{\bf Differentiability and concavity over $(1,q_c)$}} Recall that $q_c$ is defined in \eqref{assumMA}. The definition of $\tau$ clearly implies its concavity and differentiability at points at which the graphs of $\tau_\nu$ and $T$ do not cross transversally. Due to the analyticity of $\tau_\nu$ and $T$, there are at most finitely many such points in a given bounded interval.

\section{Proof of Theorem~\ref{MA}: Lower bound for the $L^q$-spectrum}
\begin{pro}\label{tau}With probability 1, conditionally on $\pi_*\mu\neq 0$,
\begin{enumerate}
\item   for all  $q\ge 1$ we have the following properties:

\begin{itemize}
\item [(i)] $\tau_{\pi_*\mu}(q) \le \tau_\mu(q)$;

\item [(ii)] if $T(q)>0$, then $\tau_{\pi_*\mu}(q)\ge  \min (\tau_\nu(q),T(q))$;

\item [(iii)] if $T^*(T'(q))\ge 0$ then  $T(q)>0$. If, in addition,  $\min (\tau_\nu(q),T(q))=T(q)$,  then $\tau_{\pi_*\mu}(q)=T(q)$.
\end{itemize}

\item For all  $0< q\le 1$,  we have
$
\tau_{\pi_*\mu}(q)\ge \tau(q).
$
\end{enumerate}
\end{pro}

Since,  as a $L^q$-spectrum, the function $\tau_{\pi_*\mu}$ is  continuous over $(0,\infty)$ and $\tau$, $\tau_\nu$ and $T$ are continuous, we only need to get the desired inequalities for each $q>0$.

\begin{proof}
(1) (i) The fact that $\tau_{\pi_*\mu}(q) \le \tau_\mu(q)$ for $q\ge 1$ is general and comes from the super-additivity of $x\mapsto x^q$ over $\R_+$ applied to $\Big (\pi_*\mu([u])= \sum_{v\in\Sigma_n}\mu([u,v])\Big )^q$.

(ii) The almost sure inequality $\tau_{\pi_*\mu}(q)\ge \min(\tau_\nu(q),T(q))$ for a given  $q\ge 1$ such that
$T(q)>0$ is a direct consequence of Corollary~\ref{cor6.11}.

(iii) Let $q\ge 1$ be such that $T^*(T'(q))\ge 0$ and suppose that $T(q)\le 0$.  Recall that $T$ is concave, so its derivative is non increasing. Also, $T(1)=0$ and $T'(1)>0$. This implies that $T'$ is negative at some point of $(1,q)$, otherwise $T$ could not take non positive values over $(1,q]$. Since $T'$ is non increasing, it follows that $T$ has a unique zero $q_0$ over $(1,q]$ at which $T'(q_0)<0$. This implies that $T^*(T'(q_0))=q_0T'(q_0)<0$. Since  $T^*(T')$ is non increasing on $\R_+$ (its derivative is $q\mapsto qT''(q)$),  we get $T^*(T'(q))\le T^*(T'(q_0))<0$, which is a contradiction. So $T(q)>0$.

Now recall that by Theorem~\ref{AB}, we have $T(q)=\tau_\mu(q)$ as soon as $T^*(T'(q))\ge 0$. Thus, if $\min (\tau_\nu(q),T(q))=T(q)$, the equality $\tau_\mu(q)=T$ comes from (i) and (ii).


\medskip

(2)  For $0< q\le 1$ and $q\le s\le 1$, using Jensen's inequality, for each $n\ge 1$ we get
\begin{eqnarray*}
\E\Big (\sum_{u\in\Sigma_n}\nu([u])^q X(u)^q\Big )
&=& \E\Big (\sum_{u\in\Sigma_n}\nu([u])^q X(u)^{s \cdot q/s}\Big )\\
&\le& \sum_{u\in\Sigma_n}\nu([u])^q \E(X(u)^s)^{q/s}.
\end{eqnarray*}
Then, using the definition of $X(u)$, the fact that $\E(Y^s) \le \E(Y)^s=1$, and the branching property, we obtain
\begin{eqnarray*}
 \sum_{u\in\Sigma_n}\nu([u])^q \E(X(u)^s)^{q/s}&= &\sum_{u\in\Sigma_n}\nu([u])^q \E\Big (\sum_{v\in\Sigma_n} \Big (\frac{\mu([u,v])}{\nu([u])}\Big )^s\Big )^{q/s}\\
&= & \sum_{u\in\Sigma_n}\nu([u])^q \E\Big (\sum_{v\in\Sigma_n} Y(u,v)^s \prod_{k=1}^nV_{u_k,v_k}(u_{|k-1},v_{|k-1})^s\Big )^{q/s}\\
&=&\E(Y^s)^{q/s} \sum_{u\in\Sigma_n}\prod_{k=1}^n p_{u_k}m^{-T_{u_k}(s)q/s}\le \Big (\sum_{i=0}^{m-1} p_i^q m^{-qT_i(s)/s}\Big )^n.
\end{eqnarray*}
Since this holds for all $s\in [q,1]$, for each $n\ge 1$ we obtain
$$
\E\Big (\sum_{u\in\Sigma_n}\nu([u])^q X(u)^q\Big )\le \Big (\inf_{q\le s\le 1}\sum_{i=0}^{m-1} p_i^q m^{-qT_i(s)/s}\Big )^n.
$$
 Consequently, Lemma~\ref{Borel} yields $\tau_{\pi_*\mu}(q)\ge  \tau(q)$ almost surely.
\end{proof}

\section{Proof of Theorem~\ref{MA}: Upper bound for the $L^q$-spectrum and validity of the multifractal formalism}\label{sec-4.3}

Proposition~\ref{tau} yields the following lemma.

\begin{lem}With probability 1, conditionally on $\mu\neq 0$, we have $\tau_{\pi_*\mu}\ge \tau$ over $[0,\widetilde q_c)$.
\end{lem}
Consequently, due to the general inequality $\dim E(\pi_*\mu,\alpha)\le \tau_{\pi_*\mu}^*(\alpha)$,  valid for all $\alpha$,
 to prove the validity of the multifractal formalism at any $\alpha\in [\tau'(q^+),\tau'(q^-)]$ for some $q\in (0,\widetilde q_c)$ or at $\alpha=\tau'(0+)$ almost surely, as well as  the almost sure equality of $\tau_{\pi_*\mu}=\tau$ over $[0,\widetilde q_c)$, it is enough to show that, for each $q\in [0,\widetilde q_c)$, with probability 1, conditionally on if $\mu\neq 0$, we have $\dim E(\pi_*\mu,\alpha)\ge \tau^*(\alpha)$ for $\alpha\in [\tau'(q^+),\tau'(q^-)]$ if $q>0$ and $\alpha=\tau'(0+)$ if $q=0$.

Indeed, once this is done, we automatically have that almost surely, conditionally on $\mu\neq 0$, $\tau^*(\alpha)=\alpha q-\tau(q)\le \dim E(\pi_*\mu,\alpha)\le  \tau_{\pi_*\mu}^*(\alpha)\le \alpha q-\tau_{\pi_*\mu}(q)\le \alpha q-\tau(q)$. Moreover,   the information $\dim E(\pi_*\mu,\alpha)\ge \tau^*(\alpha)$ for $\alpha=\tau'(q)$, where $q$ describes a dense countable subset of values of $q$ is enough to get the equality $\tau=\tau_{\pi_*\mu}$ over $[0,\widetilde q_c)$. Also, the fact   $\tau_{\pi_*\mu}(q)=qT'(q_c-)$
 for $q\ge q_c$ when $\widetilde q_c=q_c<\infty$ follows from Proposition~\ref{linearization}.

Then, to get \eqref{convtaun} for $q\in (0,\widetilde q_c)$, we notice that if $\alpha\in\{\tau'(q^+),\tau'(q^-)\}$, for any $\epsilon>0$,  for $n$ large enough, one has $\#\{u\in\Sigma_n: \pi_*\mu([u])\ge m^{-n(\alpha+\epsilon)}\}\ge m^{n(\tau^*(\alpha)-\epsilon)}$, for otherwise a simple covering argument would give $\dim E(\pi_*\mu,\alpha)<\tau^*(\alpha)$. This implies
$$
\sum_{|u|=n} \mathbf{1}_{\{\pi_*\mu([u])>0\}}\pi_*\mu([u])^q\ge m^{n(\tau^*(\alpha)-\epsilon)}m^{-nq(\alpha+\epsilon)}\ge m^{-n(\tau(q) +(q+1)\epsilon)}.
$$
Since $\epsilon$ is arbitrary, this yields $\limsup_{n\to\infty} -\frac{1}{n}\log_m\sum_{|u|=n} \mathbf{1}_{\{\pi_*\mu([u])>0\}}\pi_*\mu([u])^q\le \tau(q)$. Moreover, we already know (by Proposition~\ref{tau}) that
$$
\tau_{\pi_*\mu}(q)=\liminf_{n\to\infty} -\frac{1}{n}\log_m\sum_{|u|=n}\mathbf{1}_{\{\pi_*\mu([u])>0\}} \pi_*\mu([u])^q\ge \tau(q).
$$
The case $q=0$ just comes from the fact that $\dim_H K=\dim_B K$.

\begin{rem}\label{taupsitaunu} {\rm To follow the different cases distinguished below, it is useful to have the following properties in mind.
\begin{itemize}
\item[(1)] If $\alpha=\tau'(1)$, our study of the exact dimensionality of $\pi_*\mu$ and \eqref{tau'1} show that $\dim E(\pi_*\mu,\alpha)=\alpha=\tau^*(\alpha)$ almost surely conditionally on $\mu\neq 0$.

\item[(2)] The study of the differentiability of $\tau$ achieved in  Section~\ref{study of tau} shows that if $q\in (0,1)$ then either $\tau(q)=T(q)$ and  $\tau'(q)=T'(q)$ or $\tau(q)=\tau_\nu(q)$ and $\tau'(q)=\tau_\nu'(q)$.

\item[(3)] Simple considerations about the concave function $\min (\tau_\nu,T)$ show that at $q\in (1,q_c)$, if $\tau(q)=T(q)<\tau_\nu(q)$ then $\tau'(q)=T'(q)$, if $q\in (1,\widetilde q_c)$ and $\tau(q)=\tau_\nu(q)<T(q)$ then $\tau'(q)=\tau_\nu'(q)$, and if $q\in (1,q_c)$ and $\tau(q)=T(q)=\tau_\nu(q)$, then $\{\tau'(q^+),\tau'(q^-)\}=\{T'(q),\tau_\nu'(q)\}$.
\end{itemize}
}
\end{rem}

\subsection{The case $\alpha=T'(q)$ and $\tau(q)=T(q)$ with $q\in (0,q_c)\setminus\{1\}$}\label{tauegalpsi} At first we must recall some facts about the multifractal analysis of $\mu$.

For $q\ge 0$, let $\mu_q$ be the Mandelbrot measure built with the random vectors
$$
W_q(u,v)= (m^{T(q)}W_{i,j}(u,v)^q)_{0\le i,j\le m-1},\quad (u,v)\in \bigcup_{n\ge 0}\Sigma_n\times\Sigma_n.
$$
According to the study achieved in \cite{Ba00}, with probability 1, conditionally on $\mu\neq 0$, all the  Mandelbrot measures $\mu_q$, $q\in [0,q_c)$ are defined simultaneously and one has $\dim (\mu_q)=T^*(T'(q))>0$ and $E(\mu,T'(q))$ is of full $\mu_q$-measure.

\begin{pro}\label{preservdim} Fix $q\in (0,q_c)\setminus\{1\}$ such that $\tau(q)=T(q)$. With probability 1, conditionally on $\{\mu\neq 0\}$, we have $\dim(\pi_*\mu_q)= T^*(T'(q))$.
\end{pro}

The following corollary is our main goal.
\begin{cor}\label{4.6}
Fix $q\in (0,q_c)\setminus\{1\}$ such that $\tau(q)=T(q)$. With probability 1, conditionally on $\{\mu\neq 0\}$, we have $\dim E(\pi_*\mu,T'(q))\ge  T^*(T'(q))=\tau^*(T'(q))$.
\end{cor}

We start with the proof of the corollary.

\begin{proof} Suppose $\mu\neq 0$. At first, we show that $\tau_{\pi_*\mu}'(0+)\ge \tau'(0+)\ge  T'(q)$. To see this, observe at first  that $\tau_{\pi_*\mu}(0)=\tau(0)$ since  $-\tau_{\pi_*\mu}'(0)$ is the upper box dimension of $\pi(K)$ and by \eqref{dimproj} we have $-\tau(0)=\dim_B\pi(K)$.  Since, moreover, we have $\tau_{\pi_*\mu}\ge \tau$ over $(0,1]$ by Proposition~\ref{tau}(2), we get the first inequality. Now  if $\tau'(0+)< T'(q)$, the equality $\tau(q)=T(q)$ yields
$$
T^*(T'(q))=qT'(q)-T(q)>q\tau'(0+) -\tau(q)\ge  \tau^*(\tau'(0+))=-\tau(0)=\dim_B(\pi(K)).
$$
However, by Proposition~\ref{preservdim}, we have $\dim(\pi_*\mu_q)= T^*(T'(q))$, so $\dim(\pi_*\mu_q)>\dim_B\pi(K)$, which is impossible since $\pi_*\mu_q$ is supported on $\pi(K)$. Thus $\tau'(0)\ge T'(q)$.

There is a subset $F_q$ of $\mathrm{supp}(\mu)$ of full  $\mu_q$-measure such that for all $t\in F_q$, $\dim_{\rm loc}(\mu_q,t)=\dim_{\rm loc}(\pi_*\mu_q,\pi(t))=T^*(T'(q))$ (by Proposition~\ref{preservdim}) and $\dim_{\rm loc}(\mu,t)=T'(q)$ (by the multifractal analysis of $\mu$ \cite{Ba00}). This implies that for all $t\in F_q$ we  have $\underline \dim_{\rm loc}(\pi_*\mu, \pi(t))\le\overline \dim_{\rm loc}(\pi_*\mu, \pi(t))\le \overline\dim_{\rm loc}(\mu,t)= T'(q)$. On the other hand, since $T'(q)\le \tau_{\pi_*\mu}'(0+)$, for all $\alpha'<T'(q)$, by \eqref{MF} we have
\begin{eqnarray*}
\dim \underline E^{\le }(\pi_*\mu, \alpha') \le \tau_{\pi_*\mu}^*(\alpha')\le   \alpha' q-\tau_{\pi_*\mu}(q)< T'(q)q-\tau_{\pi_*\mu}(q)\le T'(q)q-\tau(q)=T^*(T'(q)).
  \end{eqnarray*}
Consequently, since the family $( \underline E(\pi_*\mu, \alpha'))_{\alpha'<T'(q)}$ is non decreasing and $\dim (\pi_*\mu_q)=T^*(T'(q))$, we get $\pi_*\mu_q\big (\bigcup_{\alpha'<T'(q)} \underline E(\pi_*\mu, \alpha'))=0$. Now, set $ \widetilde F_q=\pi(F_q)\setminus \bigcup_{\alpha'<T'(q)}\underline E(\pi_*\mu, \alpha')$. By construction we have $\widetilde F_q\subset E(\pi_*\mu,T'(q))$ and $\pi_*\mu_q(\widetilde F_q)>0$. Finally $\dim E(\pi_*\mu,T'(q))\ge  T^*(T'(q))$. Moreover, by  Remark~\ref{taupsitaunu}, if $q\le 1$ then  $\tau'(q)=T'(q)$, and  if $q>1$, then $T'(q)\in\{\tau'(q^+),\tau'(q^-)\}$, so $T^*(T'(q))=\tau^*(T'(q))$.
\end{proof}

Proposition~\ref{preservdim} is a consequence of Theorem~\ref{DIM} and the following lemma.

\begin{lem}\label{muq/espmuq}
If $q\in (0,q_c)\setminus\{1\}$ and $\tau(q)=T(q)$,   then, conditionally on $\mu_q\neq 0$,  $\dim(\mu_q)=T^*(T'(q))\le \dim (\E(\pi_*\mu_q))$.
\end{lem}

\begin{proof} We first show that for any $q\in (0,q_c)$, almost surely, conditionally on $\mu_q\neq0$, we have
\begin{equation}
\label{e-diff}
\dim(\mu_q)-\dim (\E(\pi_*\mu_q))=\sum_{i=0}^{m-1} p_i^qm^{T(q)-T_i(q)} T_i^*(T'_i(q)).
\end{equation}
To see this, notice that $\E(\pi_*\mu_q)$ is a Bernoulli product measure on $\Sigma$ generated by the probability vector $(p_0',\ldots, p_{m-1}')$ with
$$
p_i': = \sum_{j=0}^{m-1}m^{T(q)}\E(W_{i,j}^q)=p_i^q m^{T(q)-T_i(q)}.
$$
 A  simple computation yields that
 \begin{equation}
 \label{e-split}
 \begin{split}
 \dim(\E(\pi_*\mu_q))
 &=-\frac{1}{\log m} \sum_{i=0}^{m-1} p_i'\log p_i'\\
 &= -T(q)+\left(\sum_{i=0}^{m-1} p_i^q m^{T(q)-T_i(q)} T_i(q)\right)-\frac{q}{\log m}
 \left(\sum_{i=0}^{m-1} p_i^q \log p_i m^{T(q)-T_i(q)}\right).
 \end{split}
 \end{equation}
 In the meantime, since $\sum_{i=0}p_i^q m^{T(q)-T_i(q)}=1$,  differentiating with respect to $q$ yields
 \begin{equation}
 \label{e-psi'}
 T'(q)=\left( \sum_{i=0}^{m-1} p_i^q m^{T(q)-T_i(q)}T_i'(q)\right)-\frac{1}{\log m} \left(
 \sum_{i=0}^{m-1} p_i^q \log p_i m^{T(q)-T_i(q)}
 \right).
 \end{equation}
 Since $\dim(\mu_q)=T^*(T'(q))=T'(q)q-T(q)$ almost surely, by \eqref{e-split} and \eqref{e-psi'} we obtain~\eqref{e-diff}.

 Next we show that if $\tau(q)=T(q)$ for some $q\in (0, q_c)\setminus\{1\}$, then
 $$\sum_{i=0}^{m-1} p_i^qm^{T(q)-T_i(q)} T_i^*(T'_i(q))\leq 0.$$

We consider the cases $q\in (1, q_c)$ and $0<q<1$ separately.
First suppose $q\in (1, q_c)$. For $1\le s\le q$ and $n\ge 1$ we have
\begin{eqnarray*}
\E\Big (\sum_{u\in\Sigma_n}\pi_*\mu([u])^q\Big )&\ge& \sum_{u\in\Sigma_n}\nu([u])^q \E(X(u)^s)^{q/s}\\
&= & \sum_{u\in\Sigma_n}\nu([u])^q \E\Big (\sum_{v\in\Sigma_n} \Big (\frac{\mu([u,v])}{\nu([u])}\Big )^s\Big )^{q/s}\\
&= & \sum_{u\in\Sigma_n}\nu([u])^q \E\Big (\sum_{v\in\Sigma_n} Y(u,v)^s \prod_{k=1}^nV_{u_k,v_k}(u_{|k-1},v_{|k-1})^s\Big )^{q/s}\\
&=&\E(Y^s)^{q/s} \sum_{u\in\Sigma_n}\prod_{k=1}^n p_{u_k}m^{-T_{u_k}(s)q/s}\\
&\ge &
\Big (\sum_{i=0}^{m-1} p_i^q m^{-qT_i(s)/s}\Big )^n,
\end{eqnarray*}
since $1=\E(Y)\le \E(Y^s)^{1/s}$. Consequently, due to Corollary~\ref{momestimate}, we have
\begin{equation}
\label{e-sup}
-\tau(q)=\max (-\tau_\nu(q),-T(q))\ge \sup_{1\le s\le q} \log_m\sum_{i=0}^{m-1} p_i^q m^{-qT_i(s)/s}.
\end{equation}
Since $\tau(q)=T(q)$, this implies that the supremum is reached at $s=q$. Differentiating with respect to $s$ at $s=q$ then yields $\sum_{i=0}^{m-1} p_i^qm^{-T_i(q)} T_i^*(T'_i(q))\le 0$, hence $\sum_{i=0}^{m-1} p_i^qm^{T(q)-T_i(q)} T_i^*(T'_i(q))\le 0$.

In the end, suppose that $0< q<1$.  By the definition of $\tau$, the condition  $\tau(q)=T(q)$ also implies that the following infimum
$$
\inf_{q \le s\leq 1} \log_m\sum_{i=0}^{m-1} p_i^q m^{-qT_i(s)/s}
$$
is attained at $q$.  Hence  differentiating with respect to $s$ at $s=q$ yields $$\sum_{i=0}^{m-1} p_i^qm^{-T_i(q)} T_i^*(T'_i(q))\le 0.$$  This completes the proof of the lemma.
\end{proof}

\subsection{The case $\alpha=\tau'(q)$ with $\tau(q)\neq T(q)$ and $q\in (0,\widetilde q_c)\setminus\{1\}$,  or $\alpha\in \{\tau'(q^+),\tau'(q^-)\}$ when a first order phase transition occurs at $q\in(1,q_c)$}\label{taudifpsi}
$\ $

In this section, we suppose that we do not have $\tau\equiv\tau_\nu\equiv  T$ over $[0,\widetilde q_c)$, i.e. we are not in the case where  for each $0\le i\le m-1$ such that $p_i>0$ the function $T_i$ is equal to 0.

We will use the notations of Section~\ref{study of tau}, and we set $s(q)=1$ if both $q>1$ and $\tau(q)=\tau_\nu(q)$ hold. Also we recall Remark~\ref{taupsitaunu}.

For $q\in (0,\widetilde q_c)$ such that $s(q)$ is defined, for $0\le i\le m-1$ set
$$
p'_i=p'_{q,i}=m^{\tau(q)}p_i^qm^{-qT_i(s(q))/s(q)}.
$$
Also let $\nu'=\nu'_q$ be the Bernoulli measure associated with $p'=(p'_0,\ldots,p'_{m-1})$.

For $s>0$ and $0\le i,j\le m-1$,  set
\begin{equation}\label{V'i}
V'_{s,i,j}=\mathbf {1}_{\{V_{i,j}>0\}} V_{i,j}^sm^{T_i(s)},
\end{equation}
so that for $q'\ge 0$
$$
T_{V'_{s,i}}(q'):=-\log_m\sum_{j=0}^{m-1}\E({V'_{s,i,j}}^{q'})= T_i(q's)-q'T_i(s).
$$
Set $W'_s=(W'_{s,i,j}=p'_i V'_{s,i,j})_{0\le i,j\le m-1}$. We have
$$
T_{W'_s}(q')=\sum_{i=0}^{m-1}(p'_i)^{q'} m^{-T_{V'_{s,i}}(q')}.
$$
For all $(u,v)\in\bigcup_{n\ge 1}\Sigma_n\times\Sigma_n$, let $W'_{s,i,j}(u,v)=(p'_i \mathbf {1}_{\{V_{i,j}(u,v)>0\}} V_{i,j}(u,v)^sm^{T_i(s)})_{0\le i,j\le m-1}$. This family of random weights generates a Mandelbrot mesure $\mu_{W'_s}$ simultaneously with~$\mu_W$.

We start with a first lemma.

\begin{lem}\label{lem-4.8}
\begin{enumerate}
\item
If $q\in (0,1)$ and $s(q)\in(0,1)$, then for all $s\in (0,s(q))$  we have $\sum_{i=0}^{m-1}p_i'T_i^*(T_i'(s))>0$.

\item  If $q\in  (0,\widetilde q_c)\setminus\{1\}$ and $s(q)=1$, then either  $\sum_{i=0}^{m-1}p_i'T_i^*(T_i'(s))=0$ for all $s\in [0,1]$, or $\sum_{i=0}^{m-1}p_i'T_i^*(T_i'(s))>0$ for all $s\in (0,1)$ according to whether $T_i$ is affine (and equal to $q\mapsto (q-1)\log_m(\mathbb{E}(N_i)$) for each $i$ such that $p_i>0$  and $\sum_{i=0}^{m-1}p_i'T_i'(1)=0$, or not.

Moreover, either the set $\widetilde S$ of those $q\in (0,\widetilde q_c)$ for which $\sum_{i=0}^{m-1}p_i'T_i^*(T_i'(s))=0$ for all $s\in [0,1]$ is discrete or it is equal to $\mathbb (0,\widetilde q_c))$. The later case holds if and only if property~$(\mathcal P)$ of Remark~\ref{specialtau}(2) holds. In particular, $T$ is finite over $\R_+$, $\widetilde q_c=\infty$, and one has  $\tau=\tau_\nu>T$ over $(0,1)$ and $\tau=\tau_\nu<T$ over $(1,\infty)$.
\end{enumerate}
\end{lem}
\begin{proof} (1) Suppose $q\in (0,1)$ and $s(q)\in (q,1)$.  The study of the differentiability of $\tau$ achieved in Section~\ref{difftau} yields $\sum_{i=0}^{m-1}p_i'T_i^*(T_i'(s(q)))=m^{\tau(q)}g(q,s(q))=0$ and since $\frac{\partial g}{\partial s}(q,s(q))<0$, we have $g(q,s)=m^{-\tau(q)} \sum_{i=0}^{m-1}p_i'T_i^*(T_i'(s))>0$ for all $s\in (0,s(q))$.

(2) Suppose  that $q\in(0,1)$ and $s(q)=1$. That means that we have $\tau(q)=\tau_\nu(q)$. Here again, we can use the study of $\tau$ to get that $\sum_{i=0}^{m-1}p_i'T_i^*(T_i'(1))=\sum_{i=0}^{m-1}p_i'T_i'(1) =m^{\tau_\nu(q)}g(q,1)\ge 0$. Now, notice that the derivative of $s\mapsto \sum_{i=0}^{m-1}p_i'T_i^*(T_i'(s))$ is $s\mapsto \sum_{i=0}^{m-1}p_i' sT_i''(s)$. If one of the $T_i$ is not affine, then by an argument given in the study of the differentiability of $\tau$ we have that $T_i''$ is strictly negative so  $\sum_{i=0}^{m-1}p_i'T_i^*(T_i'(s))>0$ for all $s\in (0,1)$. Otherwise, the function $\sum_{i=0}^{m-1}p_i'T_i^*\circ T_i'$ is identically equal to $ 0$ over its domain by analyticity.

Suppose now that $q\in (1,\widetilde q_c)$ and $s(q)=1$. We have $s(q)=1$. The condition $\tau(q)=\tau_\nu(q)\le T(q)$ implies that $\sum_{i=0}^{m-1}p'_i m^{-T_i(q)}=\sum_{i=0}^{m-1}p_i^q m^{\tau_\nu(q)-T_i(q)}\le 1$.  Since, moreover, $\sum_{i=0}^{m-1}p'_i m^{-T_i(1)}=1$, by convexity of $q\mapsto \sum_{i=0}^{m-1}p'_i m^{-T_i(q)}$, we must have $\sum_{i=0}^{m-1}p'_i T'_i(1)\ge 0$. Then, the same arguments as in previous  paragraph yield the same conclusion.

For each $q$ such that $s(q)=1$ and  $\sum_{i=0}^{m-1}p_i'T_i^*(T_i'(s))=0$ for all $s\in [0,1]$, the functions $T_i$ are linear and  we have  $p_i'=p_i^qm^{\tau_\nu(q)}$, so $\sum_{i=0}^{m-1}p_i^q T'_i(1)=-\sum_{i=0}^{m-1}p_i^q\log(\mathbb E(N_i))=0$. If the set of such points $q$ has an accumulating point, then by analyticity, we must have $\sum_{i=0}^{m-1}p_i^q \log(\mathbb E(N_i))=0$ for all $q$. It is then not hard to conclude that property $(\mathcal P)$ holds. Then, $T$ is finite over $\R_+$, and the study of $\inf_{q\le s\le 1} \log_m\sum_{i=0}^{m-1}p_i^qm^{-qT_i(s)/s}$ for $q\in (0,1)$ and $\sup_{1\le s\le q} \log_m\sum_{i=0}^{m-1}p_i^qm^{-qT_i(s)/s}$ for  $q\in (1,\infty)$ shows that both are uniquely reached  at $s=1$, so $\tau=\tau_\nu>T$ over $(0,1)$ and $\tau=\tau_\nu<T$ over $(1,\infty)$.
\end{proof}

\begin{lem}\label{4.9} Let $q\in (0,\widetilde q_c)$ such that $s(q)$ is defined. Suppose that  $s>0$ is such that $\sum_{i=0}^{m-1}p_i'T_i^*(T_i'(s))\ge 0$. With probability 1, the Mandelbrot measure $\mu_{W'_s}$ has the same topological support as $\mu$. If, moreover, $\sum_{i=0}^{m-1}p_i'T_i^*(T_i'(s))>0$ then, conditionally on $\mu_{W'_s}\neq 0$, the measure $\pi_*\mu_{W'_s}$ is absolutely continuous with respect to $\nu'$. In particular, $\nu'(\pi(K))>0$.
\end{lem}
\begin{proof} At first we notice that $\sum_{i=0}^{m-1} p'_i T_{V'_{s,i}}'(1-)=\sum_{i=0}^{m-1}p_i'T_i^*(T_i'(s))$. Thus, due to \eqref{dimrel} our assumption implies  $T_{W'_s}'(1-)\ge \dim(\nu')$, hence $\mu_{W'_s}$ is non degenerate. Moreover, since the weights $W'_{s,i,j}$ and $W_{i,j}$ vanish simultaneously, Proposition~\ref{Kmu} shows that $\mu_{W'_s}$ and $\mu$ have   almost surely the same topological support. If, in addition,  $\sum_{i=0}^{m-1}p_i'T_i^*(T_i'(s))>0$, then $T_{W'_s}'(1-)> \dim(\nu')$ and by Theorem~\ref{ABS}(1)(a), this implies that  $\pi_*\mu_{W'_s}$ is almost surely absolutely continuous with respect to $\E(\pi_*\mu_{W'_s})=\nu'$, so  $\nu'(\pi(K))>0$.
\end{proof}

Now, for $q\in(0,\widetilde q_c)\setminus\{1\}$, if $s(q)<1$ or if $s(q)=1$ and $\sum_{i=0}^{m-1}p_i'T_i^*(T_i'(s))>0$ for all $s\in (0,1)$, let $\widetilde\nu_q=\nu'_q$. Otherwise, i.e. if $q\in \widetilde S$ ($\widetilde S$ is defined in Lemma~\ref{lem-4.8}) set $\widetilde\nu_q=\pi_*\mu_{W'_1}$ (recall that this Mandelbrot measure is defined before Lemma~\ref{lem-4.8} and it has teh same topological support as $\mu$ almost surely by Lemma~\ref{4.9}). The main result of this section is the following.

\begin{pro}\label{tauneqpsi}
Let $q\in (0,\widetilde q_c)\setminus\{1\}$ at which $\tau(q)\neq T(q)$ or $q\in (1,q_c)$ at which  $\tau(q)=\tau_\nu(q)=T(q)$. Set $\alpha=\tau'(q)$ if $s(q)<1$ and $\alpha=\tau_\nu'(q)$ otherwise.

With probability 1, conditionally on $\mu\neq 0$, we have $\widetilde \nu_q(E(\pi_*\mu,\alpha))>0$, and $\dim(\widetilde\nu_q)=\tau^*(\alpha)$. Consequently, $\dim_H E(\pi_*\mu,\alpha)\ge \tau^*(\alpha)$.
\end{pro}

From now on we fix $q\in (0,\widetilde q_c)\setminus\{1\}$ at which $\tau(q)\neq T(q)$ or $\tau(q)=\tau_\nu(q)=T(q)$.

\begin{lem}
Suppose that $\widetilde\nu_q=\nu'$. Let $\mathscr S$ stand for a maximal open interval of points $s>0$ such that  $\sum_{i=0}^{m-1}p_i'T_i^*(T_i'(s))>0$ and $\E(Y^s)<\infty$. With probability 1, conditionally on $\mu\neq 0$,   for $\nu'$-almost every $x$ in  $\pi(K)$,  for all $s\in \mathscr S$ we have
$$
\lim_{n\to\infty} \frac{-1}{n} \log_m\sum_{v\in \Sigma_n}\Big (\frac{\mu([x_{|n},v])}{\nu(x_{|n})}\Big )^s=\sum_{i=0}^{m-1} p'_iT_i(s).
$$
\end{lem}
\begin{proof}

By convexity,  we only need to check this for each $s$ in a dense countable set $\mathcal S$ of $\mathscr{S}$.  Indeed, if this is done, there exists a subset of $\{\mu\neq 0\}$  of probability $\mathbb P(\mu\neq 0)$ such that the sequence of concave functions $f_n(s)= \frac{-1}{n} \log_m\sum_{v\in \Sigma_n}\Big (\frac{\mu([x_{|n},v])}{\nu(x_{|n})}\Big )^s$ converge pointwise on $\mathcal S$, and this is enough to get the convergence over $\mathscr S$.

Fix $s\in\mathcal S$. For $n\ge 1$ and $x$ in the topological support of $\nu'$, set
\begin{eqnarray*}
Z_{s,n}(x)&=&\Big (\prod_{k=1}^n m^{T_{x_k}(s)}\Big )\sum_{v\in \Sigma_n}\Big (\frac{\mu([x_{|n},v])}{\nu([x_{|n}])}\Big )^s\\
&=&\sum_{v\in \Sigma_n} Y(x_{|n},v)^s\cdot \prod_{k=1}^n m^{T_{x_k}(s)} V_{x_k,v_k}(x_{|k-1},v_{|k-1})^s.
\end{eqnarray*}
Define $V'_{s,i}$, $0\le i\le m-1$, as in \eqref{V'i}. Since $\mathscr S$ is open, we have  $\E(Y^{q's})$ and $\sum_{i=0}^{m-1} p'_iT_{V'_{s,i}}(q')>-\infty$ for some $q'>1$, and since $\sum_{i=0}^{m-1} p'_iT'_{V'_{s,i}}(1)>0$, we also have $\sum_{i=0}^{m-1} p'_iT_{V'_{s,i}}(q')>0$ if $q'$ is close enough to 1. By Proposition~\ref{XntoX} applied with $\eta=\nu'$ and $U_i=V'_{s,i}$, the sequence $Z_{s,n}(x)$ converges $\mathbb{P}\otimes\nu'$ almost surely to the same non degenerate limit  $\widetilde Z_s(x)$ as the Mandelbrot martingale in random environment
$$
\widetilde Z_{s,n}(x)=\sum_{v\in \Sigma_n}  \prod_{k=1}^n m^{T_{x_k}(s)} V_{x_k,v_k}(x_{|k-1},v_{|k-1})^s.
$$
This variable satisfies the equation
\begin{equation}\label{eqZs}
\widetilde Z_{s}(x)=\sum_{j=0}^{m-1} V_{x_1,j}\widetilde Z_s(\sigma x,j),
\end{equation}
where the $\widetilde Z_s(\sigma x,j)$ are independent copies of $\widetilde Z_s(\sigma x)$, which are also independent of $V_{x_1}$.

Equation \eqref{eqZs} shows that $\mathbb{P}(\{\widetilde Z_{s}(x)=0\})$ is $\{f_i\}_{0\le i\le m-1; p_i>0}$-stationnary  in the sense of Appendix~\ref{A}, where $f_i$ stands for the generating function of the random integer $N_i$. Moreover, we assumed from the beginning that there exists  $0\le i\le m-1$ such that $p_i>0$ for which $\mathbb{P}(N_i=1)<1$. Consequently, Proposition~\ref{f-stat} shows that for $\nu'$-almost every $x$,  $\mathbb{P}(\{\widetilde Z_{s}(x)=0\})$ is less than 1 (because $\widetilde Z_{s,n}(x)$ is non degenerate) and  independent of $s\in \mathcal S$.

Also, for each $s\in\mathcal S$, the event $\{\widetilde Z_{s}(x)=0\}$ contains the event $\bigcup_{n\ge 1} \{\widetilde Z_{s,n}(x)=0\}$, which due to the definition of $\widetilde Z_{s,n}$ is independent of $s$ and is equal to the extinction of the branching process defining the Galton-Watson tree in random environment $T_n(x)=\{v\in\Sigma_n:\  Q(x_{|n},v)>0\}$.  In addition, the function $\mathbb{P}(\bigcup_{n\ge 1} \{T_n(x)=\emptyset\})$ is $\{f_i\}_{0\le i\le m-1; p_i>0}$-stationnary  as well, and it cannot be equal to 1 since it is smaller than or equal to $\mathbb{P}(\{\widetilde Z_{s}(x)=0\})$. Consequently, we conclude that for $\nu'$-almost every $x$, the event $\{\widetilde Z_{s}(x)>0\text{ for all $s\in\mathcal{S}$}\}$ equals  $A_x=\bigcap_{n\ge 1} (A_{x,n}:=\{\{v\in\Sigma_n:\  Q(x_{|n},v)>0\}\neq\emptyset\})$ up to a set of probability 0.

We have
\begin{eqnarray*}
&&\int\nu'(\{x:Z_s(\omega,x)> 0 \ \forall \ s\in \mathcal S,\ \text{and } \omega\in A_x\}\, \mathbb {P}(\mathrm{d}\omega)\\
&&\quad =\E_{\mathbb{P}\otimes \nu'}(\mathbf{1}_{\{Z_s(\omega,x)> 0 \ \forall \ s\in \mathcal S,\ \text{and } \omega\in A_x\}})\\
&&\quad=\int \mathbb{P}( Z_s(\omega,x)> 0 \ \forall \ s\in \mathcal S,\ \text{and } \omega\in A_x\})\,\nu'(\mathrm{d}x)\\
&&\quad= \int \mathbb{P}(A_x)\,\nu'(\mathrm{d}x).
\end{eqnarray*}
Notice that the events $A_{x,n}$ are non increasing so $\mathbb{P}(A_x=\bigcap_{n\ge 1} A_{x,n})=\lim_{n\to\infty} \mathbb{P}(A_{n,x})$. Consequently,
\begin{eqnarray*}
&&\int\nu'(\{x:Z_s(\omega,x)> 0 \ \forall \ s\in \mathcal S,\ \text{and } \omega\in A_x\}\, \mathbb {P}(\mathrm{d}\omega)\\&&\quad=\int\lim_{n\to\infty}\mathbb{P}(A_{x,n})\,\nu'(\mathrm{d}x)\\
&&\quad=\lim_{n\to\infty}\int\mathbb{P}(A_{x,n})\,\nu'(\mathrm{d}x)\\
&&\quad=\lim_{n\to\infty}\E(\nu'(\{x: \{v\in\Sigma_n:\  Q(x_{|n},v)>0\}\neq\emptyset\}))\\
&&\quad=\lim_{n\to\infty}\E(\nu'(\pi(K_n)))\\
&&\quad=\E(\nu'(\lim_{n\to\infty}\pi(K_n)))\\
&&\quad=\E(\nu'(\pi(K)).
\end{eqnarray*}
Since the inclusion $\{x: Z_s(\omega,x)> 0\}\subset \pi(K(\omega))$ holds by construction, we obtained that $\nu'(\{x\in \pi(K(\omega)):Z_s(\omega,x)> 0 \ \forall \ s\in \mathcal S\}=\nu'(\pi(K(\omega)))$ almost surely. In other words,  with probability 1, conditionally on $\mu\neq 0$, for $\nu'$-almost every $x$ in $\pi(K)$, for all $s\in\mathcal S$ we have $\widetilde Z_s(x)>0$. Finally, $\widetilde Z_s(x)$ is the positive limit of $Z_{s,n}(x)$. Since by definition we have
$$
\sum_{v\in \Sigma_n}\Big (\frac{\mu([x_{|n},v])}{\nu([x_{|n}])}\Big )^s=\Big (\prod_{k=1}^n m^{T_{x_k}(s)}\Big )^{-1}Z_{s,n}(x)
$$
we conclude that
$$
\lim_{n\to\infty} \frac{-1}{n} \log_m\sum_{v\in \Sigma_n}\Big (\frac{\mu([x_{|n},v])}{\nu(x_{|n})}\Big )^s=\lim_{n\to\infty}\frac{1}{n}\sum_{k=1}^n T_i(s)=\sum_{i=0}^{m-1} p'_iT_i(s),
$$
due to the ergodic theorem applied to $\nu'$.

\end{proof}

\begin{lem}\label{lem4.12}
Suppose that $\widetilde\nu_q= \nu'$.  Let
 $$
s_0=\sup\left\{s>0: \sum_{i=0}^{m-1}p_i'T_i^*(T_i'(s))>0\text{ and }\E(Y^s)<\infty\right\}.
$$
With probability 1, for $\nu'$-almost every $x\in \pi(K)$, we have $\lim_{n\to\infty} -\frac{1}{n}\log_m X_n(x)=0$ or $\sum_{i=0}^{m-1}p_i'T'_i(s_0)$  according to whether $s_0> 1$ or $s_0\le 1$.
\end{lem}
\begin{proof}
We notice that $s_0=s(q)$ when $s(q)<1$. Due to the previous lemma,  with probability $1$,  for $\nu'$-almost every $x\in\mathrm{supp}(\pi(K))$, defining
$$
\tau_x(s)=\liminf_{n\to\infty} \frac{-1}{n} \log_m\sum_{v\in \Sigma_n}\Big (\frac{\mu([x_{|n},v])}{\nu([x_{|n}])}\Big )^s,
$$
we have
\begin{equation}\label{LDP}
\tau_x(s)=\lim_{n\to\infty} \frac{-1}{n} \log_m\sum_{v\in \Sigma_n}\Big (\frac{\mu([x_{|n},v])}{\nu([x_{|n}])}\Big )^s=\sum_{i=0}^{m-1} p'_iT_i(s)
\end{equation}
over $[0,s_0)$. On the other hand, we naturally have
\begin{equation}\label{ineqtau}
\tau_x(s)\ge \widetilde T(s):=\sum_{i=0}^{m-1} p'_iT_i(s)
\end{equation}
for all $s$. This is due to Lemma~\ref{Borel} and the fact that
$$
\E\Big(\sum_{v\in \Sigma_n}\Big (\frac{\mu([x_{|n},v])}{\nu([x_{|n}])}\Big )^s\Big)= \prod_{k=1}^n m^{-T_{x_k}(s)}.
$$

Now let us make a few remarks.

There exist $\alpha_0<\beta_0$ in $\R$ such that, with probability 1, conditionally on $\mu\neq 0$, we have  $m^{-n\beta_0}\le \frac{\mu([x_{|n},v])}{\nu([x_{|n}])}\le m^{-n\alpha_0}$, for all $x\in\pi(K)$, $n\ge 1$ and $v\in\Sigma_n$ such that $\mu([x_{|n},v])>0$.
Indeed, for all $x\in\pi(K)$, $n\ge 1$ we already have $(\min \{p_i: p_i>0\})^n\le \nu(x_{|n})\le (\max \{p_i: p_i>0\})^n$. Also,  we can fix $\eta>0$ such that $C_\eta=\max(\E(\mathbf{1}_{\{Y>0\}}Y^{-\eta}),\E(Y^{\eta}))<\infty$. Then, for any $A>0$, and $n\ge 1$, we have
\begin{eqnarray*}
&&\mathbb{P}\left (\exists\ (u,v)\in\Sigma_n\times\Sigma_n:  0<\mu([u,v])\le  m^{-nA} \text{ or } \mu([u,v])\ge  m^{nA}\right )\\
&&\quad\le M^{-n\eta A}\E\Big (\sum_{(u,v)\in\Sigma_n\times\Sigma_n} \mathbf{1}_{\{Q(u,v)>0\}}\mathbf{1}_{\{Y(u,v)>0\}}Q(u,v)^{-\eta}Y(u,v)^{-\eta}\Big )\\
&&\quad\quad\quad\quad+M^{-n\eta A}\E\Big (\sum_{(u,v)\in\Sigma_n\times\Sigma_n}Q(u,v)^\eta Y(u,v)^\eta\Big ) \\
&&\quad\le C_\eta M^{-n\eta A}\left ( m^{-nT(-\eta)}+m^{-nT(\eta)}\right ).
\end{eqnarray*}
Hence, if  $A$ is large enough so that $A\eta+\min (T(-\eta),T(\eta))>0$, by the Borel-Cantelli lemma we get $m^{-nA}\le \mu([x_{|n},v])\le m^{nA}$ for all $x\in\pi(K)$, $n\ge 1$ large enough and and $v\in\Sigma_n$ such that $\mu([x_{|n},v])>0$.
\medskip

If $s_0\le 1$, then $\widetilde T'(s_0)$ exists and by definition of $s_0$ we have  $\widetilde T^*(\widetilde T'(s_0))=0$. Moreover, $\widetilde T^*\circ \widetilde T'$ is strictly decreasing in a neighborhood of $s_0$ since we have already shown that when they are defined at some $s$, the functions $T_i''$ cannot vanish simultaneously there.   This, together with \eqref{ineqtau}  implies that for all $\alpha< \widetilde T'(s_0)$  we have $\widetilde\tau_x^*(\alpha)\le \widetilde T^*(\alpha)<0$. Thus,
\begin{equation}\label{LDP}
\alpha< \widetilde T'(s_0),\text{ for $n$ large enough, }\left \{v\in\Sigma_n: \frac{\mu([x_{|n},v])}{\nu(x_{|n})}\geq m^{-n\alpha}\right\}=\emptyset.
\end{equation}

Over its domain, which contains a neighborhood of $[0,1]$, the mapping $s\mapsto \widetilde T^*(\widetilde T'(s))-\widetilde T'(s)$ is increasing on the left of 1 and decreasing on the right, and it takes the maximum value 0 at 1. In other words, over its domain, the mapping $\alpha \mapsto \widetilde T^*(\alpha)-\alpha$ is strictly increasing on the left of $\widetilde T'(1)$ and strictly decreasing on the right of $\widetilde T'(1)$, since $q\mapsto T'(q)$ is decreasing.

Now, for $\alpha\in \R$, $n\ge 1$ and $\epsilon>0$ define
$$
f(n,\alpha,\epsilon)=\frac{1}{n}\log_m\#\left \{v\in\Sigma_n: m^{-n(\alpha+\epsilon)}\le \frac{\mu([x_{|n},v])}{\nu([x_{|n}])}\le m^{-n(\alpha-\epsilon)}\right\}.
$$
Fix $\eta>0$ and $\epsilon>0$. For any $\alpha\in [\alpha_0,\beta_0]$, there exists $\epsilon_\alpha\in (0,\epsilon)$ and $n_\alpha\ge 1$ such that for all $n\ge n_\alpha$ we have
$$
f(n,\alpha,\epsilon_\alpha)\le \tau_x^*(\alpha)+\eta \le \widetilde T^*(\alpha)+\eta.
$$
Set $\alpha_c=\widetilde T'(s_0)$ if $s_0\le 1$ and $\alpha_c=\widetilde T'(1)$ otherwise.
Fix a finite covering $\bigcup_{i=1}^N(\alpha_i-\epsilon_i,\alpha_i+\epsilon_i)$ of $[\alpha_0,\beta_0]\setminus (\alpha_c-\epsilon_c,\alpha_c+\epsilon_c)$, where $\epsilon_c$ stands for $\epsilon_{\alpha_c}$, and $\epsilon_i$ stands for $\epsilon_{\alpha_i}$, and set $n_0=\sup\{n_{\alpha}:\alpha\in \{\alpha_c\}\cup \{\alpha_i:1\le i\le N\}\}$. Without loss of generality we assume that  the $\alpha_i$ belong to  $[\alpha_0,\beta_0]\setminus (\alpha_c-\epsilon_c,\alpha_c+\epsilon_c)$. Moreover, due to \eqref{LDP},  if $s_0\le 1$ we can restrict the $\alpha_i$ to be larger than or equal to $\alpha_c$, and set $\alpha_0=\alpha_c$.
Then, there exists $\gamma>0$ such that for all $\alpha_i$ we have $\widetilde T^*(\alpha_i)-\alpha_i\le \widetilde T^*(\alpha_c)-\alpha_c-\gamma$.

For $n\ge n_0$ we have
\begin{eqnarray*}
X_n(x)&=&\sum_{v\in \Sigma_n}\frac{\mu([x_{|n},v])}{\nu(x_{|n})}\\
&\le& m^{f(n,\alpha_c,\epsilon_c)} m^{-n(\alpha_c-\epsilon_c)}+ \sum_{i=1}^N m^{f(n,\alpha_i,\epsilon_i)} m^{-n(\alpha_i-\epsilon_i)}\\
&\le& m^{n(\eta+\epsilon)}\Big (m^{n(\tau_x^*(\alpha_c)-\alpha_c)}+\sum_{i=1}^N m^{n(\tau_x^*(\alpha_i)-\alpha_i)}\Big )\\
&\le& m^{n(\eta+\epsilon)}\Big (m^{n(\widetilde T^*(\alpha_c)-\alpha_c)}+\sum_{i=1}^N m^{n(\widetilde T^*(\alpha_i)-\alpha_i)}\Big )\\
&\le&  m^{n(\eta+\epsilon)}m^{n(\widetilde T^*(\alpha_c)-\alpha_c)}(1+N m^{-n\gamma}).
\end{eqnarray*}
We conclude that $\liminf_{n\to\infty}-\frac{1}{n}\log_m(X_n(x))\ge \alpha_c-\widetilde T^*(\alpha_c)-\eta-\epsilon$. Since this holds for any positive $\eta$ and $\epsilon$, we get the desired lower bound: $\widetilde T'(s_0)$ if $s_0\le 1$, and 0 otherwise.

On the other hand, due to \eqref{LDP}, Gartner-Ellis theorem (see e.g. \cite{De-Zei}) ensures that for all $s\in (0,\min(s_0,1))$ one has $\lim_{\epsilon\to 0}\liminf_{n\to\infty} f(n,\widetilde T'(s),\epsilon)=\widetilde T^*(\widetilde T'(s))$.

This immediately yields $\limsup_{n\to\infty}-\frac{1}{n}\log(X_n(x))\le \widetilde T'(s))-\widetilde T^*(\widetilde T'(s))$ for all $s<\min (s_0,1)$ close enough to $\min (s_0,1)$, hence $\limsup_{n\to\infty}-\frac{1}{n}\log_m(X_n(x))\le \alpha_c-\widetilde T^*(\alpha_c)$.
\end{proof}

\begin{proof}[Proof of Proposition~\ref{tauneqpsi}] Recall that $\alpha$ stands for $\tau'(q)$ if $s(q)<1$ and $\tau_\nu'(q)$ if $s(q)=1$.
We use the writing $\pi_*\mu([x_{|n}])=\nu([x_{|n}]) X_n(x)$.

At first we suppose that $\widetilde\nu_q=\nu'$.

If $q\in (0,1)$ and $s(q)<1$, applying ergodic theorem to $\nu'$ to control the local dimension of $\nu$, and Lemma~\ref{lem4.12} to $X_n(x)$ after setting  $s_0=s(q)$, we obtain that conditionally on $\pi_*\mu\neq 0$,
for $\nu'$-almost every $x\in\pi(K)$,
\begin{eqnarray*}
\dim_{\rm loc}(\pi_*\mu,x)=\Big (-\sum_{i=0}^{m-1}p'_i\log_m(p_i) \Big )+ \widetilde T'(s_0)&=&-\sum_{i=0}^{m-1}p'_i\log_m(p_i) +\sum_{i=0}^{m-1} p'_i T_i'(s(q))\\&=&\tau'(q)=\alpha,
\end{eqnarray*}
by using \eqref{tau'q} and \eqref{tau'q2}. On the other hand,
\begin{eqnarray*}
\dim(\nu')=-\sum_{i=0}^{m-1}p'_i\log_m(p'_i)&=&-\sum_{i=0}^{m-1}p'_i(q\log_m(p_i)+\tau(q)-qT_i(q)/s(q))\\
&=&q\tau'(q)-\tau(q)=\tau^*(\alpha)
\end{eqnarray*}
by using  \eqref{tau'q} and \eqref{tau'q1}.

If $s(q)=1$, then  $\nu'=\nu_q$, and this time one applies Lemma~\ref{lem4.12} with $s_0=s(q)=1$ to control $X_n(x)$ . This yields  that conditionally on $\pi_*\mu\neq 0$,
for $\nu'$-almost every $x\in\pi(K)$,
\begin{eqnarray*}
\dim_{\rm loc}(\pi_*\mu,x)=\Big (-\sum_{i=0}^{m-1}p'_i\log_m(p_i) \Big )+ 0=\tau_\nu'(q).
\end{eqnarray*}
Moreover, $\dim (\nu_q)=\tau_\nu^*(\tau_\nu'(q))=\tau_\nu'(q) q-\tau_\nu(q)=\alpha q-\tau(q)=\tau^*(\alpha)$ since we have $\tau_\nu(q)=\tau(q)$ and $\alpha\in \{\tau'(q^+),\tau'(q^-)\}$.

Thus, at this stage, due to Corollary~\ref{4.6} and the conclusions  obtained in the previous lines, for all $q\in (0,\widetilde q_c)\setminus \widetilde S$ and  $\alpha=\tau'(q)$ or $\alpha\in\{\tau'(q^+),\tau'(q^-)\}$ if $q>1$ and the graphs of $\tau_\nu$ and $T$ cross transversally at $(q,T(q)$, we have established the desired inequality $\dim_H E(\pi_*\mu,\alpha)\ge \tau^*(\alpha)$, almost surely, conditionally on $\mu\neq 0$.

Now suppose that $q\in \widetilde S$. Recall that $\widetilde\nu_q=\pi_*\mu_{W'_1}$ and by Lemma~\ref{4.9} the measure $\mu_{W'_1}$ has almost surely the same topological support as $\mu$. Moreover, it follows from the theory of Mandelbrot measures \cite{B1999,Ba00} that, with probability 1, conditionally  on $\mu\neq 0$, for  $\mu_{W'_1}$-almost every $(x,y)$, we have
\begin{eqnarray*}
\lim_{n\to\infty}\frac{\mu([x_{|n},y_{|n}])}{-n\log(m)}&=&-\sum_{0\le i,j\le m-1}\E(W'_{1,i,j}\log_m (W_{i,j}))\\
&=&-\sum_{i=0}^{m-1}p'_i\log_m(p_i)-\sum_{i=0}^{m-1}p'_i\sum_{j=0}^{m-1}\E(V'_{1,i,j}\log_mV_{1,i,j})\\
&=& -\sum_{i=0}^{m-1}p'_i\log_m(p_i),
\end{eqnarray*}
since $V'_{1,i,j}=V_{i,j}$ and $0=\sum_{i=0}^{m-1}p'_iT'_i(1)=\sum_{i=0}^{m-1}p'_i\sum_{j=0}^{m-1}\E(V'_{1,i,j}\log_mV_{1,i,j}) $. Also,
\begin{eqnarray*}
\dim (\mu_{W'_1})&=&-\sum_{0\le i,j\le m-1}\E(W'_{1,i,j}\log_m (W'_{1,i,j}))\\
&=&-\sum_{i=0}^{m-1}p'_i\log_m(p'_i) -\sum_{i=0}^{m-1}p'_i\sum_{j=0}^{m-1}\E(V'_{1,i,j}\log_mV'_{1,i,j})\\
&=&-\sum_{i=0}^{m-1}p'_i\log_m(p'_i) +\sum_{i=0}^{m-1}p'_iT_i'(1)\\
&=&-\sum_{i=0}^{m-1}p'_i\log_m(p'_i)=\dim(\nu')=\dim (\E(\pi_*\mu_{W'_1})).
\end{eqnarray*}
Consequently, for $\pi_*\mu_{W'_1}$-almost every $x$, we have $\overline\dim_{\rm loc}(\pi_*\mu,x)\le -\sum_{i=0}^{m-1}p'_i\log_m(p_i)=\tau_\nu'(q)=\alpha$. Moreover, $\dim(\pi_*\mu_{W'_1})=\dim (\nu') =\tau_\nu^*(\tau'_\nu(q))=\tau^*(\alpha)$ (the last inequality coming from the equality $\tau_\nu(q)=\tau(q)$ and the fact that  $\alpha=\tau_\nu'(q)\in\{\tau'(q^+),\tau'(q^-)\}$. Then, the same arguments as in the proof of Corollary~\ref{4.6}  where $T$ is replaced by $\tau_\nu$ and $\mu_q$ by $\mu_{W'_1}$ yield $\dim_H E(\pi_*\mu,\alpha)\ge \tau^*(\alpha)$.
\end{proof}

\subsection{The case where $\alpha\in (\tau'(q^+),\tau'(q^-))$ when a first order phase transition occurs at $q\in(1,q_c)$} $\ $

Recall that in the case considered in this section we have $\tau_\nu(q)=T(q)$, $\tau_\nu'(q)\neq T'(q)$, and $(\tau'(q^+),\tau'(q^-))=(\tau_\nu'(q),T(q))$ or $(\tau'(q^+),\tau'(q^-))=(T'(q),\tau_\nu'(q))$.

Fix $\lambda\in [0,1]$. Let $(n_k)_{k\ge 1}$ be an increasing sequence of positive integers such that $n_k=o(n_1+\cdots n_{k-1})$ as $k\to\infty$, and  $n_1\min(\lambda,1-\lambda)>1$ if $\lambda>0$. For $k\ge 0$, let  $N_k=\sum_{i=1}^{k}n_i$ and $N_{k,\lambda}=N_{k-1}+\lfloor \lambda n_k\rfloor$. We will later further specify the sequence $(n_k)_{k\ge 1}$.

For each $n\ge 0$ and $(u,v)\in \Sigma_n\times\Sigma_n$, set
$$
\widetilde W_\lambda(u,v)=
\begin{cases}
W_q(u,v)=(p_i^qm^{T(q)}V_{i,j}^q)_{0\le i,j\le m-1}&\text{if }  N_{k-1}+1\le n\le N_{k,\lambda}\text{ for some $k$}\\
(p_{q,i} V_{i,j})_{0\le i,j\le m-1}&\text{otherwise}
\end{cases},
$$
where as previously $p_{q,i}=p_{i}^qm^{-\tau_\nu(q)}$. These random vectors can be used to build a non homogeneous Mandelbrot measure in the same way as $\mu$ and $\mu_q$: for each $n\ge 0$ and $(u,v)\in\Sigma_n\times\Sigma_n$, define
$$
\widetilde Y_\lambda(u,v)=\lim_{p\to\infty}\widetilde Y_{\lambda,p}(u,v), \text{ where }\widetilde Y_{\lambda,p}(u,v)=\sum_{|u'|=|v'|=p} \prod_{k=1}^p\widetilde W_{\lambda,u'_k,v'_k}(u\cdot u'_{|k-1},v\cdot v'_{|k-1}),
$$
and denote $\widetilde Y_\lambda(\epsilon,\epsilon)$ by $\widetilde Y_\lambda$ (each $\widetilde Y_\lambda(u,v)$ exists almost surely as limit of the non negative martingale (of expectation 1) $(\widetilde Y_{\lambda,p}(u,v))_{p\ge 0}$). Then,
$$
\widetilde\mu_\lambda([u]\times[v])=\widetilde Y_\lambda(u,v) \prod_{j=1}^n\widetilde W_{\lambda,u_j,v_j}(u_{|j-1},v_{|j-1})
$$
defines a measure almost surely. Moreover, the same argument as in Proposition~\ref{Kmu} shows that if $\widetilde\mu_\lambda$ is not equal to 0 almost surely, then  its topological support equals that of $\mu_\lambda$ almost surely. It is the situation which occurs as the following proposition shows. Also, $\widetilde\mu_1=\mu_q$, while $\widetilde\mu_0$ is a non degenerate Mandelbrot measure such that $\mathbb{E}(\pi_*\widetilde\mu_0)=\nu_q$ and by \eqref{dimrel} $\dim(\widetilde\mu_0)-\dim(\nu_q)=\sum_{i=0}^{m-1}p_{q,i}T_i'(1)$.

\begin{pro}\label{pro-4.13}
One has  $\mathbb E(\widetilde Y_\lambda)=1$; consequently $\widetilde\mu_\lambda$ is not almost surely degenerate, and with probability 1, conditionally on $\mu\neq 0$ we have $\mathrm{supp}(\widetilde\mu_\lambda)=\mathrm{supp}(\mu)$. Moreover, there exists $h\in (1,2]$ such that
$$
M(\lambda,h)=\sup\{\mathbb{E}(\widetilde Y_\lambda(u,v)^h): n\ge 0, \, u,v\in\Sigma_n\}<\infty.
$$
\end{pro}

We postpone the proof of Proposition~\ref{pro-4.13} for a while.

For all $k\ge 1$ and $(u,v)\in\Sigma_{N_k}\times\Sigma_{N_k}$, define
\begin{align*}
\widetilde \mu^{T}_1 (u,v)&=\prod_{i=1}^k \prod_{\ell=N_{i-1}+1}^{N_{i,\lambda}}   p_{u_\ell}^qm^{T(q)} V_{ u_\ell,v_\ell}(u_{|\ell-1},v_{|\ell-1})^q,\\
\mu^T(u,v)&=\prod_{i=1}^k \prod_{\ell=N_{i-1}+1}^{N_{i,\lambda}}   p_{u_\ell}V_{ u_\ell,v_\ell}(u_{|\ell-1},v_{|\ell-1}),\\
\widetilde \mu^{\tau_\nu}_0 (u,v)&=\prod_{i=1}^k  \prod_{\ell=N_{i,\lambda}+1}^{N_i}  p_{q,u_\ell}V_{u_\ell,v_\ell}(u_{|\ell-1},v_{|\ell-1}),\\
\mu^{\tau_\nu}(u,v)&=\prod_{i=1}^k  \prod_{\ell=N_{i,\lambda}+1}^{N_i}  p_{u_\ell}V_{u_\ell,v_\ell}(u_{|\ell-1},v_{|\ell-1}).
\end{align*}
We have
\begin{equation}
\label{decompmulambda}
\widetilde\mu_\lambda([u]\times[v])=\widetilde \mu^{T}_1 (u,v)\widetilde \mu^{\tau_\nu}_0 (u,v) \widetilde Y_\lambda(u,v)
\end{equation}
and
\begin{equation}
\label{decompmu}
\mu([u]\times[v])= \mu^{T} (u,v)\mu^{\tau_\nu} (u,v) Y(u,v).
\end{equation}
Define
$$
\alpha=\lambda T'(q)+(1-\lambda)\tau_\nu'(q)\quad\text{and} \quad \alpha'=\sum_{i=0}^{m-1}p_{q,i}T_i'(1).
$$
Since $\alpha\in [\tau'(q+),\tau'(q-)]$ and $\tau_\nu(q)=\tau(q)=T(q)$, we have
$$
\tau^*(\alpha)=\alpha q-\tau_\nu(q)=\lambda T^*(T'(q))+(1-\lambda)\tau_\nu^*(\tau_\nu'(q))
.
$$

We will prove the following propositions and corollary, which give the desired conclusion.

\begin{pro}\label{pro-4.14}
With probability 1, conditionally on $\mu\neq 0$, for $\widetilde\mu_\lambda$-almost every $(x,y)$, one has
\begin{align*}
\lim_{k\to\infty} \frac{\log(\widetilde \mu^T_1(x_{|N_k},y_{|N_k}))}{-N_k\log(m)}&=\lambda T^*(T'(q)),\\
 \lim_{k\to\infty}\frac{\log(\widetilde \mu^{\tau_\nu}_0(x_{|N_k},y_{|N_k}))}{-N_k\log(m)}&=(1-\lambda)\tau_\nu^*(\tau_\nu'(q))+(1-\lambda)\alpha'\\
 \lim_{k\to\infty} \frac{\log(\mu^T(x_{|N_k},y_{|N_k}))}{-N_k\log(m)}&=\lambda T'(q),\\
 \lim_{k\to\infty} \frac{\log(\mu^{\tau_\nu}(x_{|N_k},y_{|N_k}))}{-N_k\log(m)}&=(1-\lambda)\tau_\nu'(q)+(1-\lambda)\alpha',\\
 \lim_{k\to\infty} \frac{\log (\widetilde Y_\lambda(x_{|N_k},y_{|N_k}))}{-N_k\log(m)}&= \lim_{k\to\infty} \frac{\log (Y_(x_{|N_k},y_{|n}))}{-N_k\log(m)}=0;
 \end{align*}
in particular, $\dim_{\rm{loc}}(\mu,(x,y))=\alpha+(1-\lambda)\alpha'$ and $\dim_{\rm{loc}}(\widetilde\mu_\lambda,(x,y))=\tau^*(\alpha)+(1-\lambda)\alpha'$.

\end{pro}
We will see in the proof that $\alpha'\ge 0$.

\begin{pro}\label{pro-4.15} Suppose that $\lambda\in (0,1)$.
\begin{enumerate}
\item With probability 1, conditionally on $\mu\neq 0$,  one has $\underline{\dim}_{\rm{loc}}(\widetilde\mu_\lambda,x)\ge \tau^*(\alpha)$ for  $\pi_*\widetilde\mu_\lambda$-almost every $x$.

\item With probability 1, conditionally on $\mu\neq 0$,  one has both $\overline{\dim}_{\rm{loc}}(\widetilde\mu_\lambda,x)\le \tau^*(\alpha)$ and $\overline{\dim}_{\mathrm{loc}}(\pi_*\mu,x)\le \alpha$.
\end{enumerate}
\end {pro}
\begin{cor}\label{cor-4.16}
With probability 1, conditionally on $\mu\neq 0$, $\dim E(\pi_*\mu,\alpha)=\tau^*(\alpha)$.
\end{cor}

\begin{proof}[Proof of Proposition~\ref{pro-4.13}]

Let $h\in(1,2]$ and write
$$
\widetilde Y_{\lambda,p}(u,v)=\sum_{0\le i,j\le m-1} \widetilde W_{\lambda,i,j}(u,v) \widetilde Y_{\lambda,p-1}(ui,vj).
$$
We can use Kahane's original approach \cite{KP} to the moments of Mandelbrot martingales to write
$$
\widetilde Y_{\lambda,p}(u,v)^h\le\Big ( \sum_{0\le i,j\le m-1} \widetilde W_{\lambda,i,j}(u,v)^{h/2} \widetilde Y_{\lambda,p-1}(ui,vj)^{h/2}\Big )^2
$$
and then get
\begin{eqnarray*}
\mathbb{E} (\widetilde Y_{\lambda,p}(u,v)^h)&\le& \sum_{0\le i,j\le m-1} \mathbb{E}(\widetilde W_{\lambda,i,j}(u,v)^h) \mathbb{E} (\widetilde Y_{\lambda,p-1}(ui,vj)^h)\\
&&+\sum_{(i,j)\neq (i',j')}  \mathbb{E}(\widetilde W_{\lambda,i,j}(u,v)^{h/2}\widetilde W_{\lambda,i',j'}(u,v)^{h/2}).
\end{eqnarray*}
If $h$ is close enough to 1, there exists $C>0$ such that
$$
\sum_{(i,j)\neq (i',j')}  \mathbb{E}(\widetilde W_{\lambda,i,j}(u,v)^{h/2}\widetilde W_{\lambda,i',j'}(u,v)^{h/2})\le C$$ independently on $(u,v)$, by equidistribution of the $W(u,v)$ and the fact that our assumption on the domain of finiteness of $T$ we have $\mathbb{E}(W_{i,j}^{q'} )<\infty$ for all $q'<q_c$ and $0\le i,j\le m-1$.   Also, by construction $\mathbb{E} (\widetilde Y_{\lambda,p-1}(ui,vj)^h)$ does not depend on $(i,j)$. Thus, if $h<q_c$, we have
$$
\mathbb{E} (\widetilde Y_{\lambda,p}(u,v)^h)\le C+ \mathbb{E} (\widetilde Y_{\lambda,p-1}(u0,v0)^h)\sum_{0\le i,j\le m-1} \mathbb{E}(\widetilde W_{\lambda,i,j}(u,v)^h).
$$
By definition of $\widetilde W_{\lambda}(u,v)$, we have
$$\sum_{0\le i,j\le m-1} \mathbb{E}(\widetilde W_{\lambda,i,j}(u,v)^h)\in\Big \{\sum_{i=0}^{m-1}p_{q,i}^hm^{-T_i(h)}, m^{hT(q)-T(hq)}\Big\}.
$$
Since $T^*(T'(q))>0$ by our assumption $q\in (1,q_c)$, for $h$ close enough to 1 we have $hT(q)-T(hq)<0$, hence $m^{ hT(q)-T(hq)}<1$. On the other hand, since $\tau_\nu(q)=T(q)$ we have
$$
\psi(h):=\sum_{i=0}^{m-1}p_{q,i} m^{-T_i(h)}=\sum_{i=0}^{m-1}p_{i}^qm^{T(q)} m^{-T_i(h)},
$$
and $\psi (1)=\psi(q)=1$. Since $\psi$ is convex, it follows that if $h$ is taken in $(1,q)$, we have $\psi(h)\le 1$. Consequently, $\sum_{i=0}^{m-1}p_{q,i} ^hm^{-T_i(h)}\le \max\{p_{q,i}^{h-1}:0\le i\le m-1\}\psi(h)<1$, since all the positive $p_{q,i}$ belong to $(0,1)$. Notice in passing that since the derivative of $\psi$ at 1 is non positive, we have  $\alpha'=\sum_{i=0}^{m-1}p_{q,i} T'(1)\ge 0$.
Finally, if $h$ is close enough to 1,  there exists $c\in(0,1)$ independent of $(u,v)$ such that
$
\mathbb{E} (\widetilde Y_{\lambda,p}(u,v)^h)\le C+ c\, \mathbb{E} (\widetilde Y_{\lambda,p-1}(u0,v0)^h).
$
This  yields $\mathbb{E} (\widetilde Y_{\lambda,p}(u,v)^h)\le C \mathbb{E} (\widetilde Y_{\lambda,0}(u0^p,v0^p)^h)/(1-c)=C/(1-c)$, hence  both $\mathbb{E} (\widetilde Y_{\lambda}(u,v)^h)\le C/(1-c)$ and  $\mathbb{E} (\widetilde Y_{\lambda}(u,v))=1$.
\end{proof}

\begin{proof}[Proof of Proposition~\ref{pro-4.14}]  Define $\widetilde Q_\lambda(\mathrm{d}\omega, \mathrm{d}x,\mathrm{d}y)=\mathbb{P}(\mathrm{d}\omega)\widetilde\mu_{\lambda,\omega}(\mathrm{d}x,\mathrm{d}y)$, $\widetilde Q_1(\mathrm{d}\omega, \mathrm{d}x,\mathrm{d}y)=\mathbb{P}(\mathrm{d}\omega)\widetilde\mu_{1,\omega}(\mathrm{d}x,\mathrm{d}y)$ and $\widetilde Q_0(\mathrm{d}\omega, \mathrm{d}x,\mathrm{d}y)=\mathbb{P}(\mathrm{d}\omega)\widetilde\mu_{0,\omega}(\mathrm{d}x,\mathrm{d}y)$ the Peyri\`ere measures associated with $\widetilde\mu_\lambda$, $\widetilde\mu_1$ and $\widetilde\mu_0$ respectively. Also, set $\widetilde N_k=\sum_{i=1}^k\lfloor \lambda n_i\rfloor$ and $N'_k=N_k-\widetilde N_k$.

It is straightforward to write that under $\widetilde Q_\lambda$, the random vectors $\widetilde W_\lambda(x_{|n-1},y_{|n-1})$, $N_{k-1}+1\le n\le N_{k,\lambda}$, $k\ge 1$, are independent and equidistributed, with  the same law as the vectors $W_q(x_{|n-1},y_{|n-1})$, $n\ge 1$, with respect to $\widetilde Q_1(\mathrm{d}\omega, \mathrm{d}x,\mathrm{d}y)$. Moreover, since
$$
\mu_{1,n}([x_{|n}]\times [y_{|n}])=\prod_{k=1}^n W_{q,x_n,y_n}(x_{|n-1},y_{|n-1})),
$$
the strong law of large numbers yields
$$
\lim_{k\to\infty} \frac{\log (\mu_{1,\widetilde N_k}([x_{|\widetilde N_k}]\times [y_{|\widetilde N_k}]))}{-\widetilde N_k\log(m)}=-\mathbb{E}\sum_{0\le i,j\le m-1} p_i^qm^{T(q)}V_{i,j}^q\log_m(p_i^qm^{T(q)}V_{i,j}^q)=T^*(T'(q)),$$ $\widetilde Q_1$-almost surely. Since $\lim_{k\to\infty}\widetilde N_k/N_k=\lambda$, by definition of $\widetilde\mu_1^T(x_{|N_k},y_{|N_k}))$ we get the first claim.

The same idea applied with $ \mu^T(x_{|N_k},y_{|N_k})$ with respect to $\widetilde Q_\lambda$ and $\mu_{\widetilde N_k}([x_{|\widetilde N_k}]\times [y_{|\widetilde N_k}])$ with respect to $\widetilde Q_1$  yields
$$
\lim_{k\to\infty} \frac{\log (\mu^T(x_{|N_k},y_{|N_k}))}{-\widetilde N_k\log(m)}=-\mathbb{E}\sum_{0\le i,j\le m-1} p_i^qm^{T(q)}V_{i,j}^q\log_m(p_iV_{i,j})=T'(q),$$ $\widetilde Q_\lambda$-almost surely, i.e.
the third claim of the proposition since $\lim_{k\to\infty}\widetilde N_k/N_k=\lambda$.

For the second claim, one needs to consider $ \widetilde \mu_\lambda^{\tau_\nu}(x_{|N_k},y_{|N_k})$ and $\widetilde\mu_{0,N'_{k}}([x_{|N'_{k}}]\times [y_{|N'_{k}}])$ with respect to $\widetilde Q_\lambda$ and $\widetilde Q_0$ respectively; then one applyies the strong law of large numbers to $\log (\widetilde\mu_{0,N'_{k}}([x_{|N'_{k}}]\times [y_{|N'_{k}}]))/N'_k$ under $\widetilde Q_0$, and use the fact the $\lim_{k\to\infty} N'_k/N_k=1-\lambda$. The fourth claim follows similarly by considering $ \mu^{\tau_\nu}(x_{|N_k},y_{|N_k})$  and $\mu_{N'_{k}}([x_{|N'_{k}}]\times [y_{|N'_{k}}])$ with respect to to $\widetilde Q_\lambda$ and $\widetilde Q_0$ respectively.

For the last  two claims, an application of the Markov inequality shows that for any fixed $(u(k),v(k))$  in $\Sigma_{N_k}\times\Sigma_{N_k}$, for $Z\in \{Y,\widetilde Y_\lambda\}$ and $\gamma\in \{-1,1\}$, for any $\eta>0$ and $\epsilon>0$, one has
\begin{multline*}
\widetilde Q_\lambda(\{(x,y): \mathbf{1}_{\{Z(x_{|N_k},y_{|N_k})>0\}} Z(x_{|N_k},y_{|N_k})^{\gamma}>m^{N_k\epsilon})\\
\le m^{-N_k\eta\epsilon} \mathbb E \big ( \mathbf{1}_{\{Z(u(k),v(k))>0\}} \widetilde Y_\lambda (u(k),v(k)) Z^{\gamma \eta}(u(k),v(k))\big ).
\end{multline*}
Since conditionally on non vanishing $Y$ has finite negative moments, by Proposition~\ref{pro-4.13} and the H\"older inequality we can choose $\eta$ so that
$$
\sup\{\mathbb E \big ( \mathbf{1}_{\{Z(u(k),v(k))>0\}} \widetilde Y_\lambda (u(k),v(k)) Z^{\gamma \eta}(u(k),v(k))\big ):k\ge 1, \ Z\in\{Y,\widetilde Y_\lambda\}\}<\infty.
$$
Consequently
$$
\sum_{k\ge 1}\widetilde Q_\lambda(\{(x,y): \mathbf{1}_{\{Z(x_{|N_k},y_{|N_k})>0\}} Z(x_{|N_k},y_{|N_k})^{\gamma}>m^{N_k\epsilon})<\infty,
$$
and the desired claims follow from the Borel-Cantelli lemma.

Finally, the claim about the local dimensions follows from~\eqref{decompmulambda} and  \eqref{decompmu}, and the fact that $\lim_{k\to\infty}N_{k-1}/N_k=1$.
\end{proof}

\begin{proof}[Proof of Proposition~\ref{pro-4.15}(1)] We will use the following lemma.

\begin{lem}\label{lem-4.15}
We have
$$
\mathbb{E}\sum_{|u|=N_k} \pi_*\widetilde\mu_\lambda ([u])^{h}=O(N_k m^{N_k(-\tau^*(\alpha)(h-1)+ o(h-1)+O(k/N_k))})$$
as $h\to 1+$.
\end{lem}

We deduce from the previous lemma that for all $\epsilon>0$, for $h$ close enough to $1+$ we have $\mathbb{E}\sum_{k\ge 1}\sum _{|u|=N_k} m^{N_k (h-1)(\tau^*(\alpha)-\epsilon)}\pi_*\widetilde\mu_\lambda ([u])^{h}<\infty$.

This implies that with probability 1, conditionally on $\mu\neq 0$, for all $\epsilon>0$, there exists $h>1$ such that $\sum_{k\ge 1}\sum _{|u|=N_k} m^{N_k (h-1)(\tau^*(\alpha)-\epsilon)}\pi_*\widetilde\mu_\lambda ([u])^{h}<\infty$. Due to Lemma~\ref{lemma2.3} and the fact that $\lim_{k\to\infty}N_k/N_{k-1}=1$,  we get  $\underline\dim(\widetilde \mu_\lambda)\ge \tau^*(\alpha)$.

If we were able to prove that the same estimate as in the lemma holds for $h$ near $1-$, we could derive the second part of the proposition quite easily (but maybe such a bound does not hold). We have to use another approach (see below).
\end{proof}

\begin{proof}[Proof of Lemma~\ref{lem-4.15}]
 If $k\ge 1$ and $u\in\Sigma_{N_k}$, by definition of $\widetilde\mu_\lambda$ we have
\begin{multline*}
\pi_*\widetilde\mu_\lambda ([u])=\sum_{|v|=N_k}  \widetilde Y_\lambda (u,v)
\prod_{i=1}^k \Big (\prod_{\ell=N_{i-1}+1}^{N_{i,\lambda}} p_{u_\ell}^qm^{T(q)} V_{ u_\ell,v_\ell}(u_{|\ell-1},v_{|\ell-1})^q \Big )\\
\cdot \Big (\prod_{\ell'=N_{i,\lambda}+1}^{N_i} p_{q,u_\ell'}V_{u_{\ell'},v_{\ell'}}(u_{|\ell'-1},v_{|\ell'-1})\Big ) .
\end{multline*}
Setting, for  $k\ge 1$ and $h>1$ such that $hq<q_c$ (recall that $q$ is fixed)
 \begin{equation}\label{Lambda}
 \Lambda(k,h)=\big (\sum_{i=1}^{k} \lfloor \lambda n_i\rfloor\big ) (T(hq)-hT(q))+ \big (N_k- \sum_{i=1}^{k} \lfloor \lambda n_i\rfloor\big ) (\tau_\nu(hq)-h\tau_\nu(q)),
\end{equation}
 we can write
 \begin{align*}
m^{\Lambda(k,h)}\pi_*\widetilde\mu_\lambda ([u])^{h}= Z(u)^h\,  \prod_{i=1}^k \Big (\prod_{\ell=N_{i-1}+1}^{N_{i,\lambda}}   p_{u_\ell}^{hq}m^{T(hq)}m^{-hT_{u_\ell}(q)}\Big ) \Big (\prod_{\ell'=N_{i,\lambda}+1}^{N_i} p_{hq,u_{\ell'}}\Big ),
\end{align*}
where
 \begin{multline*}
 Z(u)=\sum_{|v|=N_k}  \widetilde Y_\lambda (u,v)
\prod_{i=1}^k \Big (\prod_{\ell=N_{i-1}+1}^{N_{i,\lambda}} m^{T_{u_{\ell}}(q)}V_{ u_\ell,v_\ell}(u_{|\ell-1},v_{|\ell-1})^q\Big )\\
\cdot \Big (\prod_{\ell'=N_{i,\lambda}+1}^{N_i} V_{u_{\ell'},v_{\ell'}}(u_{|\ell'-1},v_{|\ell'-1})\Big ).
\end{multline*}

 Fix $h\in(1,2]$ as in Proposition~\ref{pro-4.13} such that $M(\lambda,h)<\infty$ and set (remind that $q$ is fixed)
$$
C_1(h)=\max_{0\le i\le m-1}\sum_{0\le j\neq j'\le m-1} \mathbb E(m^{T_i(q) h/2}V_{i,j}^{hq/2}m^{T_i(q) h/2}V_{i,j'}^{hq/2})
$$
and
$$
C_2(h)= \max_{0\le i\le m-1}\sum_{0\le j\neq j'\le m-1} \mathbb E(V_{i,j}^{h/2}V_{i,j'}^{h/2}).
$$
Taking $h$ closer to 1 if necessary we have $C(h)=\max (C_1(h),C_2(h))<\infty$.
We notice that $\mathbb E(Z(u))=1$, and we can use the same approach as in Section~\ref{pmeXu} to estimate the positive moments of $X(u)$ to get

$$
\mathbb{E}(Z(u)^h)\le M(\lambda,h) C(h)\sum_{\ell=1}^{N_k} m^{-(\theta^{(1)}_{u_1}+\cdots +\theta^{(\ell)}_{u_\ell})},
$$
where $\theta^{(\ell)}_i=hT_{i}(q)-T_{i}(hq)$ if $N_{j-1}+1\le \ell\le N_{j,\lambda}$ for some $j$ and  $\theta^{(\ell)}_i= T_{i}(h)$ otherwise. It follows that, if we set $\widetilde p^{(\ell)}_{i}=p_{i}^{hq}m^{T(hq)}m^{-hT_{i}(q)}$  whenever$N_{j-1}+1\le \ell\le N_{j,\lambda}$ for some $j$ and $\widetilde p^{(\ell)}_{i}=p_{hq,i}$ otherwise,  then
\begin{align*}
&m^{\Lambda(k,h)}\mathbb{E}\sum_{|u|=N_k} \pi_*\widetilde\mu_\lambda ([u])^{h}\\
&\le M(\lambda,h) C(h)\sum_{\ell=1}^{N_k} \sum_{|u|=N_k}m^{-(\theta^{(1)}_{u_1}+\cdots +\theta^{(\ell)}_{u_\ell})}\prod_{j=1}^{N_k} \widetilde p^{(j)}_{u_j}\\
&=M(\lambda,h) C(h)\sum_{\ell=1}^{N_k} \Big( \prod_{j=1}^\ell \sum_{i=0}^{m-1}\widetilde p^{(j)}_im^{-\theta^{(j)}_i}\Big )\Big (\prod_{j'=\ell+1}^{N_k} \sum_{i'=0}^{m-1}\widetilde p^{(j')}_{i'}\Big ).
\end{align*}

We have for each $0\le j\le m-1$, either $\sum_{i=0}^{m-1}\widetilde p^{(j)}_i=\sum_{i=0}^{m-1}p_{hq,i}=1$ or $\sum_{i=0}^{m-1}\widetilde p^{(j)}_i=\sum_{i=0}^{m-1} p_{i}^{hq}m^{T(hq)}m^{-hT_{i}(q)}$. On the other hand,   the computations achieved in the proof of Lemma~\ref{muq/espmuq}
show that the derivative of $h\mapsto \sum_{i=0}^{m-1} p_{i}^{hq}m^{T(hq)}m^{-hT_{i}(q)}$ equals $\log(m)(\dim(\mu_q)-\dim(\mathbb{E}(\pi_*\mu_q)))\le 0$. So $\sum_{i=0}^{m-1}\widetilde p^{(j)}_i\le 1+o(h-1)$.

On the other hand, we have $\sum_{i=0}^{m-1}\widetilde p^{(j)}_im^{-\theta^{(j)}_i}=\sum_{i=0}^{m-1} p_{i}^{hq}m^{T(hq)}m^{-T_{i}(hq)}=1$ or $\sum_{i=0}^{m-1}\widetilde p^{(j)}_im^{-\theta^{(j)}_i}=\sum_{i=0}^{m-1}p_{hq,i}m^{-T_i(h)}$, and  the derivative at 1 of $h\mapsto \sum_{i=0}^{m-1}p_{hq,i}m^{-T_i(h)}$ equals $-\log(m)\sum_{i=0}^{m-1} p_{q,i}T'_i(1)$ which is non positive by a remark made in the proof of Proposition~\ref{pro-4.13}. So $\sum_{i=0}^{m-1}\widetilde p^{(j)}_im^{-\theta^{(j)}_i}\le 1+o(h-1)$.

Finally,
$$
m^{\Lambda(k,h)}\mathbb{E}\sum_{|u|=N_k} \pi_*\widetilde\mu_\lambda ([u])^{h}=O(N_k m^{o(h-1) N_k}).
$$
Since it is easily seen from \eqref{Lambda} that
$$
\Lambda(k,h)=N_k(\lambda\, T^*(T'(q))+(1-\lambda)\tau_\nu^*(\tau_\nu'(q)))(h-1)+ N_k\, o(h-1)+O(k)
$$
and  we know that $\tau^*(\alpha)=\lambda\, T^*(T'(q))+(1-\lambda)\tau_\nu^*(\tau_\nu'(q))$, we get the desired conclusion.
\end{proof}

\begin{proof}[Proof of Proposition~\ref{pro-4.15}(2)] If $\alpha'=0$, the result  directly follows from Proposition~\ref{pro-4.14} since projecting does not increase the upper local dimensions.

Suppose now that $\alpha'>0$. Recall that $N'_k=\sum_{i=1}^k n_i-\lfloor \lambda n_i\rfloor$. Conditionnaly on $\mu\neq 0$, the behavior of $\widetilde\mu^{\tau_\nu}_0(x_{|N_k},y_{|N_k})$, $\widetilde\mu_\lambda$-almost everywhere,  is the same as that of $\widetilde \mu_{0,N'_k}([x_{|N'_k}]\times [y_{|N'_k}])$, $\widetilde \mu_0$-almost everywhere. Moreover,  we deduce from Theorem~\ref{DIM}(2) and the proof of Proposition~\ref{pro-4.14}  that for $\pi_*\widetilde\mu_0$-almost every $x$, for $\widetilde\mu_0^x$-almost every $y$, we have both
\begin{eqnarray}
\label{lim1}
\lim_{k\to\infty} \frac{\log (\widetilde \mu_{0,N'_k}([x_{|N'_k}]\times [y_{|N'_k}]))}{-N'_k\log(m)}&=&\tau_\nu^*(\tau_\nu'(q))+\alpha'\\
\label{lim2}\text{and}\quad \lim_{k\to\infty} \frac{\log(\widetilde\mu_0^x([y_{|N'_k}]))}{-N'_k\log(m)}&=&\alpha'.
\end{eqnarray}
In particular, if we denote by $E$ a set of full $\widetilde\mu_0$-measure such that \eqref{lim1} holds for all $(x,y)\in E$, due to the exact dimensionality \eqref{lim2} of $\widetilde\mu_0^x$ we can find a subset $E'$ of $E$ of full  $\widetilde\mu_0$-measure such that in addition for $\pi_*\widetilde\mu_0$-almost every $x\in \pi(E')$, we have
$$
\lim_{k\to\infty}\frac{\log \#\{v\in\Sigma_{N'_k}: [x_{|N'_k}]\times [v]\cap E\neq\emptyset\}}{N'_k\log(m)}=\alpha'.
$$

Now we can transfer these properties to $\widetilde\mu_\lambda$. We can find two sets $\widetilde {E}'\subset \widetilde E$ of full $\widetilde\mu_\lambda$-measure such that for all $(x,y) \in \widetilde E $ we have
$$
\lim_{k\to\infty} \frac{\log (\widetilde \mu_{0}^{\tau_\nu}(x_{|N_k},y_{|N_k}))}{-N'_k\log(m)}=\tau_\nu^*(\tau_\nu'(q))+\alpha'
$$
and for all $x\in \pi(\widetilde{E}')$,
$$
\lim_{k\to\infty}\frac{\log \#\{v\in\Sigma_{N_k}: [x_{|N_k}]\times [v]\cap \widetilde E\neq\emptyset\}}{N'_k\log(m)}=\alpha'.
$$
Due to Proposition~\ref{pro-4.14}, we can also assume that for all $(x,y)\in \widetilde E$ we have
$$
\lim_{k\to\infty} \frac{\log(\widetilde \mu^T_1(x_{|N_k},y_{|N_k}))}{-N_k\log(m)}=\lambda T^*(T'(q))\quad\text{and}\quad \lim_{k\to\infty} \frac{\log(\widetilde Y_\lambda(x_{|N_k},y_{|N_k}))}{N_k\log (m)}=0.
$$
Set  $\beta_q=T^*(T'(q))$, $\widetilde \beta_q=\tau_\nu^*(\tau_\nu'(q))$, and for $N\ge 1$ and $\epsilon>0$ set
$$\widetilde E_{N,\epsilon}=\left \{(x,y): \forall\, k\ge N,
\begin{cases}
m^{-N'_k(\widetilde \beta_q+\alpha'-\epsilon)}\ge \widetilde \mu_{0}^{\tau_\nu}(x_{|N_k},y_{|N_k}))\ge m^{-N'_k(\widetilde \beta_q+\alpha'+\epsilon)},\\
m^{-N_k (\lambda \beta_q -\epsilon)}\ge   \widetilde \mu^T_1(x_{|N_k},y_{|N_k}))\ge m^{-N_k (\lambda \beta_q +\epsilon)},\\
 m^{N_k\epsilon} \widetilde Y_\lambda(x_{|N_k},y_{|N_k}))\ge  m^{-N_k\epsilon}
  \end{cases}
\right\}.
$$
The previous properties can be precised as follows: for $\pi_*\widetilde\mu_\lambda$-almost every $x$, for all $\epsilon>0$, there exists $N\ge 1$ such that for $k\ge N$ there are at least $m^{N'_k (\alpha'-\epsilon)}$ words $v\in\Sigma_{N_k}$ such that $[x_{|N'_k}]\times[v]\cap \widetilde E_{N,\epsilon}\neq\emptyset$, so due to \eqref{decompmulambda}
$$
\widetilde \mu_\lambda([x_{|N_k}]\times [v])\ge m^{-N_k (\lambda T^*(T'(q)) +\epsilon)} m^{-N'_k (\tau_\nu^*(\tau_\nu'(q))+\alpha'+\epsilon)} m^{-N_k\epsilon}.
$$
Consequently
$$
\pi_*\widetilde\mu_\lambda([x_{|N_k}])\ge m^{-N_k \lambda T^*(T'(q)) -N'_k \tau_\nu^*(\tau_\nu'(q))} m^{-(2N_k+2N'_k)\epsilon}.
$$
Since $\lim_{k\to\infty} N'_k/N_k=1-\lambda$ and $\lim_{k\to\infty} N_{k-1}/N_k=1$, we can conclude that
$$
\overline\dim_{\rm{loc}}(\pi_*\widetilde\mu_\lambda,x)\le \lambda T^*(T'(q))+(1-\lambda)\tau_\nu^*(\tau_\nu'(q))+4\epsilon=\tau^*(\alpha) +4\epsilon,
$$
for all $\epsilon>0$. This yields $\overline\dim_{\rm{loc}}(\pi_*\widetilde\mu_\lambda,x)\le \tau^*(\alpha)$ for $\widetilde\mu_\lambda$-almost every $x$, and similar arguments using again Theorem~\ref{DIM}(1) and the information provided by Proposition~\ref{pro-4.14} about $\mu$ as well as \eqref{decompmu} yield $\overline\dim_{\rm{loc}}(\pi_*\mu,x)\le \alpha$.
\end{proof}

\begin{proof}[Proof of Corollary~\ref{cor-4.16}] Due to Proposition~\ref{pro-4.15}, we can use an argument similar to that used in the proof of Corollary~\ref{4.6}.
\end{proof}

\subsection{The case $\alpha=\tau'(0+)$} We distinguish the three cases of Proposition~\ref{tau'0}.

Notice that by the results obtained in the previous sections we know that $\tau_{\pi_*\mu}=\tau$ over $[0,\widetilde q_c)$ conditionally on $\mu\neq 0$. In particular, $\tau_{\pi_*\mu}'(0+)=\tau'(0+)$.

{\bf (i) $\tau=T$ near $0+$.} In this case, we have $\tau'(0+)=T'(0)$, and by continuity the property $\dim(\mu_q)\le \dim(\E(\mu_q))$ which holds near $0+$ by Lemma~\ref{muq/espmuq} extends to the Mandelbrot measure $\mu_0$. Also,  the approach developed in Section~\ref{tauegalpsi} still applies to give $\dim_H E(\pi_*\mu,T'(0))\ge \tau^*(T'(0))$.

\medskip

\bf (ii)  $\tau=\tau_\nu$ near $0+$.} We have $\tau'(0+)=\tau_\nu'(0)$. Let $p'=(p'_i)_{0\le i\le m-1}$ be defined as in \eqref{e-p'} and recall that $\nu'$ is the Bernoulli product associated with $p'$. Since we have $\sum_{i=0}^{m-1}p'_iT_i^*(T_i'(1))\ge 0$,  the approach used in  Section~\ref{taudifpsi} when $s(q)=1$ still works and shows that conditionally on $\mu\neq 0$,  $\dim_{\rm loc}(\pi_*\mu,x)=-\sum_{i=0}^{m-1}p_i'\log_m(p_i)=\tau_\nu'(0)$, either at $\nu'$-almost every $x\in \pi(K)$, or at $\pi_*\mu_{W'_1}$-almost every~$x$ if $\sum_{i=0}^{m-1}p'_iT_i^*\circ T_i'$ equals 0 over $[0,1]$ ($\mu_{W'_1}$ is the Mandelbrot measure associated with $p'$ and the vectors $V'_{1,i}$ defined in~\eqref{V'i}). Moreover,  by definition of the vector $p'$ we have  $\dim(\nu')=\dim_H (\pi(K))=-\tau(0)=\tau^*(\tau_\nu'(0))$ in the first case and  $\dim(\nu')=\dim(\pi_*\mu_{W'_1})=\dim_H (\pi(K))=-\tau(0)=\tau^*(\tau_\nu'(0))$ in the second case. This yields $\dim_H E(\pi_*\mu,\tau_\nu'(0))\ge \tau^*(\tau_\nu'(0))$. We notice that in the second case $\mu_{W'_1}$ coincides with the measure $\mu'$ considered in the proof of Corollary~\ref{DGF}.

\medskip

 {{\bf (iii) $\tau>\max(\tau_\nu, T)$ near $0+$.} Using the notations of Proposition~\ref{tau'0}, we see that if $s_0>0$ we are exactly in the same situation as in Section~\ref{taudifpsi}, with in addition the fact that $\sum_{i=0}^{m-1}p'_iT_i^*(T'_i(1))>0$ is excluded if $s_0=1$. This yields $\dim_H E(\pi_*\mu,\tau'(0+))\ge \tau^*(\tau'(0+))$ in this case.}

If $s_0=0$, consider the Mandelbrot measure $\mu_{W'_0}$ associated with $p'$ and the vectors $V'_{0,i}$ defined in~\eqref{V'i}.  Using the theory of Mandelbrot measures  \cite{B1999,Ba00} here again yields, with probability 1, conditionally on $\mu\neq 0$, for  $\mu_{W'_0}$-almost every $(x,y)$,
\begin{eqnarray*}
\lim_{n\to\infty}\frac{\mu([x_{|n},y_{|n}])}{-n\log(m)}&=&-\sum_{0\le i,j\le m-1}\E(W'_{0,i,j}\log_m (W_{i,j}))\\
&=& -\sum_{i=0}^{m-1}p'_i\log_m(p_i) -\sum_{i=0}^{m-1}p'_i\sum_{j=0}^{m-1}\E(V'_{0,i,j}\log_mV_{i,j})\\
&=& -\sum_{i=0}^{m-1}p'_i\log_m(p_i) +\sum_{i=0}^{m-1}p'_iT_i'(0)=\tau'(0+).
\end{eqnarray*}
Also,
\begin{eqnarray*}
\dim (\mu_{W'_0})&=&-\sum_{0\le i,j\le m-1}\E(W'_{0,i,j}\log_m (W'_{0,i,j}))\\
&=&-\sum_{i=0}^{m-1}p'_i\log_m(p'_i) -\sum_{i=0}^{m-1}p'_i\sum_{j=0}^{m-1}\E(V'_{0,i,j}\log_mV'_{0,i,j})\\
&=&-\sum_{i=0}^{m-1}p'_i\log_m(p'_i) -\sum_{i=0}^{m-1}p'_iT_i(0)\\
&=&-\sum_{i=0}^{m-1}p'_i\log_m(p'_i)=\dim(\nu')=\dim (\mathbb{E}(\pi_*\mu_{W'_0}))
\end{eqnarray*}
(notice that this time $\mu_{W'_0}$ is here again  the Mandelbrot measure $\mu'$ considered in the proof of Corollary~\ref{DGF}). Consequently, for $\pi_*\mu_{W'_0}$-almost every $x$, we have $\overline\dim_{\rm loc}(\pi_*\mu,x)\le \tau'(0+)$, and $\dim(\pi_*\mu_{W'_0})=\dim (\nu')=\tau(0)=\tau^*(\tau'(0+))$. Then, an argument similar to that used in the proof of Corollary~\ref{4.6} again yields the desired conclusion.

\section{Moment estimates}\label{MEST}

We start by establishing two basic lemmas on concave functions in Section~\ref{sec-lem}.  Then Sections~\ref{pmeXu} and~\ref{nmeXu} respectively provide positive moments and negative moments estimates for $X(x_{|n})$ with respect to $\mathbb{P}\otimes \eta$, where $\eta$ is a Bernoulli product.

\subsection{Lemmas}\label{sec-lem}
We begin with an elementary observation.
\begin{lem}
\label{lem-1.1}
Let $q>1$ and $f:[1,q]\to \R$ be a continuous concave function with $f(1)=0$.  Let $k\in \N$.
Suppose that $q_1,\ldots, q_k\geq 1$ with $\sum_{i=1}^k q_k\leq q$. Then
\begin{itemize}
\item[(i)] $\sum_{i=1}^k f(q_i)\geq  f(q)$ provided that  $\sum_{i=1}^k f(q_i)\leq 0$;
\item[(ii)]
$\sum_{i=1}^k f(q_i)\geq \min \{0, f(q)\}$.
\end{itemize}
\end{lem}
\begin{proof}
Clearly (ii) follows from (i). To prove (i), assume that  $\sum_{i=1}^k f(q_i)\leq 0$. We show below that  $\sum_{i=1}^k f(q_i)\geq  f(q)$.

Set $\displaystyle t_i= \frac{f(q_i)-f(1)}{q_i-1}=\frac{f(q_i)}{q_i-1}$ for $1\leq i\leq k$, and $\displaystyle t=\frac{f(q)}{q-1}$. By concavity we have
$t\leq t_i$ for every  $1\leq i\leq k$. Since  $\sum_{i=1}^k f(q_i)\leq 0$, we have $t_i=f(q_i)/(q_i-1)\leq 0$ for some~$i$, and thus $t\leq t_i\leq 0$.  Therefore
\begin{equation*}
\begin{split}
\sum_{i=1}^k f(q_i)=\sum_{i=1}^k t_i(q_i-1)&\geq \sum_{i=1}^k t (q_i-1)\\
&\geq t (q-k)\\
&\geq t(q-1)=f(q).
\end{split}
\end{equation*}
\end{proof}

\begin{rem}{\rm
Under the condition of  Lemma \ref{lem-1.1}, it is possible that $0<\sum_{i=1}^k f(q_i)< f(q)$; for instance letting $f(x)=x-1$, $q_1=2$ and $q=3$, we have $0<f(q_1)<f(q)$.
}\end{rem}
\begin{lem} \label{lem-1.2}
Let $q>1$ and $f_1,\ldots, f_m$ be continuous concave functions defined on $[1,q]$  satisfying $f_j(1)=0$ for $1\leq j\leq m$. Let $(p_1', \dots, p_m')$ be a probability vector.
Suppose that $q_1,\ldots, q_k\geq 1$ with $\sum_{i=1}^k q_k\leq q$. Then
\begin{equation}
\label{e-0}
\sum_{j=1}^m p_j' m^{- \sum_{i=1}^k f_j(q_i)}\leq \max\left\{1,  \sum_{j=1}^m p_j' m^{- f_j(q)}\right\}.
\end{equation}
Moreover, if $ \sum_{j=1}^m p_j' m^{- f_j(q)}<1$, then $\sum_{j=1}^m p_j' m^{- \sum_{i=1}^k f_j(q_i)}<1$.
\end{lem}

\begin{rem}
{\rm Under the condition of  Lemma \ref{lem-1.2}, it is possible that
$$1>\sum_{j=1}^m p_j' m^{- \sum_{i=1}^k f_j(q_i)}>\sum_{j=1}^m p_j' m^{- f_j(q)}.$$
 For instance letting $f(x)=x-1$, $q_1=2$ and $q=3$,  we have
$1>m^{-f(q_1)}>m^{-f(q)}$. 
}\end{rem}

\begin{proof}[Proof of Lemma \ref{lem-1.2}]
 We first show that
\begin{equation}
\label{e-1}
\sum_{j=1}^m p_j' m^{- \sum_{i=1}^k f_j(q_i)}\leq 1+ \sum_{j=1}^m p_j' m^{- f_j(q)}.
\end{equation}
Set $\Lambda=\{1\leq j\leq m:\; \sum_{i=1}^k f_j(q_i)<0\}$. By Lemma \ref{lem-1.1}, we have
$\sum_{i=1}^k f_j(q_i)\geq f_j(q)$ for each $j\in \Lambda$. Hence we have
\begin{equation*}
\begin{split}
\sum_{j=1}^m p_j' m^{- \sum_{i=1}^k f_j(q_i)}& \leq 1+\sum_{j\in \Lambda}p_j' m^{- \sum_{i=1}^k f_j(q_i)}\\
&\leq 1+ \sum_{j\in \Lambda}p_j' m^{- f_j(q)}\\
&\leq 1+\sum_{j=1}^m p_j' m^{- f_j(q)}.
\end{split}
\end{equation*}
This proves \eqref{e-1}.

Next we show that
\begin{equation}
\label{e-2}
\left( \sum_{j=1}^m p_j' m^{- \sum_{i=1}^k f_j(q_i)}\right)^n \leq 1+ \left(\sum_{j=1}^m p_j' m^{- f_j(q)}\right)^n
\end{equation}
for any $n\in \N$, from which \eqref{e-0} follows.
Indeed setting $p_{j_1\ldots j_n}'=p_{j_1}'\ldots p_{j_n}'$ and $f_{j_1\ldots j_n}=f_{j_1}+\ldots+f_{j_n}$,  then \eqref{e-2} can be re-written as
$$
\sum_{1\leq j_1,\ldots, j_n\leq m} p_{j_1\ldots j_n}' m^{- \sum_{i=1}^k f_{j_1\ldots j_n}(q_i)}\leq 1+ \sum_{1\leq j_1,\ldots, j_n\leq m} p_{j_1\ldots j_n}' m^{-  f_{j_1\ldots j_n}(q)};
$$
but this is just the application of \eqref{e-1} to the probability weight $(p_{j_1\ldots j_n}')$ and the concave functions $f_{j_1\ldots j_n}$. This finishes the proof of \eqref{e-0}.

In the end, assume that $\sum_{j=1}^m p_j' m^{- f_j(q)}<1$. By \eqref{e-0}, $\sum_{j=1}^m p_j' m^{- \sum_{i=1}^k f_j(q_i)}\leq 1$.  We need to show that the inequality is strict.  Suppose on the contrary that $$\sum_{j=1}^m p_j' m^{- \sum_{i=1}^k f_j(q_i)}=1.$$
Define $g(x)=\sum_{j=1}^m p_j' m^{- x\sum_{i=1}^k f_j(q_i)}$ for $x\in \R$. Then $g$ is convex.  Notice that on a small neighborhood $U$ of $1$,
we have $\sum_{j=1}^m p_j' m^{- xf_j(q)}<1$ for $x\in U$.  For any fixed $x\in U$, applying \eqref{e-0} to the functions $xf_i$, we obtain that
$g(x)\leq 1$. Hence $g$ takes a local maximum at $x=1$. However $g$ is convex and analytic on $\R$, it follows that  $g$ is constant  on $\R$ and therefore
 $$
 \sum_{i=1}^k f_j(q_i)=0
 $$
 for any $1\leq j\leq m$. Then by Lemma \ref{lem-1.1}(i), we have $f_j(q)\leq 0$ for all $1\leq j\leq m$, which contradicts the assumption
 that $\sum_{j=1}^m p_j' m^{- f_j(q)}<1$. This finishes the proof of the lemma.
\end{proof}

\subsection{Positive moments estimates for $X_n$}\label{pmeXu}

Let us first recall  some notations.  We are given $W=(W_{i,j})_{0\leq  i,j \leq m-1}$, a non-negative random vector  with $\E(\sum_{i,j}W_{i,j})=1$.
Let $q>1$ and assume that $\E(\sum_{i,j}W_{i,j}^q)<\infty$.  Set
$p_i=\E(\sum_j W_{i,j})$.   Set
$$V_{i,j}=\left\{ \begin{array}{ll}
W_{i,j}/p_i, & \mbox{ if }p_i\neq 0,\\
{1}/{m}, & \mbox{ if }p_i=0.
\end{array}
\right.
$$

 For $t\in [0,q]$, set
$$T(t)=-\log_m \E(\sum_{i,j}W_{i,j}^t),\quad T_i(t)=-\log_m \E(\sum_{j}V_{i,j}^t).$$
Then $T$ and $T_i$ ($0\leq i\leq m-1$) are well defined continuous concave functions on $[0,q]$, with $T(1)=T_i(1)=0$.
Set $\Sigma=\{0,1,\ldots, m-1\}^\N$. Let $\mu$ be the (random) Mandelbrot measure on $\Sigma\times \Sigma$ generated by $W$.  Set $Y=\|\mu\|$ to be the total mass of $\mu$ and assume that $T(q)>0$.  By Kahane-Peyriere~\cite{KP} and Durrett-Liggett~\cite{DL},  this is equivalent to the property that $0<\E(Y^q)<\infty$.  For  each
$(u,v)\in (\Sigma\times \Sigma)_*$, let  $Y(u,v)$ be defined as in \eqref{Yuv}.  We defined in Section~\ref{pfDIM}
\begin{equation}
\label{e-3}
X(u)=\sum_{v\in \Sigma_{|u|}} Y(u,v) \prod_{j=1}^{|u|} V_{u_j, v_j}(u|_{j-1}, v|_{j-1}),\quad u\in \Sigma_*.
\end{equation}
and $X_n(x)=X(x_{|n})$ for all $x\in\Sigma$ and $n\ge 1$.

Given any Bernoulli product $\eta$ on $\Sigma$ generated by a probability vector $(p'_0,\ldots,p'_{m-1})$, we are seeking for estimates of $\E_{\mathbb{P}\otimes\eta}(X_n^q)$, i.e. $\sum_{|u|=n} \eta([u]) \E(X(u)^q)$.

For short we write $V_{u_1,v_1}=V_{u_1, v_1}(\epsilon,\epsilon)$ and
$$X_1(u,j)= \sum_{v\in \Sigma_{|u|}:\; v_1=j} Y(u,v) \prod_{k=2}^{|u|} V_{u_k, v_k}(u|_{k-1}, v|_{k-1}),\quad j=0,\ldots, m-1.$$
Then we have
\begin{equation}
\label{e-4}
X(u)=\sum_{j=0}^{m-1} V_{u_1,j} X_1(u,j).
\end{equation}

We emphasize that $X_1(u,j)$ ($j=0,\ldots, m-1$) are independent copies of $X(\sigma u)$.  Moreover, they are independent of $V_{u_1, j'}$ ($j'=0,\ldots, m-1$).

By \eqref{e-3} and the assumption that  $\E(Y^q)<\infty$, we have $\E(X(u)^q)<\infty$ for each $u\in \Sigma_*$. In particular, $\E(X(u))=1$.

For $n\in \N$, set
$$
\R^n_\leq =\{(x_1,\ldots, x_n)\in \R^n:\; x_1\leq \cdots\leq x_n\}.
$$

\begin{lem}
\label{lem-2.1}
Let $t\in (1,q]$ and $u\in \Sigma_*$.  
Then
\begin{itemize}
\item[(i)] $\E(X(u)^t)\geq m^{-T_{u_1}(t)} \E(X(\sigma u)^t)$.
\item [(ii)] There exists a positive constant $C$ (depending on $q$) such that
\begin{equation}
\label{e-5}
\E(X(u)^t)\leq m^{-T_{u_1}(t)} \E(X(\sigma u)^t)+ C+C\sum_{(q_1,\ldots, q_s)\in {\mathcal I}_t} \prod_{j=1}^s
\E\left(X(\sigma u)^{q_j}\right),
\end{equation}
 where $\lceil t \rceil$ denotes the smallest integer $\geq t$, and

 \begin{equation}
 \label{e-6}
 \begin{split}
 {\mathcal I}_t: =& \left\{  \left(\frac{k_1t}{\lceil t \rceil},\ldots, \frac{k_st}{\lceil t \rceil}\right)\in \R^s_\leq :\; s,  k_i\in \N\cap[2,\infty),\;  \sum_{i=1}^s k_i \leq \lceil t \rceil
 \right\}\\
 & \cup \left\{ \frac{kt}{\lceil t \rceil}\in \R:\; k\in \N,\; 2\leq k \leq \lceil t \rceil-1\right\}.
  \end{split}
 \end{equation}
\end{itemize}
\end{lem}
\begin{proof}
Since $t>1$, by \eqref{e-4} and using the super-additivity of $x\mapsto x^t$ on $\R_+$, we have
$$
X(u)^t=\left(\sum_{j=0}^{m-1} V_{u_1,j} X_1(u,j)\right)^t\geq \sum_{j=0}^{m-1} V_{u_1,j}^t X_1(u,j)^t.
$$
Taking  expectations on both sides, we obtain  (i).

To see (ii), by \eqref{e-4} and using the sub-additivity of $x\mapsto x^{t/ \lceil t \rceil}$ on $\R_+$, we have
\begin{equation*}
\begin{split}
X(u)^t&=\left(\left(\sum_{j=0}^{m-1} V_{u_1,j} X_1(u,j)\right)^{t/ \lceil t \rceil}\right)^{ \lceil t \rceil}\\
& \leq \left(\sum_{j=0}^{m-1} V_{u_1,j}^{t/\lceil t \rceil} X_1(u,j)^{t/\lceil t \rceil}\right)^{ \lceil t \rceil}\\
&=\sum_{k_0+\ldots+k_{m-1}=\lceil t \rceil}  \frac{\lceil t \rceil !}{k_0!\cdots k_{m-1}!} \prod_{j=0}^{m-1}  (V_{u_1, j} X_1(u, j))^{k_jt/\lceil t \rceil}.
\end{split}
\end{equation*}
Taking  expectations on both sides, we have
\begin{equation*}
\begin{split}
\E(X(u)^t)\leq \sum_{k_0+\ldots+k_{m-1}=\lceil t \rceil}  \frac{\lceil t \rceil !}{k_0!\cdots k_{m-1}!}  \E\left( \prod_{j=0}^{m-1} V_{u_1, j}^{k_jt/\lceil t \rceil}\right) \prod_{s=0}^{m-1} \E(X(\sigma u)^{k_st/\lceil t \rceil}),
\end{split}
\end{equation*}
from which \eqref{e-5} follows, thanks to the fact that $\E(X(\sigma u)^p)\leq 1$ for $p\in [0,1]$; the involved constant $C$ can be taken as
$ m^{q} \sup_{1\leq q'\leq q}  \E(\sum_{j}V_{u_1,j}^{q'})$.   Here we use the fact that
\[
\begin{split}
\E\left( \prod_{j=0}^{m-1} V_{u_1, j}^{k_jt/\lceil t \rceil}\right)&\leq
\prod_{j=0}^{m-1}  (\E(V_{u_1, j}^t))^{k_j/\lceil t \rceil}\\
&\leq \prod_{j=0}^{m-1}  \left(\E\left(\sum_{s=0}^{m-1}V_{u_1, s}^t\right)\right)^{k_j/\lceil t \rceil}\\
&=\E\left(\sum_{s=0}^{m-1}V_{u_1, s}^t\right)\leq\sup_{1\leq q'\leq q} \E\left(\sum_{j}V_{u_1, j}^{q'}\right),
\end{split}
\]
where the first `$\leq$'  comes from the  H\"{o}lder inequality.
\end{proof}

Next we would like to establish an analogue of Lemma \ref{lem-2.1} for $\prod_{j=1}^k \E(X(u)^{t_j}$, where $t_1,\ldots, t_k\in (1,q]$ with $t_1+\ldots+t_k\leq q$.  First we introduce some notation.
For $(x_1,\ldots, x_n)\in \R^n_\leq $ and $(y_1,\ldots, y_m)\in \R^m_\leq $, let $(z_1,\ldots, z_{n+m})\in \R^{n+m}_\leq $ be the vector re-ordered from the numbers $x_1,\ldots, x_n, y_1,\ldots, y_m$; and write
$$
(x_1,\ldots, x_n)\oplus (y_1,\ldots, y_m):=(z_1,\ldots, z_{n+m}).
$$
    Clearly, the operation  `$\oplus$' is commutative.   By convention, we write $(x_1,\ldots, x_n)\oplus \emptyset=(x_1,\ldots, x_n)$, where $\emptyset$ denotes the empty set.

For $t_1,\ldots, t_k\in (1,q]$ with $t_1\leq \cdots \leq t_k$ and  $t_1+\ldots+t_k\leq q$, we  write
\begin{equation}
\label{e-8}
\I_{t_1,\ldots, t_k}=\{w_1\oplus \cdots\oplus w_k:\; w_i\in \I_{t_i}\cup \{t_i\}\cup\{\emptyset\}\}\backslash \{(t_1,\ldots, t_k)\},
\end{equation}
where $\I_t$ is defined as in \eqref{e-6}. The following simple property comes from the definition of $\I_{(\cdot)}$:

\begin{lem}
\label{lem-n}
Assume that  $\I_{t_1,\ldots, t_k}\neq \emptyset$. Then for any $(q_1,\ldots, q_\ell)\in  \I_{t_1,\ldots, t_k}$, we have $q_1+\cdots+q_\ell\leq t_1+\cdots+ t_k$. Moreover, we have either $\ell\geq k+1$ or
$q_1+\ldots+q_\ell\leq t_1+\ldots+ t_k-1/2$.
\end{lem}

 \begin{proof}
 For any vector $w\in \R^m$, let $\|w\|$ denote the sum of the absolute values of its components.  Clearly by \eqref{e-6},  for any $t>1$ and $w\in \I_t$, we have
 $\|w\|\leq t$.  Fix $(q_1,\ldots, q_\ell)\in \I_{t_1,\ldots, t_k}$. Then there exist $w_i\in \I_{t_i}\cup \{t_i\}\cup\{\emptyset\}$ ($i=1,\ldots, k$) such that $(q_1,\ldots, q_\ell)=w_1\oplus \cdots\oplus w_k$. Therefore $q_1+\cdots+q_\ell=\|w_1\|+\cdots+ \| w_k\|\leq t_1+\cdots+t_k$.  If $w_i=\emptyset$ for some $i$, then $q_1+\cdots+q_\ell\leq (t_1+\cdots t_k)-t_i<(t_1+\cdots t_k)-1$. If otherwise, we have $w_i\in \I_{t_i}\cup \{t_i\}$ for all $1\leq i\leq k$, and $w_j\in \I_{t_j}$ for  at least one $j$; in such case, either $\|w_j\|\leq t_j-\frac{t_j}{\lfloor t_j\rfloor}\leq t_j-\frac{1}{2}$ or the dimension of $w_j$ is $\geq 2$, hence we have either $q_1+\ldots+q_\ell\leq t_1+\cdots+t_k-1/2$ or $\ell\geq k+1$.
 \end{proof}
 As a direct application of Lemma \ref{lem-2.1}, we have

\begin{lem}
\label{lem-2.2}Let $t_1,\ldots, t_k\in (1,q]$ so that $t_1\leq \cdots \leq t_k$ and  $t_1+\ldots+t_k\leq q$.
Let  $u\in \Sigma_*$.  
Then
\begin{itemize}
\item[(i)] $\prod_{j=1}^k \E(X(u)^{t_j})\geq m^{-\sum_{i=1}^k T_{u_1}(t_i)} \prod_{j=1}^k \E(X(\sigma u)^{t_j})$.
\item [(ii)] There exists a positive constant $C'$ (depending on $q$) such that
\begin{equation}
\label{e-5}
\begin{split}
\prod_{j=1}^k \E(X(u)^{t_j})\leq & m^{-\sum_{i=1}^k T_{u_1}(t_i)} \prod_{j=1}^k \E(X(\sigma u)^{t_j})\\
& \mbox{} + C'+C'\sum_{(q_1,\ldots, q_\ell)\in {\mathcal I}_{t_1,\ldots, t_k}} \prod_{j=1}^\ell
\E\left(X(\sigma u)^{q_j}\right).
\end{split}
\end{equation}
\end{itemize}
\end{lem}

\begin{pro}\label{mom+estimate} Let $q>1$ such that $T(q)>0$.  Let $\eta$ be the Bernoulli product measure on $\Sigma$ generated by  a probability vector $(p_0',\ldots, p_{m-1}')$.  Set $A:=\max\{1, \sum_{i=0}^{m-1} p_i' m^{-T_i(q)}\}$. Then the following statements hold:\begin{itemize}
\item[(i)] There exists  a polynomial  $f_q$ depending on $W$ and $q$ such that
\begin{equation}
\label{e-p1}
 A^n \leq  \sum_{u\in\Sigma_n}\eta([u])\mathbb{E}(X(u)^q)\leq  f_q(n)  A^n,\quad \forall n\in \N.
\end{equation}
Moreover, if $q\in (1,2]$ and $\sum_{i=0}^{m-1} p_i' m^{-T_i(q)}<1$, then  the polynomial $f_q$  can be replaced by a positive constant.
\item[(ii)]
More generally,  for any $t_1,\ldots, t_k\in (1, q]$ with $t_1\leq \cdots \leq t_k$ and $t_1+\ldots+t_k\leq q$, there exists a polynomial $f_{t_1,\ldots, t_k}$ such that
\begin{equation}
\label{e-p1'}
 1 \leq \sum_{u\in\Sigma_n}\eta([u])\prod_{j=1}^k \mathbb{E} (X(u)^{t_j}) \leq  f_{t_1,\ldots, t_k}(n)  A^n,\quad \forall n\in \N.
\end{equation}
\end{itemize}
  \end{pro}
 \begin{proof}
    Since $q>1$,  we have $\E(X(u)^q)\geq \E(X(u))^q=1$ for each $u\in \Sigma_*$, and thus
  \begin{equation}
  \label{e-10}
  \sum_{u\in\Sigma_n}\eta([u]) \E (X(u)^q)\geq 1.
  \end{equation}
  Similarly we have
  \begin{equation}
  \label{e-10'}
  \sum_{u\in\Sigma_n}\eta([u]) \prod_{j=1}^k\E (X(u)^{t_j})\geq 1.
  \end{equation}

 On the other hand,  by Lemma \ref{lem-2.1}(i),  we have
 \begin{equation}\label{e-11}
 \begin{split} \sum_{u\in\Sigma_n}\eta([u]) \E(X(u)^q)& \geq \left(\sum_{i=0}^{m-1} p_i' m^{-T_i(q)}\right)
   \sum_{u\in\Sigma_{n-1}}\eta([u])\E(X(u)^q)\\
 &\geq \left(\sum_{i=0}^{m-1} p_i' m^{-T_i(q)}\right)^n \E(Y^q)\geq \left(\sum_{i=0}^{m-1} p_i' m^{-T_i(q)}\right)^n.
 \end{split}
 \end{equation}
 Combining \eqref{e-11} with \eqref{e-10}, we have
 \begin{equation}
 \label{e-12}
  \sum_{u\in\Sigma_n}\eta([u]) \E( X(u)^q)\geq A^n.
 \end{equation}
This completes the proof of the first inequality in \eqref{e-p1}.

 To show the second inequality in \eqref{e-p1}, let  $t_1,\ldots, t_k\in (1,q]$ with $t_1\leq \cdots \leq t_k$ and  $t_1+\ldots+t_k\leq q$. By Lemma \ref{lem-1.2},
 \begin{equation}
 \sum_{j=0}^{m-1} p_j' m^{-\sum_{i=1}^k T_j(t_i)}\leq A.
 \end{equation}
 This together with Lemma \ref{lem-2.2}(ii) yields
 \begin{equation}
\label{e-14}
\begin{split}
\sum_{u\in \Sigma_n} \eta([u])\prod_{j=1}^k \E(X(u)^{t_j})\leq & A \sum_{u\in \Sigma_{n-1}} \eta([u]) \prod_{j=1}^k \E(X(u)^{t_j})+C'\\
& \mbox{} +C'\sum_{(q_1,\ldots, q_\ell)\in {\mathcal I}_{t_1,\ldots, t_k}} \sum_{u\in \Sigma_{n-1}}\eta([u]) \prod_{j=1}^\ell
\E\left(X( u)^{q_j}\right),
\end{split}
\end{equation}
Write $S_n(t_1,\ldots, t_k):=\sum_{u\in \Sigma_n} \eta([u])\prod_{j=1}^k \E(X(u)^{t_j})$. Then \eqref{e-14} can be re-written as
\begin{equation}
\label{e-14'}
\begin{split}
S_n(t_1,\ldots, t_k)\leq & A S_{n-1}(t_1,\ldots, t_k)+C'+C'\sum_{(q_1,\ldots, q_\ell)\in {\mathcal I}_{t_1,\ldots, t_k}} S_{n-1} (q_1,\ldots, q_\ell)
\end{split}
\end{equation}
for $n\in \N$.

 We claim that there exists an increasing  polynomial function $f_{t_1,\ldots, t_k}$  such that
\begin{equation}
\label{e-13}
S_n(t_1,\ldots, t_k) \leq f_{t_1,\ldots, t_k}(n) A^n,\qquad \forall n\in \N.
\end{equation}

Clearly the claim is true in the case when ${\mathcal I}_{t_1,\ldots, t_k}=\emptyset$. Indeed in such case, by \eqref{e-14'},
we have
$$
S_n(t_1,\ldots, t_k)\leq A\; S_{n-1}(t_1,\ldots, t_k)+ C',\quad \forall n\in \N,
$$
and thus
\[
\begin{split}
S_n(t_1,\ldots, t_k)&=A^n \; S_0(t_1,\ldots, t_k)+\sum_{j=1}^n A^{n-j}\left(S_j(t_1,\ldots, t_k)-A\;S_{j-1}(t_1,\ldots, t_k)\right)\\
&\leq A^n \;S_0(t_1,\ldots, t_k)+\sum_{j=1}^n C' A^{n-j}\leq nA^n (C'+S_0(t_1,\ldots, t_k)).
\end{split}
\]

Next we consider the case when ${\mathcal I}_{t_1,\ldots, t_k}\neq\emptyset$. Suppose that for each $(q_1,\ldots, q_\ell)\in {\mathcal I}_{t_1,\ldots, t_k}$, there exists  an increasing  polynomial function $f_{q_1,\ldots, q_\ell}$ such that
\begin{equation*}
S_n(q_1,\ldots, q_\ell) \leq f_{q_1,\ldots, q_\ell}(n) A^n,\qquad \forall n\in \N.
\end{equation*}
Set $g=C'+C'\sum_{(q_1,\ldots, q_\ell)\in {\mathcal I}_{t_1,\ldots, t_k}}f_{q_1,\ldots, q_\ell}$. Then $g$ is an increasing polynomial. By \eqref{e-14'}, we have
$$
S_n(t_1,\ldots, t_k)-A\;S_{n-1}(t_1,\ldots, t_k)\leq g(n-1)A^{n-1},\qquad \forall n\in \N.
$$
Therefore
\[
\begin{split}
S_n(t_1,\ldots, t_k)-A^n \; S_0(t_1,\ldots, t_k)&=\sum_{j=1}^n A^{n-j}( S_j(t_1,\ldots, t_k)-A\; S_{j-1}(t_1,\ldots, t_k))\\
& \leq A^{n-1} \sum_{j=1}^n g(j-1)\leq A^{n-1}n g (n),
\end{split}
\]
Hence $S_n(t_1,\ldots, t_k)$ is bounded by $f_{t_1,\ldots, t_k}(n) A^n$ with $f_{t_1,\ldots, t_k}(x): = xg(x)+S_0(t_1,\ldots, t_k)$.

 According to the arguments in the above two paragraphs, if the claim \eqref{e-13} is false at $T_1:=(t_1,\ldots, t_k)$, then $\I_{T_1}\neq \emptyset$ and moreover there exists $T_2\in \I_{T_1}$ such that \eqref{e-13} is false at $T_2$.  Repeatedly applying the arguments, we see that there exist
 $$T_n\in \I_{T_{n-1}}\neq \emptyset,\; n=2, 3,\ldots $$
  such that  \eqref{e-13} is false at $T_n$. However, by Lemma \ref{lem-n}, the sequence $(\|T_n\|)_{n=1}^\infty$ is non-increasing and is bounded above by $q$; and moreover, there are infinitely many $n$ such that $\|T_n\|\leq  \|T_{n-1}\|-1/2$ (because the dimension of $T_n$ can not keep strictly increasing for $q$ consecutive integers of  $n$), which leads to a contradiction. This proves the claim \eqref{e-13}.

  Applying \eqref{e-13} to the particular case when $k=1$, we have
  $$
  \sum_{u\in \Sigma_n} \eta([u]) \E(X(u)^q)\leq f_q(n) A^n, \qquad \forall n\in \N
  $$
  for some polynomial $f_q$. This, together with  \eqref{e-12}, yields \eqref{e-p1}.
  In the meantime, \eqref{e-p1'} follows from \eqref{e-13} and \eqref{e-10'}.

  In the end, assume that $q\in (1,2]$ and $\sum_{i=0}^{m-1} p_i' m^{-T_i(q)}<1$. By the definition \eqref{e-6}, we have $\I_q=\emptyset$. Hence applying \eqref{e-5} yields
  \begin{equation}\label{e-6''}
\sum_{u\in \Sigma_n} \eta([u]) \E(X(u)^{q})\leq B\sum_{u\in \Sigma_{n-1}} \eta([u])  \E(X(u)^q)+C',
\end{equation}
  with $B:= \sum_{i=0}^{m-1} p_i' m^{-T_i(q)}<1$.
  Iterating \eqref{e-6''} yields that $$\sum_{u\in \Sigma_n} \eta([u]) \E(X(u)^{q})\leq C'(1+B+B^2+\cdots)=\frac{C'}{1-B}.$$
     This finishes the proof of the proposition.
   \end{proof}

   \begin{cor}\label{momestimate} Let $q>1$ such that $T(q)>0$. Then there exists  a polynomial  $f_q$ depending on $W$ and $q$ such that
\begin{equation}
\label{e-pp1}
 m^{-n \min \{\tau_\nu(q), T(q)\}} \leq  \mathbb{E}\Big (\sum_{u\in\Sigma_n}\pi_*\mu([u])^q\Big )\leq  f_q(n)  m^{-n \min \{\tau_\nu(q), T(q)\}}
\end{equation}
for all $n\in \N$. Furthermore, if $q\in (1,2]$ and $\tau_\nu(q)<T(q)$, the polynomial  $f_q$  can be replaced by a positive constant.
 \end{cor}
 \begin{proof}
 Let $\nu_q$ denote the Bernoulli product measure on $\Sigma$ generated by the probability weight $(p_0',\ldots, p_{m-1}')$, where
 $p_i':=p_i^q/\sum_{j=0}^{m-1}p_j^q$.   Then
 \begin{equation}
 \label{e-9'}
 \sum_{u\in\Sigma_n}\pi_*\mu([u])^q=\sum_{u\in\Sigma_n}\nu([u])^qX(u)^q=m^{-n\tau_\nu(q)} \sum_{u\in\Sigma_n}\nu_q([u])X(u)^q.
 \end{equation}
 Set $A=\max\{1, \sum_{j=0}^{m-1} p_j' m^{-T_j(q)}\}$. Then $A=\max\{1, m^{\tau_v(q)-T(q)}\}$, due to the fact that
 $\sum_{j=0}^{m-1} p_j' m^{-T_j(q)}=m^{\tau_\nu(q)} \sum_{j=0}^{m-1} p_j^q m^{-T_j(q)} =m^{\tau_\nu(q)-T(q)}$ (cf. \eqref{psi}).

    By Proposition \ref{mom+estimate},   there is a polynomial function $f_q$ such that
    \begin{equation}
    \label{e-XA}
    A^n\leq  \sum_{u\in\Sigma_n}\nu_q([u])\E(X(u)^q)\leq f_q(n) A^n, \quad \forall n\in \N.
    \end{equation}
  Now   \eqref{e-pp1} follows directly from \eqref{e-9'} and \eqref{e-XA}.
  \end{proof}

\begin{cor}\label{cor6.11}
Let $q>1$ such that $T(q)>0$. Then $\tau_{\pi_*\mu}(q)\geq \min \{\tau_\nu(q), T(q)\}$.
\end{cor}
\begin{proof}
It is a direct consequence of Corollary \ref{momestimate} and Lemma \ref{Borel}.
\end{proof}

 \subsection{Negative moments estimates for $X_n$}\label{nmeXu}

We begin with a known lemma.

\begin{lem}[Lemma 4.4, \cite{Liu99}]
\label{lem-liu99}
 Let $Z$ be a positive random variable. For $0 < a < \infty$, consider
the following statements:
\begin{itemize}
\item[(i)] $\E(Z^{-a}) < \infty$;
\item[(ii)] $\E(e^{-t Z}) = O(t^{-a})\;\; (t \to \infty)$;
\item[(iii)]${\Bbb P}(Z\leq z) = O(z^a)\;  \; (z\to 0)$;
\item[(iv)] $\forall \; b \in (0, a)$, $\E(Z^{-b})<\infty$.
\end{itemize}
Then the following implications hold: $(i)\Longrightarrow (ii)\Longleftrightarrow(iii)\Longrightarrow (iv)$.
\end{lem}

Next we present our assumptions and some direct consequences.

The random variables $X(u)$ and $X_n$ are still defined as in the previous section.

Let $\eta$ be the Bernoulli product measure on $\Sigma$ generated by a probability vector $(p_0',\ldots, p_{m-1}')$. Suppose that both $T(q)>0$ and $\sum_{i=0}^{m-1} p_i' m^{-T_i(q)}<1$ for some $q\in (1,2)$.

The first assumption yields  $0<\E(Y^s)<\infty$ for all $t\in [1,q]$.  By convexity of $s\mapsto \sum_{i=0}^{m-1} p_i' m^{-T_i(s)}$ and the fact that $\sum_{i=0}^{m-1} p_i' m^{-T_i(1)}=1$, we have  $\sum_{i=0}^{m-1} p_i' m^{-T_i(s)}<1$ for all $s\in (1,q]$, hence $\sum_{i=0}^{m-1} p_i' T_i'(1)>0$.  Consequently, \begin{equation}
\label{positive}
\sum_{i=0}^{m-1} p_i' T_i(s)>0
\end{equation}
for any $s\in (1,q]$ close enough to 1.

Fix  a number $s\in (1,q]$ so that  \eqref{positive} holds.
Then, Proposition~\ref{XntoX} applied with $q=s$, $U_i=V_i$ and $Z=Y$ implies that $X_n$ converges $\Bbb P\times \eta$-a.e. to a random variable $\widetilde X_V(\omega,x)$ that we simply denote by $X$.  Also,  since $\sum_{i=0}^{m-1} p_i' m^{-T_i(s)}<1$, $\E(Y^s)<\infty$ and $1<s<2$, Proposition~\ref{mom+estimate}  yields $\sup_{n\ge 1}\E_{\Bbb P\times \eta} \left( X_n^{s}\right)  <\infty$, from which we get
$$
\E_{\Bbb P\times \eta} \left( X^{s}\right)  <\infty.
$$
This fact will be used to evaluate the negative moments of $X$ in Propositions~\ref{pro-2.11} and~\ref{negmom1}. For this purpose,  we extend an idea  used in~\cite{HuLi12} in the  context of branching processes in random environments. Then we proceed to the uniform control of the negative moments of $X_n$ in Proposition~\ref{negmom2}. We notice that these moments cannot be directly derived from those of $X$ because $X_n$ cannot be expressed as the conditional expectation of $X$, due to the presence of the factors $Y(u,v)$ in the definition of $X(u)$.

Define \begin{equation}
\label{delta}
\delta=s-1.
\end{equation}
For $t>0$ and $x\in \Sigma$, define
$$
\phi_x(t)=\E(e^{-X(\cdot,x)t}).
$$

Choose $C>1$ large enough so that
$$
e^{-z}<1-z+\frac{C}{1+\delta}z^{1+\delta}, \qquad \forall z>0.
$$

Then we have
$$
e^{-tX(x)} \leq 1-tX(x)+Ct^{1+\delta} X(x)^{1+\delta}/(1+\delta).
$$
Taking expectation and using the fact that $\E(X(x))=1$ for $x\in \Sigma$,  we have
\begin{equation}
\label{phixt}
\phi_x(t)<1-t+Ct^{1+\delta} \E(X(x)^{1+\delta})/(1+\delta).
\end{equation}

Set $K_x=(C\E(X(x))^{1+\delta})^{1/\delta}$. Then $K_x>1$.
Notice that the polynomial in the right-hand side of \eqref{phixt} takes its minimum
$$
\beta_{x}:=1-\frac{\delta}{(1+\delta)K_x} <1
$$
at $t=\frac{1}{K_x}$.  Hence whenever $t\geq 1/K_x$, we have $\phi_x(t)\leq \phi_x(1/K_x)\leq \beta_x$.
In particular, we have
\begin{equation}
\label{phi1}
\phi_x(t)\leq \beta_x, \quad \forall t\geq 1,  \; x\in \Sigma.
\end{equation}

Notice that $X(x)$ satisfies
$$
X(x)=\sum_{j=0}^{m-1} V_{x_1,j} X(x, j),
$$
where $X(x, j)$, $j=0,\ldots, m-1$,  are independent copies of  $X(\sigma x)$, and  are independent to $V_{x_1, j'}$ ($j'=0,\ldots, m-1$).
Hence
\begin{equation}
\label{e-phi}
\phi_x(t)=\int \prod_{j=0}^{m-1} \phi_{\sigma x} (t V_{x_1, j}) \;d\Bbb P.
\end{equation}

First consider a simple model.
\begin{pro}
\label{pro-2.11}
Assume that $T(q)>0$ and $\sum_{i=0}^{m-1} p_i' m^{-T_i(q)}<1$ for some $q\in (1,2]$. Assume, moreover,  that there exists $c\in (0,1)$ such that ${\Bbb P} (V_{i,j}>c)=1$ for all $0\leq i,j\leq m-1$.
Then
$$
\E_{\Bbb P\times \eta}\left( X(x)^{-b}\right)<\infty
$$
for any $b\in (0, -\delta \log m/\log c)$, where $\delta$ is given as in \eqref{delta}.
\end{pro}

\begin{proof}
Since ${\Bbb P} (V_{i,j}>c)=1$ for all $0\leq i,j\leq m-1$,
by \eqref{e-phi}, we have
$$
\phi_x(t)\leq (\phi_{\sigma x}(tc))^m \leq \ldots\leq (\phi_{\sigma^n} (t c^n))^{m^n}, \quad \forall  n\in \N.
$$
 Therefore by \eqref{phi1},
$$
\phi_x(c^{-n})\leq (\phi_{\sigma^n x}(1))^{m^n}\leq (\beta_{\sigma^n x})^{m^n}.
$$
Hence
$$\int \phi_x(c^{-n}) \; d\eta(x)\leq \int (\beta_{\sigma^n x})^{m^n}\; d\eta(x)=\int (\beta_x)^{m^n}\; d\eta(x).$$

Notice that
\begin{eqnarray*}
\begin{split}
(\beta_x)^{m^n} &=\exp\left(m^n \log \left(1-\frac{\delta}{(1+\delta)K_x}\right)\right)\\
 &\leq \exp (- \gamma  m^n /K_x)
\end{split}
\end{eqnarray*}
for some constant $\gamma>0$.  Take $\epsilon>0$ so that $m-\epsilon>1$.  Then
$$
\int \exp (- \gamma  m^n /K_x) \; d\eta (x) \leq \exp \Big(- \gamma  \Big(\frac{m}{m-\epsilon}\Big)^n \Big) +\eta\{x: \; K_x\geq (m-\epsilon)^{n}\}.
$$
Notice that by Markov's inequality,
\begin{eqnarray*}
\eta(\{x: \; K_x\geq (m-\epsilon)^{n}\})&\leq & (m-\epsilon)^{-\delta n} \int K_x^\delta \; d\eta(x)\\
&=&  C (m-\epsilon)^{-\delta n} \int \E(X(x))^{1+\delta} \; d\eta(x)\\
&\leq & C' (m-\epsilon)^{-\delta n}
\end{eqnarray*}
for some constant $C'>0$. It follows that
\begin{eqnarray*}
\E_{\Bbb P\times \eta}\left( \exp(-c^{-n}X(x))\right)  &= & \int \phi_x(c^{-n}) \; d\eta(x)\\
& \leq & \exp \Big(- \gamma  \Big(\frac{m}{m-\epsilon}\Big)^n \Big) +C' (m-\epsilon)^{-\delta n}\\
&\leq & O(c^{na}),
\end{eqnarray*}
with  $a=-\delta \log (m-\epsilon)/\log c$.  Hence by Lemma \ref{lem-liu99},
$\E_{\Bbb P\times \eta} \left(X(x)^{-b}\right)<\infty$ for any $b\in (0, a)$.
\end{proof}

 Next we consider a more general model.

 \begin{pro}\label{negmom1}
 Assume that $T(q)>0$ and $\sum_{i=0}^{m-1} p_i' m^{-T_i(q)}<1$ for some $q\in(1,2]$. Assume that there exists $c\in (0,1)$ such that the following two conditions are satisfied:
 \begin{itemize}
 \item[(i)] $\Bbb P(\sup_{0\leq j\leq m-1} V_{i,j}>c)=1$,\quad  $\forall i$;
 \item[(ii)] $\E(\#\{j: V_{i,j}>c\})>1,\quad \forall i$.
 \end{itemize}
 Then there exists $a>0$ such that
 $$\E_{\Bbb P\times \eta} \left(X(x)^{-a} \right)<\infty.$$
 \end{pro}
\begin{proof}
For $x\in \Sigma$, set $N_0(\epsilon)=1$ and
$$
N_n(x_{|n}):=\# \{v=v_1\ldots v_n\in \Sigma_n: \; V_{x_j, v_j}(x_{|j-1}, v_{|j-1})\geq c \mbox{ for } 1\leq j\leq n\}
$$
for $n\geq 1$.
Then $(N_n(x_{|n}))_{n\geq 0}$ satisfies the following relation:
$$
N_{n+1}(x_{|n+1})=\sum_{v\in {\mathcal A}(x,n)} \#\{j: \; V_{x_{n+1}, j}(x_{|n}, v)>c\},
$$
with ${\mathcal A}(x,n):=\{v=v_1\ldots v_n\in \Sigma_n: \; V_{x_j, v_j}(x_{|j-1}, v_{|j-1})\geq c \mbox{ for } 1\leq j\leq n\}$.

Since  the random variables $\#\{j: \; V_{x_{n+1}, j}(x_{|n}, v)>c\}$, $v\in  {\mathcal A}(x,n)$, are independent with the same distribution (depending on $x$), and are independent of $N_n(x_{|n})$,    $(N_n(x_{|n}))_{n\geq 0}$ is exactly a branching process in the random environment  $x$ (cf. \cite{HuLi12}) picked according to~$\eta$.
For $0\leq i\leq m-1$ and $0\leq k\leq m$, set
$$
\gamma_i(k)=\Bbb P(\#\{j:\; V_{i, j}>c\}=k).
$$
and
$$
M_i=\sum_{k=0}^m k \gamma_i(k).
$$
The assumptions (i)-(ii) guarantee that $\gamma_i(0)=0$ and $M_i>1$ for all $i$.

For $t\in \R$, set $\Lambda(t)=\log (\sum_{i=0}^{m-1} p_i' M_i^t)$.
By \cite[Corollary 1.2]{HuLi12}, we have
\begin{equation}
\label{e-large}
\lim_{n\to \infty} \frac{1}{n} \log \Bbb P\times \eta  \left(\frac{\log N_n(x_{|n})}{n}\leq z\right)\leq  \inf_{t\in \R} \{\Lambda (t)-tz\}
\end{equation}
 for $z <\Lambda'(0)=\sum_{i=0}^{m-1} p_i' \log M_i$. Pick $0<z_0<\sum_{i=0}^{m-1} p_i' \log M_i$. Since $\Lambda$ is  convex and smooth, we have
$$ T_0:=\inf_{t\in \R} \{\Lambda (t)-tz_0\}<0.$$
Take any $T_1\in (T_0, 0)$. By \eqref{e-large}, there exists $L>0$ such that
\begin{equation}
\label{e-Peta}
\Bbb P\times \eta  \left( N_n(x_{|n})\leq e^{nz_0}\right)\leq e^{nT_1}
\end{equation}
for any $n\geq L$.

Similar to \eqref{e-phi}, we have for any $n\in \N$ and $t>0$,
\begin{equation}
\label{e-phi'}
\phi_x(t)=\int \prod_{v=v_1\ldots v_n\in \Sigma_n}  \phi_{\sigma^n x} (t Q_V(x_{|n}, v)) \;d\Bbb P,
\end{equation}
with $Q_V(x_{|n}, v)=\prod_{j=1}^n V_{x_j, v_j}(x_{|j-1}, v_{|j-1}).$  Since
$$
\#\{v\in \Sigma_n:\; Q_V(x_{|n}, v)\geq c^n\}\geq N_n(x_{|n}),
$$
by \eqref{phi1} we have
\begin{equation}
\label{e-phi''}
\phi_x(c^{-n})\leq \int   (\phi_{\sigma^n x}(1))^{N_n(x_{|n})}\;d\Bbb P\leq \int   (\beta_{\sigma^n x})^{N_n(x_{|n})}\;d\Bbb P.
\end{equation}

Hence for $n\geq L$,  we have
\begin{eqnarray*}
\E_{\Bbb P \times \eta} (\exp(-c^{-n}X(x)))&\leq & \iint (\beta_{\sigma^nx})^{N_n(x_{|n})} d\Bbb Pd\eta \\
&\leq & \iint {\bf 1}_{\{N_n(x_{|n})\geq e^{nz_0} \}} (\beta_{\sigma_n x})^{e^{nz_0} }d\Bbb Pd\eta+\Bbb P\times \eta\left(N_n(x_{|n})\leq e^{nz_0}\right) \\
&\leq & \int (\beta_x)^{e^{nz_0}} \; d\eta(x)+ e^{nT_1}.
\end{eqnarray*}

Repeating the argument in the proof of Proposition \ref{pro-2.11}, we can show that there exists $a'>0$ such that
$$\int (\beta_x)^{e^{nz_0}}\; d\eta(x) =O(c^{na'}).$$
Hence $\E_{\Bbb P \times \eta} (e^{-c^{-n}X(x)})=O(c^{na'}+e^{nT_1})$. Since $T_1<0$, by Lemma \ref{lem-liu99} this is enough to conclude the desired result.
\end{proof}

Now we estimate moments of negative orders of $X_n=X(x_{|n})$.

\begin{pro} \label{negmom2} Assume the hypotheses of Proposition~\ref{negmom1}. Then there exists $b>0$ such that $\sup_{n\ge 1}\E_{\mathbb P\otimes \eta} (X_n^{-b})<\infty$.
\end{pro}

\begin{proof} Fix $q\in (1,2]$ as in the beginning of this section, i.e. $T(q)>0$ and  $ \sum_{i=1}^{m-1}p'_im^{-T_i(q)}<1$, and  notice that for $n\ge 1$, we have $\mathbb{E}_{\mathbb P\otimes\eta}(\sum_{v\in\Sigma_n}Q_V(x_{|n}, v)^q)=\Big (\sum_{i=1}^{m-1}p'_im^{-T_i(q)}\Big )^n$. Consequently,  there exist two positive numbers  $\alpha$ and $\gamma$ such that
\begin{equation}
\label{e-PE}
\mathbb P\otimes\eta(\max_{v\in \Sigma_n}Q_V(x_{|n}, v)>m^{-\alpha n})\le  m^{-\gamma n}.
\end{equation}
 Now, we estimate  $\mathbb{E}_{\mathbb P\otimes\eta}(e^{-tX_n})$ in two steps.

At first, we write
$$
X_n=\sum_{v\in \Sigma_n}  Q_V(x_{|n}, v) Y(x_{|n},v)
$$
and use the independence and equidistribution properties of the random variables $Y(x_{|n},v)$ with respect to $\mathbb{E}_{\mathbb P\otimes\eta}$  to get that, for any $t\ge 0$, $$\mathbb{E}_{\mathbb P\otimes\eta}(e^{-tX_n})=\mathbb{E}_{\mathbb P\otimes\eta}\Big (\prod_{v\in \Sigma_n}\phi_Y(tQ_V(x_{|n}, v))\Big ).
$$
We know that $\phi_Y(u)\le e^{-u/2} $ for $u$ small enough since $\E(Y)=1$. Consequently, there exists $\epsilon_0>0$ such that for $n$ large enough, if $0<t\le \epsilon_0 m^{n\alpha}$, on $\{\max_{v\in \Sigma_n}Q_V(x_{|n}, v)\le m^{-n\alpha})\}$ we have
$$
\prod_{v\in \Sigma_n}\phi_Y(tQ_V(x_{|n}, v))\le e^{-\frac{t}{2}\sum_{v\in\Sigma_n}Q_V(x_{|n}, v)}=e^{-\frac{t}{2}\widetilde X_n}.
$$
Moreover, the Mandelbrot martingale in random environment $(\widetilde X_n)$ being uniformly integrable and $u\mapsto e^{-u/2}$ being convex,  for all $n\ge 1$ we have
$$
\mathbb{E}_{\mathbb P\otimes\eta}(e^{-\frac{t}{2}\widetilde X_n})\le \mathbb{E}_{\mathbb P\otimes\eta}(e^{-\frac{t}{2}\widetilde X}).
$$
Thus,  for $n$ large enough, for $0< t\le \epsilon_0 m^{n\alpha}$, we have
\begin{eqnarray*}
\mathbb{E}_{\mathbb P\otimes\eta}(e^{-tX_n})
&=&\mathbb{E}_{\mathbb P\otimes\eta}\Big (\prod_{v\in \Sigma_n}\phi_Y(tQ_V(x_{|n}, v))\Big )\\
&\le & \mathbb{E}_{\mathbb P\otimes\eta}(e^{-\frac{t}{2}\widetilde X})+ \mathbb P\otimes\eta(\max_{v\in \Sigma_n}Q_V(x_{|n}, v)>m^{-\alpha n}).
\end{eqnarray*}
Due to Proposition \ref{negmom1}, Lemma \ref{lem-liu99} and the estimate \eqref{e-PE}, we get  constants $C_0>0$ and $b_0>0$ such that for $0< t\le \epsilon_0 m^{n\alpha}$ we have
$$
\mathbb{E}_{\mathbb P\otimes\eta}(e^{-tX_n})\le C_0 t^{-b_0}.
$$
Moreover, taking $\alpha$ slightly smaller we can assume without loss of generality that $\epsilon_0=1$. %

The second estimate of  $\mathbb{E}_{\mathbb P\otimes\eta}(e^{-tX_n})$ is as follows. Since, with probability 1, we have $\#\{(i,j)\in\{0,\ldots,m-1\}^2: W_{i,j}\ge c p_i\}\ge 1$ almost surely, we know from \cite[Theorem 4.1]{liu01} that for some $\beta_0>0$ we have $\E(Y^{-\beta})<\infty$ for all $\beta\in(0,\beta_0]$, hence $\mathbb{E}_{\mathbb P\otimes\eta}(Y^{-\beta}(x_{|n},v))<\infty$ for all $n\ge 1$, $x\in\Sigma$ and $v\in\Sigma_n$ and  $\beta\in(0,\beta_0]$. Fix $\beta\in (0,\beta_0]$ and write
$$
X_n(x)=\sum_{j=0}^{m-1,*}  V_{x_1,j} X_{n-1}(\sigma x,j),
$$
where $*$ means that  we sum over those $j$ such that $V_{x_1,j}>0$. Setting $N_{x_1}=\#\{j:V_{x_1,j}>0\}$  and using the convexity of $t\mapsto t^{-N_{x_1}}$ in  a similar way as in \cite{Mol,B1999} to study the moments of negative orders of $Y$,  we get
$$
X_n(x)^{-\beta}\le N_{x_1}^{-\beta} \Big (\prod_{j=0}^{m-1,*} V_{x_1,j}^{-\beta/N_{x_1}}\Big )\prod_{j=0}^{m-1,*}  X_{n-1}(\sigma x,j)^{-\beta/N_{x_1}},
$$
which yields
\begin{eqnarray*}
\nonumber\mathbb{E}_{\mathbb P\otimes\eta}(X_n(x)^{-\beta})&\le& \mathbb{E}_{\mathbb P\otimes\eta}\left (N_{x_1}^{-\beta} \Big (\prod_{j=0}^{m-1,*}V_{x_1,j}^{-\beta/N_{x_1}}\Big ) \mathbb{E}_{\mathbb P\otimes\eta}\Big (\prod_{j=0}^{m-1,*}  X_{n-1}(\sigma x,j)^{-\beta/N_{x_1}}|(x_1,V_{x_1})\Big )\right )\\
&\le & \mathbb{E}_{\mathbb P\otimes\eta}\left (N_{x_1}^{-\beta} \Big (\prod_{j=0}^{m-1,*}V_{x_1,j}^{-\beta/N_{x_1}}\Big ) \prod_{j=0}^{m-1,*} \E_{\mathbb P\otimes\eta}(X_{n-1}(\sigma x,j)^{-\beta}|(x_1,V_{x_1}))^{1/N_{x_1}}\right )\\
&= &\mathbb{E}_{\mathbb P\otimes\eta}\Big (N_{x_1}^{-\beta} \Big (\prod_{j=0}^{m-1,*}V_{x_1,j}^{-\beta/N_{x_1}}\Big )\Big ) \mathbb{E}_{\mathbb P\otimes\eta} (X_{n-1}(x)^{-\beta}),
\end{eqnarray*}
where we have successively used generalized H\"older inequality  to bound from above  the term $\mathbb{E}_{\mathbb P\otimes\eta}\Big (\prod_{j=0}^{m-1,*}  X_{n-1}(\sigma x,j)^{-\beta/N_{x_1}}|(x_1,V_{x_1})\Big )$,  and the equality in distribution of the random variables $X_{n-1}(\sigma x,j)$ with $X_{n-1}(x)$ as well as their independence with respect to $(x_1,V_{x_1})$ under $\mathbb P\otimes\eta$.  Consequently, using the previous inequality recursively we obtain
$$
\mathbb{E}_{\mathbb P\otimes\eta}(X_n^{-\beta})\le g(\beta)^n  \mathbb E(Y^{-\beta}),
 $$
where
$$
g(\beta)=\mathbb{E}_{\mathbb P\otimes\eta}\Big (N_{x_1}^{-\beta} \Big (\prod_{j=0}^{m-1,*}V_{x_1,j}^{-\beta/N}\Big )\Big ).
$$
Now, for any $t>1$, setting $u=\log(t)^2/t$ we have
\begin{eqnarray*}
\mathbb{E}_{\mathbb P\otimes\eta}(e^{-tX_n})&\le&\mathbb P\otimes\eta(X_n\le u)+e^{-tu}\\
&\le & u^\beta\mathbb{E}_{\mathbb P\otimes\eta}(X_n^{-\beta})+e^{-\log(t)^2}\\
&\le& \mathbb E(Y^{-\beta})  \log(t)^{2\beta} t^{-\beta} g(\beta)^n +e^{-\log(t)^2}.
\end{eqnarray*}
It is not hard to check that $g(\beta)$ is differentiable at $0$. Hence we can fix  $\theta>0$ such that for $\beta$ small enough we have $g(\beta)^n\le m^{\beta\theta n}$. Fix such a $\beta\in (0,\beta_0]$. For $t> m^{3\theta n}$, the previous inequalities yield
$$
\mathbb{E}_{\mathbb P\otimes\eta}(e^{-tX_n})=O( \log(t)^{2\beta} t^{-2\beta /3})=O(t^{-\beta/3}),
$$
where $O$ is uniform with respect to $n$. Consequently, if $3\theta< \alpha$, so that $m^{3\theta n}<  m^{\alpha n}$, our two estimates for $\mathbb{E}_{\mathbb P\otimes\eta}(e^{-tX_n})$ yield  $C>0$ such that for all $n\ge 1$ and $t>1$ we have   $\mathbb{E}_{\mathbb P\otimes\eta}(e^{-tX_n})\le C t^{-\min (b_0,\beta/3)}$.  Otherwise,  for  $ m^{\alpha n}< t <m^{3\theta n}$, using the log convexity of $u\mapsto \mathbb{E}_{\mathbb P\otimes\eta}(e^{-uX_n})$ and our estimates of this map at $m^{\alpha n}$ and $m^{3\theta n}$, we get  $\mathbb{E}_{\mathbb P\otimes\eta}(e^{-tX_n})\le C t^{-\min (b_0\alpha/(3\theta),\beta/3)}$ for all $t>1$ and independently of~$n$.

Finally, since for $h>0$ we have  $\mathbb{E}_{\mathbb P\otimes\eta}(X_n^{-h})=\Gamma(h)^{-1}\int_0^\infty \mathbb{E}_{\mathbb P\otimes\eta}(e^{-tX_n}) t^{h-1}\mathrm{d}t$, we conclude that $\sup_{n\ge 1}\mathbb{E}_{\mathbb P\otimes\eta}(X_n^{-b})<\infty$ for all $b\in (0,\min (b_0,b_0\alpha/(3\theta),\beta/3))$.
\end{proof}

\section{Final remarks}

As a consequence of our study of the multifractal formalism, we can achieve a part of the multifractal analysis of the number  $N_n(x)$  of cylinders of generation $n$ of the form $[x_{|n},v]$, $v\in\Sigma_n$,  which intersect the support $K_n$ of $\mu_n$.  Specifically, if $n\ge 1$ and $u\in\Sigma_n$ we set
$$
N(u)=\#\{v\in\Sigma_n: Q(u,v)>0\}.
$$
Then $N_n(x)=N(x_{|n})$. This number measures  the overlapping amount over $[x_{|n}]$ when one projects $K_n$ onto $\pi(K)$.

\begin{cor}\label{covnumb}
\begin{enumerate}
\item  Suppose that $\E(N_i)\le 1$ for all $0\le i\le m-1$ such that $\E(N_i)>0$. With probability one, conditionally on $K\neq\emptyset$, for all $x\in \pi(K)$ one has $\lim_{n\to\infty}\frac{\log N_n(x)}{n}=0$.

\item Suppose that $\E(N_i)>1$ for at least one $0\le i\le m-1$.  Let $\varphi$ be defined as in  \eqref{varphi}. Let $q_0$ be the unique point at which $\varphi$ attains its minimum over $[0,1]$. Define
\begin{equation}\label{pressure}
P: q\mapsto \begin{cases}
\log(m)\cdot \varphi(q_0) &\text{ if }0\le q\le q_0\\
\log(m)\cdot \inf\{\varphi(q/s): q\le s\le 1\} &\text{ if }q_0<q\le 1\\
\max (\log\E(N), \log(m)\cdot \varphi(q))&\text{ if }q> 1
\end{cases}.
\end{equation}
If $q_0<1$ or $q_0=1$ and  $\varphi'(1)=0$, then $P$ is differentiable over $\R_+$, analytic over $[0,q_0)\cup(q_0,\infty)$ and it has a second order phase transition at $q_0$. Specifically, $P\equiv \log(m)\cdot \varphi(q_0)$ over $ [0,q_0)$ and $P\equiv \log(m)\cdot \varphi $ over $(q_0,\infty)$.

If $q_0=1$ and  $\varphi'(1)<0$, then  there exists a unique $q'_0>1$ such that $P(q'_0)= \log \E(N)$, and  $P$ is analytic over $[0,q'_0)\cup(q'_0,\infty)$, with  $P\equiv  \log \E(N)$ over $ [0,q'_0)$ and $P\equiv \log(m)\cdot \varphi $  over $(q'_0,\infty)$. Moreover, $P$ has a first order phase transition at $q'_0$.

With probability 1, conditionally on $\pi(K)\neq\emptyset$, for all $q\ge 0$ we have
\begin{equation}\label{PntoP}
\lim_{n\to\infty} \frac{1}{n}\log \sum_{|u|=n}\mathbf{1}_{\{N(u)\ge 1\}}N(u)^q=P(q).
\end{equation}
\item If $\alpha\in \{P'(q^-),P'(q^+)\} $ for some $q>0$ or $\alpha=P'(0+)$, then, with probability 1, conditionally on $K\neq\emptyset$,
$$
\dim_H\left\{x\in\pi(K):\frac{\log N_n(x)}{n}=\alpha\right\}=\frac{1}{\log(m)}\inf\{P(q)-\alpha q: q\ge 0\}.
$$

\end{enumerate}

\end{cor}
 Parts (2)-(3) of this corollary follow from the application of Theorem~\ref{MA} to the branching measure, i.e. the Mandelbrot measure $\mu'$ associated with $$W'=(\E(N)^{-1}\mathbf{1}_{\{W_{i,j}>0\}})_{0\le i,j\le m-1}.$$
 More precisely, one writes that $N_n(x)=\E(N)^n \mu'_n([x_n])$ and use Remark \ref{rem-2.9}.

 For Part (1),  under our assumptions property \eqref{PntoP} still holds, with $P$ given by~\eqref{pressure}, for the same reason as in item (2). It is then direct to check that $P(q)=\log \E(N)$ for all $q\ge 0$. Consequently, conditionally on $\pi(K)\neq\emptyset$, for any $\epsilon>0$, for any $q>0$, if $n$ is large enough, we have
 \begin{eqnarray*}
 \#\{u\in\Sigma_n: N(u)\ge m^{n\epsilon}\}&\le& m^{-nq\epsilon}  \sum_{|u|=n}\mathbf{1}_{\{N(u)>0\}}N(u)^q\\
 &\le& m^{-nq\epsilon}m^{n(\log(\E(N))+\epsilon)}\\
 &=& m^{-n((q-1)\epsilon-\log\E(N))}.
 \end{eqnarray*}
 Choosing $q>1$ such that $(q-1)\epsilon-\log\E(N)>0$ yields that for $n$ large enough, $\#\{u\in\Sigma_n: N(u)\ge m^{n\epsilon}\}<1$  so $\{u\in\Sigma_n: N(u)\ge m^{n\epsilon}\}$ is empty. Thus,  for all $x\in\pi(K)$, we have $\limsup_{n\to\infty}\frac{\log N_n(x)}{n}\le \epsilon$. Since $\epsilon$ is arbitrary and $N_n(x)\ge 1$, this yields $\lim_{n\to\infty}\frac{\log N_n(x)}{n}=0$ for all $x\in\pi(K)$.

\appendix \label{A}
\section{\ }
\begin{pro}\label{Kmu}The events $\{\mu\neq 0\}$ and $\{K:=\bigcap_{n\ge 1}\mathrm{supp}(\mu_n)\neq\emptyset\}$ coincide up to a set of probability 0, over which $K=\mathrm{supp}(\mu)$.
\end{pro}
\begin{proof}
Recall that we defined $N=\sum_{1\le i,j\le m-1} \mathbf{1}_{\{W_{i,j}>0}$ and that our assumptions on~$W$ imply that $\E(N)>1$. Consequently, the generating function of $N$, i.e. $f(x)=\sum_{n\ge 0} \mathbb{P}(N=n)x^n$, has a unique fixed point smaller than 1, which equals the probability of extinction of the associated Galton-Watson process generated by $N$, i.e. the probability of the event  $\{K=\emptyset\}=\bigcup_{n\ge 1} \{K_n=\emptyset\}$. Also, since
$$
Y=\|\mu\|=\sum_{1\le i,j\le m-1} W_{i,j} Y(i,j),
$$
where the $Y(i,j)$ are independent copies of $Y$, and are also independent of $W$, the probability of $\{\mu=0\}$ is also a fixed point of $f$. Moreover, by construction, $\{K=\emptyset\}\subset \{\mu=0\}$. Since $\mathbb P(\mu\neq 0)<1$, $\{K=\emptyset\}$ and $\{\mu=0\}=\{Y=0\}$ must be equal up to a set of probability 0.

Also, we have $\mathrm{supp}(\mu)\subset K$ almost surely. Moreover,  by the previous paragraph and statistical self-similarity, for each $n\ge 1$ and each cylinder $[u,v]$ of the $n$-th generation,  $\{K\cap [u,v]\neq\emptyset\}$ coincides with the event $\{Q(u,v)>0\}\cap\{Y(u,v)>0\}$ up to a set of probability 0. Consequently, with probability 1, for all the cylinder $[u,v]$, $K\cap [u,v]\neq\emptyset$ implies $\mu([u,v])>0$, that is $K\subset \mathrm{supp}(\mu)$.
\end{proof}

\bigskip

For $i=0,1,\ldots, m-1$, define polynomial functions $f_i$ by
$$
f_i(x)=\sum_{\ell=0}^m {\Bbb P}(N_i=\ell)\;x^\ell,
$$
where $N_i=\#\{0\leq j\leq m:  \; V_{i, j}\neq 0\}$.

A Borel measurable function $p:\Sigma\to [0,1]$ is called {\it $\{f_i\}_{i=0}^{m-1}$-stationary}, if $p$ satisfies the following condition:\;
$$p({\bf i})=f_{i_1}(p(\sigma {\bf i})), \quad \forall\; {\bf i}=(i_n)_{n=1}^\infty\in \Sigma.$$   Let $\nu'$ be a Bernoulli product measure on $\Sigma$. Two functions $p$ and $p'$ on $\Sigma$ are called {\it equivalent} if $p({\bf i})=p'({\bf i})$ for $\nu'$-a.e.~${\bf i}$; for brevity we write $p=p'$ a.e.~ if they are equivalent.

Notice that the constant function
$1$ on $\Sigma$ is always $\{f_i\}_{i=0}^{m-1}$-stationary.

\begin{pro}\label{f-stat} Assume that there exists at least one $i$ so that $\mathbb P (N_i=1)<1$; equivalently,  there exists $i$ so that $f_i(x)\not\equiv x$.
Then there exist at most one  $\{f_i\}_{i=0}^{m-1}$-stationary function on $\Sigma$ which is not equivalent to the constant function $1$.
\end{pro}

\begin{proof}
 Let $\mathcal G$ denote the collection of functions $f:\; [0,1]\to [0,1]$ so that $f$ is increasing, continuous, convex and $f(1)=1$. Notice that by convexity, for any $0<a<1$, and $f\in \mathcal G$, we have
$$\sup_{0\leq x,y\leq a}\left|\frac{f(y)-f(x)}{y-x}\right|\leq \frac{ f(1)-f(a)}{1-a}\leq \frac{1}{1-a}.$$
Therefore, $\mathcal G$ is equicontinuous on $[0,a]$ for any $a\in (0,1)$.

Let $\mathcal{I}=\{i:\; f_i(x)\not\equiv x\}$. By our assumption, $\mathcal{I}\neq \emptyset$.  Notice that  by convexity, for each $i\in \mathcal{I}$, either $f_i(x)>x$ for any $x\in [0,1)$,  or $f$ has exactly one attractive fixed point in $[0,1)$.   In the first case, $f_i^n(x)\to 1$ uniformly on $[0, a]$ for each $0<a<1$,  whilst in the second case, $f_i^n(x)$ converges uniformly to the attractive fixed point of $f_i$, on $[0,a]$ for each $0<a<1$.

Now we consider the following two cases separately: (A) $f_i(x)>x$ on $[0,1)$ for each $i\in \mathcal{I}$; (B) there exists at one $i\in \mathcal{I}$  such that $f$ has one fixed point in $[0,1)$.

First suppose that (A) occurs. Then $f_i(x)\geq x$ for any $0\leq i\leq m-1$.
 Pick $i_0\in \mathcal{I}$ and  let $0\leq a<1$. Then  there exists $n\in \N$ such that
$f_{i_0}^n(x)>a$ for any $[0,1]$. Let $p$ be $\{f_i\}_{i=0}^{m-1}$-stationary.
Then for $\nu'$-a.e.  ${\bf i} \in \Sigma$, there exists $k\in \N$ such that $\sigma^k {\bf i}\in [i_0^n]$, and thus
$$
p({\bf i})=f_{i_1}\circ \ldots \circ f_{i_k} \circ f_{i_0}^n (p(\sigma^{k+n}{\bf i}))\geq f_{i_0}^n (p(\sigma^{k+n}{\bf i}))\geq a.
$$
Since $a\in [0,1)$ is arbitrarily, we see that $p({\bf i})=1$ for $\nu'$-a.e. ${\bf i}$.

Next suppose that (B) occurs.
Assume that $p$ and $p'$ are both
$\{f_i\}_{i=0}^{m-1}$-stationary, and not equivalent to the constant function $1$. We show below that $p=p'$ a.e.

First we claim that $p({\bf i})<1$ and $p'({\bf i})<1$ for $\nu'$-a.e. ${\bf i}$. Without loss of generality we only prove the first inequality. Suppose on the contrary that $p=1$ on a Borel set $A\subset \Sigma$ with $\nu'(A)>0$. By the Poincare recurrence theorem,  for $\nu'$-a.e. ${\bf i}$, there exists $k=k({\bf i})\in \N$ such that $\sigma^{k}{\bf i}\in A$; and thus
$$
p({\bf i})=f_{i_1}\circ \ldots \circ f_{i_n}(p(\sigma^n {\bf i}))=f_{i_1}\circ \ldots \circ f_{i_n}(1)=1.
$$
This contradicts the assumption that $p$ is not equivalent to the constant function $1$. Hence we have $p({\bf i})<1$ for $\nu'$-a.e. ${\bf i}$.

By the above claim, we can pick $\delta>0$ such that there exists a Borel set $A=A_\delta\subset \Sigma$ with $\nu'(A)>0$ such that
$$p({\bf i})\leq 1-\delta \mbox{ and } p'({\bf i})\leq 1-\delta,\quad \forall \; {\bf i}\in A.
$$
Pick $j_0\in \mathcal{I}$ so that $f_{j_0}$ has an attracting point in $[0, 1)$, say, $b$. Then we have
$$\lim_{n\to \infty} f_{j_0}^n([0,1-\delta])=b,$$
here and afterwards, $f^n_{j_0}$ denotes the $n$-th iteration of $f_{j_0}$.
Since for each $n\in \N$,
$\nu'([j_0^n]\cap \sigma^{-n}A)>0$, by the Poincar\'e recurrence theorem, for $\nu'$-a.e. ${\bf i}$, there exist $k_1<k_2<\ldots$,  such that
$$
\sigma^{k_n}{\bf i} \in [j_0^n]\cap \sigma^{-n}A.
$$
Clearly $\lim_{n\to \infty} p(\sigma^{k_n}{\bf i})=\lim_{n\to \infty} p'(\sigma^{k_n}{\bf i})=b$.

Notice that the family of functions
$$
\{f_{i_1}\circ \ldots\circ f_{i_{k_n}}\}_{n\in \N}
$$
is a subset of $\mathcal G$. Hence it  is equi-continuous on $[0,a]$ for any $a<1$. By the Arzel\`{a}-Ascoli theorem, there exists a subsequence  $(t_\ell)$ of $(n_k)$ and a continuous function $g$ on $[0,1)$ such that
$$f_{i_1}\circ \ldots\circ f_{i_{t_\ell}} \mbox{ converges to $g$  uniformly } $$
 on any interval $[0, a]$ with $a<1$.   Since $b<1$, we have
 $$
 p({\bf i})=\lim_{\ell\to \infty}f_{i_1}\circ \ldots\circ f_{i_{t_\ell}}(p(\sigma^{t_\ell} {\bf i}))=g(b),
 $$
 and similarly  $p'({\bf i})=g(b)$. Hence $p({\bf i})=p'({\bf i})$. Therefore we have $p=p'$ a.e.
 \end{proof}

\section{Basic properties of Mandelbrot martingales in Bernoulli environment}\label{B1} Let $U=(U_{i,j})_{(i,j)\in\Sigma_1\times\Sigma_1}$ be a non negative random vector such that for each $0\le i\le m-1$ we have $\sum_{j=0}^{m-1} \E(U_{i,j})=1$. Let $(U(u,v))_{(u,v)\in \bigcup_{n\ge 0}\Sigma_n\times\Sigma_n}$ be a sequence of independent copies of $U$.

For each $n\ge 1$ and $(u,v)\in\Sigma_n\times\Sigma_n$ let
$$
Q_U(u,v)=\prod_{k=1}^n U_{u_k,v_k}(u_{k-1},v_{k-1})
$$
and
$$
\widetilde X_U(u):=\sum_{|v|=n}Q_U(u,v).
$$
Now for each fixed $x\in \Sigma$ and $n\ge 1$ let $\widetilde \mu_{U,n}^x$ be the measure on $\Sigma$ whose density with respect to the measure of maximal entropy is given by $m^n Q_U(x_{|n},v)$ over any cylinder $[v]$ of generation $n$. The sequence $(\widetilde \mu_{U,n}^x)_{n\ge 1}$  almost surely converges to an inhomogeneous Mandelbrot measure $\widetilde \mu_{U}^x$.

Let $\eta$ be a Bernoulli product on $\Sigma$ associated with a probability vector $(p'_0,\ldots,p'_{m-1})$. Then,  for $\eta$-almost every $x$ the sequence  $(\widetilde \mu_{U,n}^x)_{n\ge 1}$ converges almost surely weakly to a measure $\nu^x$ as well, and  $(\|\widetilde \mu_{U,n}^x\|=\widetilde X_U(x_{|n})=\widetilde X_{U,n}(x))_{n\ge 1}$ is a Mandelbrot martingale in the random environment given by $\eta$, which almost surely converges to $\|\widetilde \mu_{U}^x\|$, which we denote by $\widetilde X_U(x)$.

By construction, for each $n\ge 0$ and $J\in\Sigma_n$, we have the relation
\begin{equation}\label{muUnx}
\widetilde \mu_{U}^x([J])=\widetilde X^{x_{|n},J}(\sigma^nx)\prod_{k=1}^n U_{x_k,J_k}(x_{|k-1},J_{|k-1}),
\end{equation}
where
\begin{equation}\label{XUnx}
\widetilde X_U^{x_{|n},J}(\sigma^nx)=\lim_{p\to\infty} \sum_{K\in\Sigma_p}\prod_{\ell=1}^p U_{x_{n+\ell},K_\ell}(x_{|n+\ell-1},J(K_{|\ell-1})).
\end{equation}

For $0\le i\le m-1$, let
$$
T_{U_i}(q)=-\log_m\E\sum_{j=0}^{m-1} U_{i,j}^q\quad (q\ge 0).
$$
Suppose that there exists $0\le i\le m-1$ such that $\mathbb{P}(\{U_{i,j}\in\{0,1\}\ \forall\, 0\le j\le m-1\})<1$.

We have the following consequence of the a general result by Biggins and Kyprianou \cite[Theorem 7.1]{BK04}.

\begin{thm}\label{thmBK} Suppose that $\mathbb{P}(\sum_{j=0}^{m-1}\mathbf{1}_{\{U_{i,j}>0\}}=1)<1$ for some $0\le i\le m-1$. The following properties are equivalent:

\begin{itemize}
\item [(i)] $\mathbb P\otimes\eta(\widetilde X_U>0)>0$;
\item [(ii)] $(\widetilde X_n)_{n\ge 1}$ is uniformly integrable with respect to $\mathbb P\otimes\eta$;
\item [(iii)] $\sum_{i=0}^{m-1} p'_i T_{U_i}'(1-)=-\mathbb{E}_{\mathbb{P}\otimes\nu}\left(\sum_{j=0}^{m-1} U_{x_1,j}\log(U_{x_1,j})\right )>0$.
\end{itemize}
\end{thm}

We also have the following useful fact. Let $Z$ be an integrable random variable, and let $(Z(u,v))_{(u,v)\in\Sigma_n,\, n\ge 1}$ be a collection of copies of $Z$ such that for each $n\ge 1$ the random variables $Z(u,v)$, $(u,v)\in\Sigma_n$ are independent, and independent of $\sigma(U(u',v'): |u'|=|v'|\le n-1)$.

Let
$$
X_{U}(x_{|n})=\sum_{|v|=n}  Q_U(x_{|n},v)Z(x_{|n},v).
$$
\begin{pro}\label{XntoX}
Let $q\in (1,2]$. Suppose that $\E(|Z|^q)<\infty$, $T_{U_i}(q)$ is finite for all $0\le i\le m-1$ such that $p'_i>0$, and  $\sum_{i=0}^{m-1} p'_i T_{U_i}(q)>0$. Then, $\mathbb P\otimes\eta(\widetilde X_U>0)>0$, and with probability 1, for $\eta$-almost every $x$, we have $\lim_{n\to\infty} X_{U}(x_{|n})=\E(Z) \widetilde X_U$.
\end{pro}
\begin{proof} At first we notice by concavity of the mappings $T_{U_i}$, the fact that all the functions $T_{U_i}$ vanish at 1 together with the assumption $\sum_{i=0}^{m-1} p'_i T_{U_i}(q)>0$ imply that $\sum_{i=0}^{m-1} p'_i T_{U_i}'(1)>0$. Consequently, due to Theorem~\ref{thmBK} we have $\mathbb{P}\otimes\eta(\widetilde X_U>0)>0$.

Next, recall the following standard lemma.
\begin{lem}\cite{vB-E}\label{martdiff}
 Let $(L_j)_{j\geq 1}$ be a sequence of centered independent  real valued  random variables. For every finite  $I\subset\N_+$ and $q\in (1, 2]$,
$$\E\Big (\Big |\displaystyle\sum_{i\in I} L_{i}\Big|^q\Big ) \leq 2^{q-1}  \sum_{i\in I}  \E( \left|L_{i}\right|^q).$$
\end{lem}
For all $n\ge 1$, we have
$$
X_{U}(x_{|n+1})-\E(Z) \widetilde X_{U}(x_{|n+1})= \sum_{|v|=n} Q_U(x_{|n},v) \widetilde U(x_{|n},v),
$$
where
$$
\widetilde U(x_{|n},v)=\sum_{j=0}^{m-1} U_{x_{|n+1},v_{n+1}}(x_{|n},v) \big(Z(x_{|n+1},vj)-\E(Z)\big).
$$
By construction, conditionally on $x$ (recall that we work under $\mathbb P\otimes\eta$) the random variables $\widetilde U(x_{|n},v)$ are i.i.d and centered, and they are also independent of the $Q_U(x_{|n},v)$ invoked in $X_{U}(x_{|n+1})-\E(Z) \widetilde X_{U}(x_{|n+1})$ (with respect to $\mathbb P$). Consequently, conditioning with respect to the $Q_U(x_{|n},v)$ makes it possible to apply Lemma~\ref{martdiff} to $\{L_v=\widetilde U(x_{|n},v)\}_{v\in\Sigma_n}$ weighted by the constants $Q_U(x_{|n},v)$ and finally to get, for $q\in (1,2]$:
\begin{equation}\label{expconv}
\mathbb{E}(|X_{U}(x_{|n+1})-\E(Z)\widetilde X_{U}(x_{|n+1})|^q)\le 2^{q-1} \sum_{|v|=n} \E(Q_U(x_{|n},v)^q)\E(|\widetilde U(x_{|n},v_0)|^q),
\end{equation}
where $v_0$ is any element of $\Sigma_n$.

The branching property yields  $\E(Q_U(x_{|n},v)^q)=\prod_{k=1}^n m^{-T_{U_{x_k}}(q)}$, and applying triangular inequality and a convexity inequality yields
$\E(|\widetilde U(x_{|n},v_0)|^q)\le 2^q\E(|Z|^q) m^{-T_{x_{n+1}}(q)}$, which  is bounded by a constant independent of $x$ since $T_{U_i}(q)$ is finite for all $0\le i\le m-1$ such that $p'_i>0$. Also, the strong law of large numbers yields
$$
\lim_{n\to\infty} n^{-1}\log \prod_{k=1}^n m^{-T_{U_{x_k}}(q)}=-\sum_{i=0}^{m-1} p'_i T_{U_i}(q)<0
$$
for $\eta$-almost every $x$. Consequently, the estimate \eqref{expconv} shows that for $\eta$-almost every $x$ we have $\sum_{n\ge 1} \big (\mathbb{E}(|X_{U}(x_{|n+1})-\E(Z)\widetilde X_{U}(x_{|n+1})|^q)\big )^{1/q}<\infty$, which implies that $X_{U}(x_{|n+1})$ converges $\mathbb{P}$-almost surely to the same limit as $\E(Z) \widetilde X_{U}(x_{|n+1})$, that is $\E(Z) \widetilde X_{U}(x)$.
\end{proof}

\section{}
\begin{lem}
\label{Borel}
Let $(Z_n)_{n\geq 1}$ be a sequence of non-negative random variables on a probability space $(\Omega, {\Bbb P})$. Then we have almost surely
$$
\limsup_{n\to \infty} \frac{\log Z_n}{n}\leq \limsup_{n\to \infty} \frac{\log {\Bbb E}(Z_n)}{n}.
$$

\begin{proof}
Let $b>a>\limsup_{n\to \infty} \frac{\log {\Bbb E}(Z_n)}{n}$. Then by  Markov's inequality,
$${\Bbb P}(\{(1/n)\log z_n\geq b\})\leq {\Bbb E}(Z_n)e^{-bn}\leq e^{n(a-b)}$$
when $n$ is large enough. Hence $\sum_{n=1}^\infty{\Bbb P}\{(1/n)\log z_n\geq b\}<\infty$. The Borel-Cantelli lemma implies that
$\limsup_{n\to \infty} \frac{\log Z_n}{n}\leq b$ almost surely. Letting $b$ tend to $\limsup_{n\to \infty} \frac{\log {\Bbb E}(Z_n)}{n}$ yields the
desired result.
\end{proof}

\end{lem}
\section{}
\label{D}
The following proposition explains how to transfer our result on the multifractal analysis of projections to $\Sigma$ of Mandelbrot measures on $\Sigma\times\Sigma$ to  projections on $[0,1]\times \{0\}$ of  Mandelbrot measures on $[0,1]^2$. Indeed, the natural projection form $\Sigma$ to $[0,1]$ does not map cylinders to centered balls, so we need some additional argument.
\begin{pro} Let $\rho$ be a positive and finite Borel measure on $\Sigma$. Let  $\alpha\in [0,\tau_\rho'(0+)]$ and suppose that  there exists a  positive and finite Borel measure $\widetilde\rho_\alpha$  on $\Sigma$ such that $\rho_\alpha(E(\rho,\alpha))>0$ and $\underline \dim_H(\rho_\alpha)\ge \tau_\rho^*(\alpha)>0$. Then $\dim_H E(\widetilde \rho,\alpha)=\tau_{\widetilde\rho}^*(\alpha)$, where  $\widetilde\rho$ and $\widetilde\rho_\alpha$ respectively stand for the natural projections of $\rho$ and $\rho_\alpha$  onto $[0,1]$.
\end{pro}
\begin{proof}
It is a simple fact that for all $q\ge 0$ we have $\tau_\rho(q)=\tau_{\widetilde\rho}(q)$, hence $ \tau_\rho^*$ and $\tau_{\widetilde \rho}^*$ coincide on $[0,\tau_\rho'(0+)]$.

It  is also clear that $\underline \dim_H(\widetilde \rho_\alpha)=\underline \dim_H(\rho_\alpha)$, hence $\underline \dim_H(\widetilde \rho_\alpha)\ge \tau_{\widetilde \rho}^*(\alpha)$. According to \eqref{MF}, for all $0\le \alpha'<\alpha$, we have $\dim_H \underline E^{\le }(\widetilde\rho,\alpha')\le \tau_{\widetilde\rho}^*(\alpha')$. Moreover, it follows from the fact that $\tau^*_\rho$ takes at lesat one positive value that  $\tau^*_\rho=\tau_{\widetilde\rho}^*$ is strictly increasing over $(-\infty,\tau_\rho'(0+)]\cap \mathrm{dom}(\tau_\rho^*)$ if this interval is not reduced to a singleton. This implies that $\widetilde\rho_\alpha ( \bigcup_{0\le \alpha'<\alpha}\underline E^{\le }(\widetilde\rho,\alpha'))=0$. Now, denote by $\Pi$ the natural projection from $\Sigma$ onto $[0,1]$. Let  $G_\alpha\subset \Pi(E(\rho,\alpha))$ of full $\widetilde \rho_\alpha$-positive measure. Without loss of generality we assume that  $G_\alpha$ contains no $m$-adic number and no element of $\bigcup_{0\le \alpha'<\alpha}\underline E^{\le }(\widetilde\rho,\alpha')$, i.e. $\underline{\dim}_{\rm loc}(\widetilde\rho,t)\ge \alpha$ for all $t\in G_\alpha$. Fix $t\in G_\alpha$. For $n\ge1$, denote by $I_n(t)$ the $m$-adic interval of generation $n$ which contains $t$. For all $\epsilon>0$, for $n$ large enough, we have $\widetilde \rho(I_n(t))\ge  \rho([(\Pi^{-1}(x))_{|n}])\ge m^{n(\alpha+\epsilon)}$,  hence $\widetilde \rho(B(t,m^{-n}))\ge m^{n(\alpha+\epsilon)}$. Consequently,
$\overline{\dim}_{\rm loc}(\widetilde\rho,t)\le \alpha+\epsilon$ for all $\epsilon>0$. Since we also have  $\underline{\dim}_{\rm loc}(\widetilde\rho,t)\ge \alpha$, we get $G_\alpha\subset E(\widetilde\rho,\alpha)$ hence the desired lower bound $\dim_H E(\widetilde\rho,\alpha)\ge \tau_{\widetilde\rho}^*(\alpha)$.
\end{proof}

Finally, we add a general about $L^q$ spectra. It is certainly not new but difficult to find explicitly written in the literature, except in the context of multiplicative chaos and statistical mechanics (see \cite{CK} for instance), where it signs a glassy phase transition.
\begin{pro}\label{linearization}
Let $\rho$ be a positive and finite Borel measure on $\Sigma$. For all $q \ge q'> 0$ we have $\tau_\rho(q)\ge \frac{q}{q'}\tau_\rho(q')$. As a consequence, if $\tau_\rho(q_c)=\tau_\rho'(q_c-)q_c$ at some $q_c>0$, then $\tau_\rho(q)=\tau_\rho'(q_c-)q$ for all $q\ge q_c$.
\end{pro}
\begin{proof}
The first claim follows from writing $\sum_{|u|=n}\rho([u])^q=\sum_{|u|=n}(\rho([u])^{q'})^{q/q'}$ and using the subbadditivity of $x\ge 0\mapsto x^{q/q'}$. The second claim follows from the concavity of $\tau_\rho$. which implies that for $q\ge q_c$ one has $\tau_\rho(q)\le \tau_\rho'(q_c-)q$ for $q\ge q_c$, while the first claim implies $\tau_\rho(q)\ge \frac{q}{q_c}\tau_\rho(q_c)=\tau_\rho'(q_c-)q$.
\end{proof}

\noindent{\bf Acknowledgements}. The research of both authors was supported in part by University of Paris 13, the
HKRGC GRF grants (projects  CUHK401013, CUHK14302415),  and the
France/Hong Kong joint research scheme PROCORE (33160RE,   F-CUHK402/14).

\end{document}